\documentclass[11pt]{scrartcl}
\usepackage{amsmath,amssymb, amsthm}
\usepackage{graphicx, color, psfrag}
\usepackage{bm, bbm}
\usepackage{algorithm}

\usepackage[noend]{algpseudocode}

\usepackage[applemac]{inputenc}
\usepackage[colorlinks, linkcolor=black, citecolor=black, urlcolor=black]{hyperref}

\usepackage{breakcites}
\usepackage{url}

\usepackage{caption}
\usepackage{blindtext}
\setcapindent{0em}

\usepackage{enumitem}
\usepackage{comment}
\usepackage{subcaption}
\usepackage{multirow}

\theoremstyle{plain}
	\newtheorem{theorem}{Theorem}
	\newtheorem{proposition}[theorem]{Proposition}
	\newtheorem{lemma}[theorem]{Lemma}
	\newtheorem{corollary}[theorem]{Corollary}

\theoremstyle{definition}

\theoremstyle{definition}
	\newtheorem{examplex}[theorem]{Example}
	\newenvironment{example}
  {\pushQED{\qed}\examplex}
  {\popQED\endexamplex}

\newcommand{\E}{\mathbb{E}}
\newcommand{\V}{\mathbb{V}}
\newcommand{\R}{\mathbb{R}}

\newcommand{\Pp}{\mathbb{P}} 

\newcommand{\BMM}{\hat{\theta}_{\textnormal{BMM}}}
\newcommand{\aBMM}{\hat{\theta}_{\textnormal{aBMM}}}
\newcommand{\IS}{\hat{\theta}_{\textnormal{IS}}}

\newcommand{\Tp}{\hat{\theta}_{\mathbf{p}}}
\newcommand{\SM}{\overline{\theta}}
\newcommand{\MM}{\hat{\theta}_{\textnormal{MM}}}
\newcommand{\HL}{\hat{\theta}_{\textnormal{HL}}}

\usepackage{accents}
\DeclareMathSymbol{\widehatsym}{\mathord}{largesymbols}{"62}
\newcommand\lowerwidehatsym{%
  \text{\smash{\raisebox{-1.3ex}{%
    $\widehatsym$}}}}
\newcommand\widhat[1]{%
  \mathchoice
    {\accentset{\displaystyle\lowerwidehatsym}{#1}}
    {\accentset{\textstyle\lowerwidehatsym}{#1}}
    {\accentset{\scriptstyle\lowerwidehatsym}{#1}}
    {\accentset{\scriptscriptstyle\lowerwidehatsym}{#1}}
}

\DeclareMathOperator{\MSE}{MSE}
\DeclareMathOperator{\Cov}{Cov}

\DeclareMathOperator{\iid}{\stackrel{\text{iid}}{\sim}}

\DeclareMathOperator{\med}{median}
\DeclareMathOperator{\Dir}{Dir}

\let\skew\relax%
\DeclareMathOperator{\skew}{skew}

\let\Im\relax%
\DeclareMathOperator{\Im}{Im}

\newcommand*{\email}[1]{\bgroup\color{blue}\href{mailto:#1}{#1}\egroup}
\newcommand*{\ppara}[1]{\noindent\textbf{\textsf{#1}}\,\,}

\begin{document}

\title{\textbf{Robust Mean Estimation with the Bayesian Median of Means}}
\author{Paulo Orenstein\footnote{Stanford University, Stanford, California 94305, USA. \email{pauloo@stanford.edu}}}
\date{}
\maketitle

\begin{abstract}
	\ppara{Abstract:} The sample mean is often used to aggregate different unbiased estimates of a real parameter, producing a final estimate that is unbiased but possibly high-variance. This paper introduces the Bayesian median of means, an aggregation rule that roughly interpolates between the sample mean and median, resulting in estimates with much smaller variance at the expense of bias. While the procedure is non-parametric, its squared bias is asymptotically negligible relative to the variance, similar to maximum likelihood estimators. The Bayesian median of means is consistent, and concentration bounds for the estimator's bias and $L_1$ error are derived, as well as a fast non-randomized approximating algorithm. The performances of both the exact and the approximate procedures match that of the sample mean in low-variance settings, and exhibit much better results in high-variance scenarios. The empirical performances are examined in real and simulated data, and in applications such as importance sampling, cross-validation and bagging.
	
\end{abstract}

\section{Introduction} \label{sec:introduction}
The problem of combining many unbiased estimates $\hat{\theta}_1, \ldots, \hat{\theta}_n$ of a parameter $\theta \in \R$ to form a single estimate $\hat{\theta}$ arises throughout statistics. A common way to perform aggregation when the distribution of $\hat{\theta}_i$ is unknown is via the sample mean, $\SM = \frac{1}{n}\sum_{i=1}^{n} \hat{\theta}_i$. However, this can often lead to poor performance if the underlying distribution is very skewed or heavy-tailed. 

For example, the importance sampling estimate of an integral $\theta = \int f(x) p(x) dx$ is formed by taking samples $X_i \iid q$, letting $\hat{\theta}_i=f(X_i)p(X_i)/q(X_i)$ and combining the estimates using the sample mean, $\frac{1}{n}\sum_{i=1}^{n}\hat{\theta}_i \approx \theta$. The estimator is unbiased and guaranteed to converge almost surely to $\theta$, as long as $\theta$ exists. Still, in many cases this ratio estimator has a complicated distribution with extremely high or infinite variance, making the estimates unreliable. Similar situations arise in various modern statistical procedures such as cross-validation, bagging and random forests. 

If the underlying distribution of $\hat{\theta}_i$ were known, aggregation could be performed using maximum likelihood, which enjoys great theoretical properties. In general, both its bias and variance decrease as $O(1/n)$, so some amount of bias is introduced at the expense of variance reduction. However, since mean squared error can be decomposed as the sum of bias squared and variance, the bias component is asymptotically negligible compared to the variance. In the problem at hand, the distribution of $\hat{\theta}_i$ is not known, but one can still look for a non-parametric estimator that similarly introduces a small, asymptotically negligible bias, to obtain a significant variance reduction relative to the sample mean.

Given many unbiased estimators $\hat{\theta}_1, \ldots, \hat{\theta}_n \in \R$, consider the following procedure:
\begin{enumerate}[leftmargin=4\parindent]
\item draw $\mathbf{p}^{(j)} \sim \Dir_n(\alpha, \ldots, \alpha)$ for $j=1, \ldots, J$;
\item compute $Y_j = \sum_{i=1}^{n} p_{i}^{(j)}\hat{\theta}_i$, for $j=1, \ldots, J$; 
\item estimate $\BMM = \widhat{\med}(Y_1, \ldots, Y_j)$.
\end{enumerate}
Here, $\Dir_n(\alpha, \ldots, \alpha)$ denotes the Dirichlet distribution with parameter vector $(\alpha, \ldots, \alpha) \in \R^n$.

This procedure, referred to as the \emph{Bayesian median of means}, is detailed in Section \ref{sec:bayesian_median_of_means}. Since the median and mean of a symmetric distribution coincide, one can think of this estimator as symmetrizing $\hat{\theta}_1, \ldots, \hat{\theta}_n$ via averaging before applying the median for added robustness. Section \ref{sec:theoretical_guarantees} studies the theoretical properties of this estimator, in particular proving that, when $J = O(n)$, the squared bias grows as $O(1/n^2)$ while the variance only grows as $O(1/n)$, so asymptotically the amount of bias introduced is insignificant; finite-sample guarantees are also provided, as well as consistency results. Section \ref{sec:empirical_results} investigates the empirical performance of this estimator in many different settings, showing that in general the trade-off is well worth it.

The procedure only has one hyperparameter, the Dirichlet concentration level $\alpha$. When $\alpha \to 0$ and $J$ is large the Bayesian median of means approximates the sample median, whereas when $\alpha \to \infty$ it approximates the sample mean. Thus, $\alpha$ can be thought of as controlling where the procedure stands between mean and median. In Proposition \ref{prop:aBMM}, it is shown that the Bayesian median of means is first-order equivalent to a skewness-corrected sample mean, with the correction controlled by $\alpha$, suggesting an approximate, non-randomized version of the procedure above, dubbed the \emph{approximate Bayesian median of means}:
\begin{equation*}
 \aBMM = \SM - \frac{1}{3} \frac{\sqrt{s^2_{\hat{\bm{\theta}}}}}{n\alpha+2}\widhat{\skew}(\hat{\bm{\theta}}),
\end{equation*}
where $s^2_{\hat{\bm{\theta}}}$ is the sample variance of $\hat{\bm{\theta}}=(\hat{\theta}_1, \ldots, \hat{\theta}_n)$ and $\widhat{\skew}(\hat{\bm{\theta}})$ is the sample skewness. Section \ref{sec:choosing_alpha} argues that $\alpha=1$ gives a good compromise between robustness and efficiency.

In many of the results below, $\hat{\theta}_1, \ldots, \hat{\theta}_n$ are regarded as fixed, and the statistical procedures are analyzed conditionally. This is intended to reproduce the common setting in which the unbiased estimators $\hat{\theta}_i$ have been obtained before the Bayesian median of means is used. It also implies some of the analyses hold even if $\hat{\theta}_1, \ldots, \hat{\theta}_n$ are not independent, which is the case for many important applications, such as cross-validation. Regarding $\hat{\theta}_i$ as fixed also bypasses the issue of computation, since in some cases the computational effort of drawing Dirichlet samples could be directed to obtaining further data points, as in Example \ref{example:IS_intro_example_MSE} below. When obtaining more points is relatively cheap, the approximate Bayesian median of means remains an attractive alternative that does not require further sampling.

\begin{example}[Importance sampling] \label{example:IS_intro_example_MSE}
As a first example, consider the importance sampling setting discussed above. Let $X_1, \ldots, X_n \iid q = \text{Expo}(1)$ be independent and identically distributed Exponential random variables, to be used in estimating the mean of $p = \text{Expo}(\lambda)$, that is, $1/\lambda$. Given a sample $X_1, \ldots, X_n$, one forms the importance sampling terms $\hat{\theta}_i = X_i\cdot p(X_i)/q(X_i) = \lambda X_i e^{-(\lambda-1)X_i}$, and estimates $\hat{\theta}_{\text{IS}} = \frac{1}{n} \sum_{i=1}^{n}\hat{\theta}_i$. This sample mean will be compared to the Bayesian median of means estimator, $\BMM$, and its approximation, $\aBMM$, formed using the procedures outlined above. 

\begin{figure}[htbp]
 \begin{center}
{\fontfamily{cmss}\selectfont \textbf{MSE and MAD of $\aBMM$ and $\BMM$ vs $\IS$}}

\vspace{5mm}

\hspace*{1em}\includegraphics[width=1\textwidth]{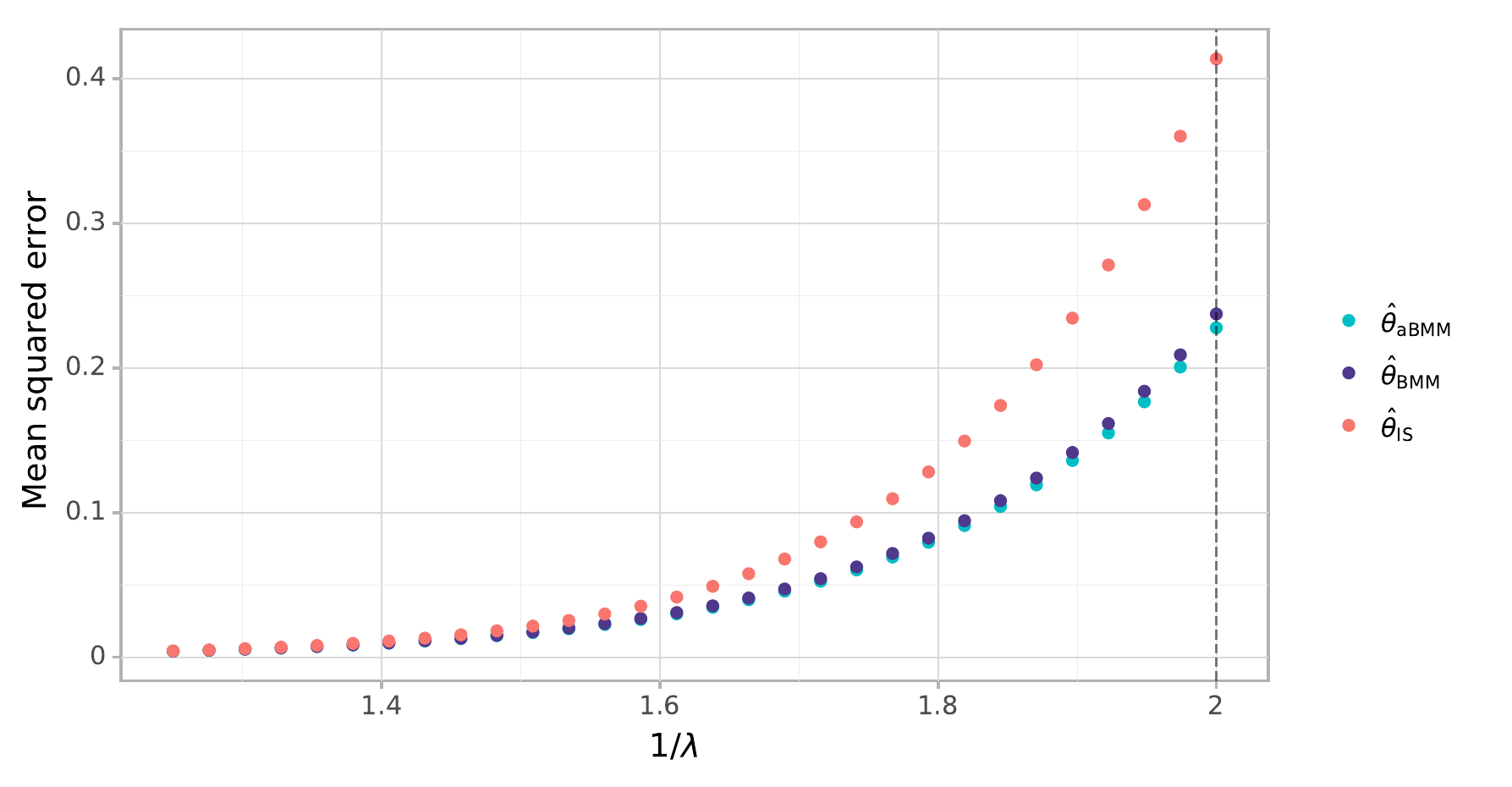}

\hspace*{1em}\includegraphics[width=1\textwidth]{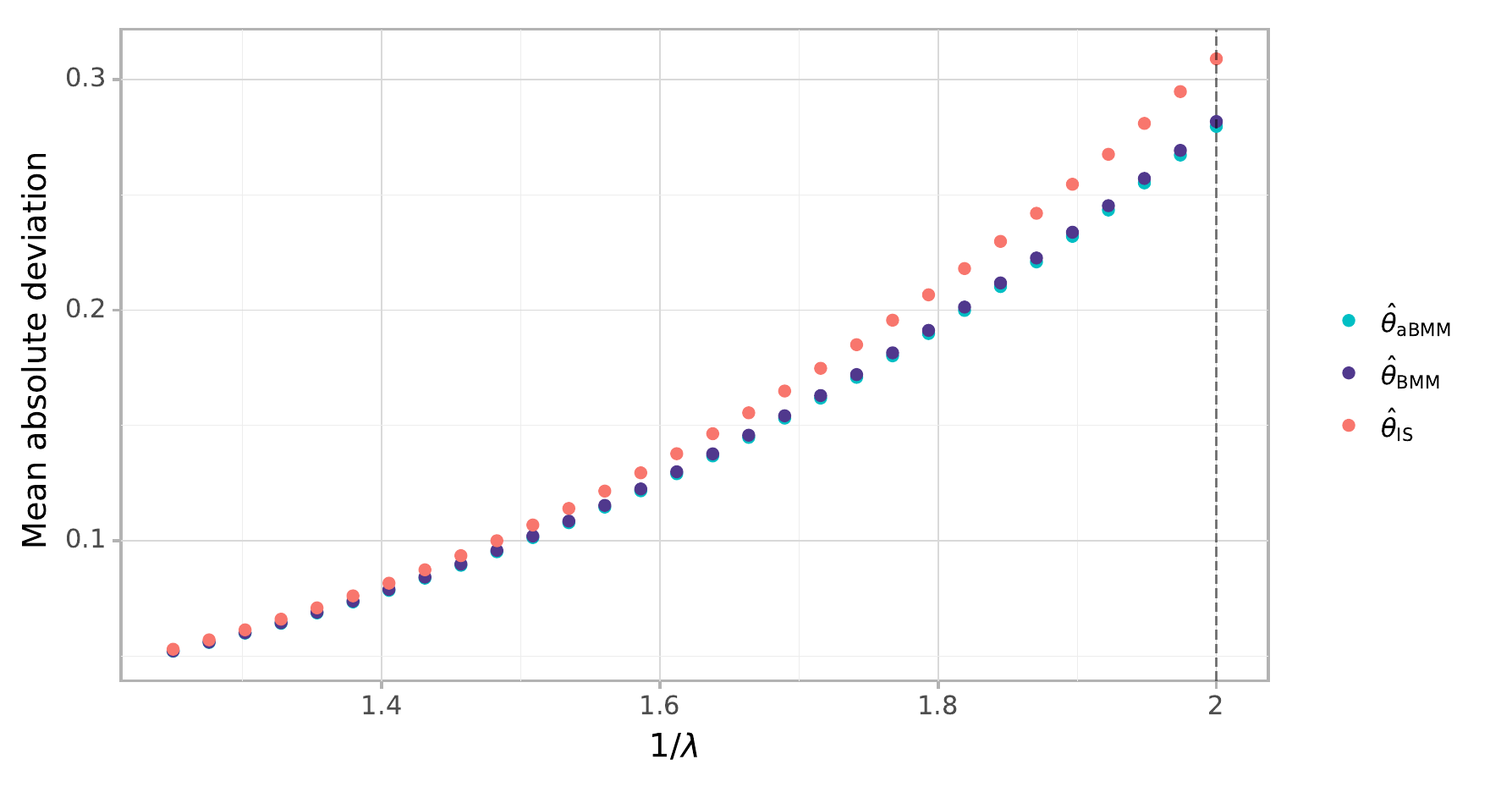}

\caption{Mean squared error (MSE) and mean absolute deviation (MAD) of importance sampling estimates using Bayesian median of means (BMM), its deterministic approximation (aBMM) and the usual importance sampling estimator (IS). Cases range from easy ($1/\lambda\approx 1.25$) to hard ($1/\lambda\approx 2$). When $1/\lambda \geq 2$ the IS estimator has infinite variance. Both BMM and aBMM outperform IS 30 out of 30 times in both MSE and MAD.}
\label{fig:IS_intro_example_MSE}
 \end{center}
\end{figure}

Figure \ref{fig:IS_intro_example_MSE} shows the mean squared error (MSE) and mean absolute deviation (MAD) of both procedures, with $30$ values of $1/\lambda$, equispaced between $1.25$ and $2$. This ranges from easy settings (when $1/\lambda \approx 1.25$) to complicated ones ($1/\lambda \approx 2$). In particular, when $1/\lambda \geq 2$ the variance of $\hat{\theta}_{\text{IS}}$ becomes infinite. Here, the number of importance sampling draws is $n=1000$, the number of Dirichlet draws is $J=1000$, and $\alpha=1$. For each $\lambda$ the estimator is computed $5000$ times to estimate the errors. Note the Bayesian median of means requires sampling $J=1000$ additional Dirichlet random variables, and this extra computational allowance could alternatively have been used to generate more points $\hat{\theta}_i$; the approximate algorithm, on the other hand, requires no extra sampling and exhibits similar performance.

In this simulation, the MSE of the Bayesian median of means and its approximation were smaller than the importance sampling estimator's 30 out of 30 times. The same was true for mean absolute deviation (MAD). For small values of $\lambda$ both estimators perform comparably. As $\lambda$ increases, it is easier to see the stability brought about by the Bayesian median of means makes it significantly better. In all 30 instances, the usual importance sampling estimator had larger standard deviation but smaller bias than both procedures, displaying the bias-variance trade-off involved. While considering values of $1/\lambda$ closer to 1 should favor the usual estimator, taking 30 equispaced values of $\lambda$ between $1$ and $2$ still sees better MSE and MAD performance by the Bayesian median of means 29 and 28 times; the approximate version is still better 30 and 29 times. Section \ref{sec:applications} considers more extensive simulations and real data examples, reaching similar conclusions.
\end{example}

\vspace{5mm}

\textbf{Paper structure.} Section \ref{sec:related_work} reviews the literature and similar attempts in using the median to add robustness to mean estimates. Section \ref{sec:median_of_weighted_means} introduces the general idea of symmetrizing estimates before using the median to estimate location parameters, and investigates to what degree that can help decrease mean squared error. Section \ref{sec:bayesian_median_of_means} settles on a particular distribution for symmetrizing estimates, giving rise to the Bayesian median of means. It also analyzes its theoretical properties and gives both asymptotic and finite-sample guarantees, and presents the (non-randomized) approximate Bayesian median of means. Finally, it also considers different ways of setting $\alpha$, the only hyperparameter in the procedure. Section \ref{sec:empirical_results} looks at the performance of these estimators in a myriad settings, including both real and simulated data, in particular comparing them to the sample mean. Finally, Section \ref{sec:conclusion} concludes by giving further research directions.

\section{Related Work} \label{sec:related_work}

The idea of combining mean and median estimators has been visited several times in the statistical robustness literature, particularly for the estimation of location parameters in symmetric distributions. For instance, \cite{lai1983adaptive} propose an adaptive estimator that picks either the sample mean or median to estimate the center of a symmetric distribution, while \cite{damilano2004efficiency} and \cite{chan1994simple} investigate using linear combinations of mean and medians, with weights picked according to asymptotic criteria.

More recently, there has been intense work on the so-called median of means estimator (see \cite{alon1999space}, \cite{jerrum1986random}, \cite{nemirovsky1983problem}). Given independent and identically distributed random variables $\hat{\theta}_1, \ldots, \hat{\theta}_n$, the median of means estimator for $\E[\hat{\theta}_1]$ is given by dividing the data into blocks with $k$ elements and estimating
\begin{equation} \label{eq:median_of_means}
 \MM = \widhat{\med}\left(\frac{1}{k}\sum_{i=1}^{k}\hat{\theta}_i, \ldots, \frac{1}{k}\sum_{i=n-k+1}^{n}\hat{\theta}_i\right),
\end{equation}
with minor adjustments if $n/k$ is not an integer. For instance, \cite{devroye2016sub} discuss the mean estimation problem from a non-asymptotic perspective and show that the median of means estimator, among others, can outperform the sample mean in terms of concentration bounds, via the following proposition, proved in the same paper. 

\begin{proposition} \label{prop:mm_concentration}
Consider an iid sample $\hat{\theta}_1, \ldots, \hat{\theta}_n$ with mean $\theta$ and variance $\sigma^2$. Define the median of means estimator $\MM$ as in (\ref{eq:median_of_means}), with $k$ elements in each of the $g=n/k$ groups. Then, with probability $1-\delta$,
\begin{equation*}
|\widhat{\med}(\overline{\theta}_1, \ldots, \overline{\theta}_g)-\theta| \leq 6 \sigma \sqrt{\frac{\log(1/\delta)}{n}}.
\end{equation*}
\end{proposition}

Such concentration cannot be achieved by the sample mean in general, unless stronger hypotheses, such as sub-Gaussianity of $\hat{\theta}_i$, are assumed. Variants of this estimator are further analyzed in the heavy-tailed setting of \cite{brownlees2015empirical}, \cite{hsu2016loss} and \cite{bubeck2013bandits}. The estimator was also used to combine Bayesian posterior updates in split datasets in \cite{pmlr-v32-minsker14}. 

Unfortunately, however, there are practical challenges in using the median of means estimator. First, there is little guidance in how to pick the number of groups, which essentially amounts to the estimator's willingness to trade off bias for variance. Furthermore, in spite of its theoretical properties, the median of means underutilizes the data available by only using each datapoint once. This guarantees independence between blocks, but limits the number of means one can obtain in a given dataset. This requirement can be relaxed to a degree, but not completely (see \cite{joly2016robust}). The estimator considered here has no such restrictions, and can be viewed as a smoothed version of the median of means. Besides, the randomness introduced in sampling the blocks allow for probabilistic analyses and parameter choices conditional on the realized datapoints. A further benefit is that, unlike the median of means, its smoothed counterpart does not depend on the order of the datapoints while still being computationally tractable.

In fact, the median of means itself can be cast as a computation compromise on the celebrated Hodges-Lehmann estimator proposed in \cite{hodges1963estimates}:
\begin{equation} \label{def:HL}
\HL = \widhat{\med}\left(\left\{\frac{\hat{\theta}_i+\hat{\theta}_j}{2},\quad i, j=1, \ldots, n, \text{ and } i\neq j\right\}\right).
\end{equation}
Many theoretical properties are known about it; for instance, it has an asymptotic breakdown point of $0.29$, meaning a contamination of up to $29\%$ of the data taking arbitrary value still leaves the estimator bounded (unlike the sample mean, which has an asymptotic breakdown point of 0). However, generalizing it to $m$-averages in a straightforward way, with $m\geq 2$, is harder both in theoretical and computational terms.

Closer in spirit to the present work is the suggestion in \cite{buhlmann2003bagging} to compute sample means using bootstrap samples and then aggregate them via the median. This is similar to bagging estimators $\hat{\theta}_1, \ldots, \hat{\theta}_n$, but using the median instead of the mean in assembling the averages. In particular, \cite{buhlmann2003bagging} empirically observed that the median helped make the final estimators better in terms of MSE in non-convex problems. This work generalizes \cite{buhlmann2003bagging} by resampling the weights used in the averaging using a Dirichlet distribution and by providing a more extensive theoretical analysis, in particular, deriving a non-randomized approximation. This leads to a smoother estimator, particularly relevant in the heavy-tailed or skewed setting, and also allows for theoretical results and hyperparameter recommendations based on the literature on Dirichlet means (see \cite{cifarelli1990distribution}, \cite{cifarelli2000some}, \cite{regazzini2002theory} and \cite{diaconis1996some}) and, more generally, on $P$-means (\cite{pitman2018random}).

\section{Median of Weighted Means} \label{sec:median_of_weighted_means}

Given a probability space $(\Omega, \mathcal{F}, \mathbb{P}_{\hat{\theta}})$, let $\hat{\theta}_i : \Omega \to \R$, $i=1, \ldots ,n$ be a collection of independent and identically distributed random variables, with $\E[\hat{\theta}_1]=\theta$ and $\V[\hat{\theta}_1]=\sigma^2 < \infty$. This will be denoted below by $\hat{\theta}_i \iid [\theta, \sigma^2]$. Consider the problem of obtaining an estimator for $\theta$ given observations $\hat{\bm{\theta}}=(\hat{\theta}_1, \ldots, \hat{\theta}_n)$.

A common aggregation procedure is the sample mean, $\SM = \frac{1}{n}\sum_{i=1}^{n} \hat{\theta}_i$. Besides retaining unbiasedness, it also possesses many satisfying theoretical properties; for example, the sample mean is the best linear unbiased estimator for the population mean, minimizes the maximum expected mean squared error populations with bounded variance (see \cite{bickel1981minimax}), is an admissible procedure in one dimension (\cite{10.2307/2239052}), and is the maximum likelihood estimate in exponential families under independent sampling. For these reasons, as well as computationally simplicity, the sample mean is a widely adopted non-parametric procedure for aggregating one-dimensional estimators.

However, in many settings the underlying distribution of $\hat{\theta_1}$ is very skewed or heavy tailed, in which case the sample mean becomes highly unstable. %
A robust solution to the aggregation problem is to use the sample median. If the distribution is symmetric, then mean and median coincide, the sample median is still an unbiased estimator, and the median can have better asymptotic mean squared error than the bias, as in Example \ref{example:ARE}.

\begin{example}[Mean and median] \label{example:ARE}

Let $\hat{\theta}_1, \ldots, \hat{\theta}_n \iid N(\theta, \sigma^2)$, and consider estimators $\SM = \frac{1}{n}\sum_{i=1}^n \hat{\theta}_i$ and $\widhat{\med}(\hat{\theta}_1, \ldots, \hat{\theta}_n)$, which are unbiased for estimating $\theta$. In this case, the asymptotic mean squared error is given by the asymptotic variance. Since $n \cdot \V[\SM] \to \sigma^2$ as $n \to \infty$, and also $n \cdot \V[\widhat{\med}(\hat{\theta}_1, \ldots, \hat{\theta}_n)]\to \pi \sigma^2/2$ (see Proposition \ref{prop:medianCLT} below), it holds that, asymptotically, $\MSE(\SM) = (2/\pi) \cdot \MSE(\widhat{\med}(\hat{\theta}_1, \ldots, \hat{\theta}_n))$, so the sample mean exhibits better asymptotic performance.

On the other hand, consider $\hat{\theta}_1, \ldots, \hat{\theta}_n \iid \text{Laplace}(\theta, \sigma^2)$. As before, both $\SM = \frac{1}{n}\sum_{i=1}^n \hat{\theta}_i$ and $\widhat{\med}(\hat{\theta}_1, \ldots, \hat{\theta}_n)$ are unbiased for estimating $\theta$. In particular, note $\widhat{\med}(\hat{\theta}_1, \ldots, \hat{\theta}_n)$ is the maximum likelihood estimator. When $n \to \infty$, $n \cdot \V[\SM] \to \sigma^2$ while $n \cdot \V[\widhat{\med}(\hat{\theta}_1, \ldots, \hat{\theta}_n)]\to \sigma^2/2$, so $\MSE(\SM) = 2 \cdot \MSE(\widhat{\med}(\hat{\theta}_1, \ldots, \hat{\theta}_n))$, and the median is asymptotically better than the sample mean. In more extreme cases, say if $\hat{\theta}_1, \ldots, \hat{\theta}_n \iid \text{Cauchy}(0, 1)$, the median can be asymptotically infinitely better than the sample mean.

One can try to use a compromise between the two estimators, for instance the Hodges-Lehmann estimator, $\HL$, defined in (\ref{def:HL}), which is still unbiased for symmetric distributions. Asymptotically, with a Normal sample, $\MSE(\SM) = (3/\pi) \cdot \MSE(\HL)$, and with a Laplace sample, $\MSE(\SM) = 1.5 \cdot \MSE(\HL)$. In fact, if $\mathcal{S}$ is the set of symmetric distributions centered at $\theta$, it can be shown that asymptotically $\inf_{\{F_{\hat{\bm{\theta}}}: F_{\hat{\bm{\theta}}} \in \mathcal{S}\}} \MSE(\SM)/\MSE(\HL)=0.864$ (see \cite{hodges1956efficiency}), so $\HL$ never fares much worse than $\SM$ but can sometimes do much better.
\end{example}

In general, however, the median will be a biased estimator. For instance, if the distribution of $\hat{\theta}_i$ is highly skewed, $\widhat{\med}(\hat{\theta}_i)$ might be very different from $\theta = \E[\hat{\theta}_i]$. One way to soften the bias is to take the median of weighted averages of $\hat{\theta}_i$, which are more symmetric around $\theta$ by the Central Limit Theorem. Hence, the application of the median becomes less costly in terms of bias, while still guaranteeing increased robustness.

A general scheme for aggregating many unbiased estimators $\hat{\theta}_1, \ldots, \hat{\theta}_n$ using a vector of probabilities $\mathbf{p}=(p_1, \ldots, p_n)$, with $\sum_{i=1}^{n}p_i=1$, $p_i\geq0$, is as follows:
\begin{enumerate}
\item Sample:
\begin{equation} \label{eq:step1:sample}
\mathbf{p}^{(j)}\sim \Pp, \qquad j=1, \ldots, J,
\end{equation}
where $\Pp$ is a probability measure such that $\sum_{i=1}^{n} p_i^{(j)}=1$.
\item Estimate: 
\begin{equation} \label{eq:step2:median}
\Tp = \widhat{\med}\left(\sum_{i=1}^{n}p_i^{(1)}\hat{\theta}_i, \ldots, \sum_{i=1}^{n}p_i^{(J)}\hat{\theta}_i\right).
\end{equation}
\end{enumerate}

Consider the following choices for $\Pp$:
\begin{itemize}
\item if $\Pp$ sets $p_i=1$ for $i$ chosen uniformly at random and $J$ is sufficiently large, then $\Tp$ is essentially the sample median;
\item if $\Pp$ is a delta mass at $(\frac{1}{n}, \ldots, \frac{1}{n})$, then $\Tp$ is the sample mean;
\item if $\Pp$ selects sets $S$ of size $k$ in a uniformly chosen partition of $\{1, \ldots, n\}$ and sets $p_i = \frac{1}{k}$ for $i \in S$ and $0$ otherwise, then $\Tp$ is a randomly-partitioned median of means estimator;
\item if $\Pp$ is $\Dir_n(\alpha, \ldots, \alpha)$, then the resulting estimator is the \emph{Bayesian median of means}, $\BMM$. 
\end{itemize}

Hence, estimators with randomized weights give a way to interpolate between the sample mean, with low bias and possibly high variance, and the sample median, with low variance but possibly high bias. See Figure \ref{fig:median_to_mean}. The two extreme choices of $\Pp$ leading to the sample mean and sample median come from degenerate Dirichlet distributions placing all mass at the centroid of the simplex ($\alpha\to0$), or splitting the mass equally among the vertices of the simplex ($\alpha\to\infty$). More general distributions over the simplex are possible, and many of the results in this paper can readily be extended to that case (see \cite{newton1994approximate} and \cite{pitman2018random}). In particular, distributions can be picked to encode any prior information available about the sample $\hat{\theta}_1, \ldots, \hat{\theta}_n$ such as skewness or symmetry.

\begin{figure}
	\begin{center}
		\vspace{1em}
		{\fontfamily{cmss}\selectfont \textbf{Distribution of $\BMM$}}
		\includegraphics[width=\textwidth]{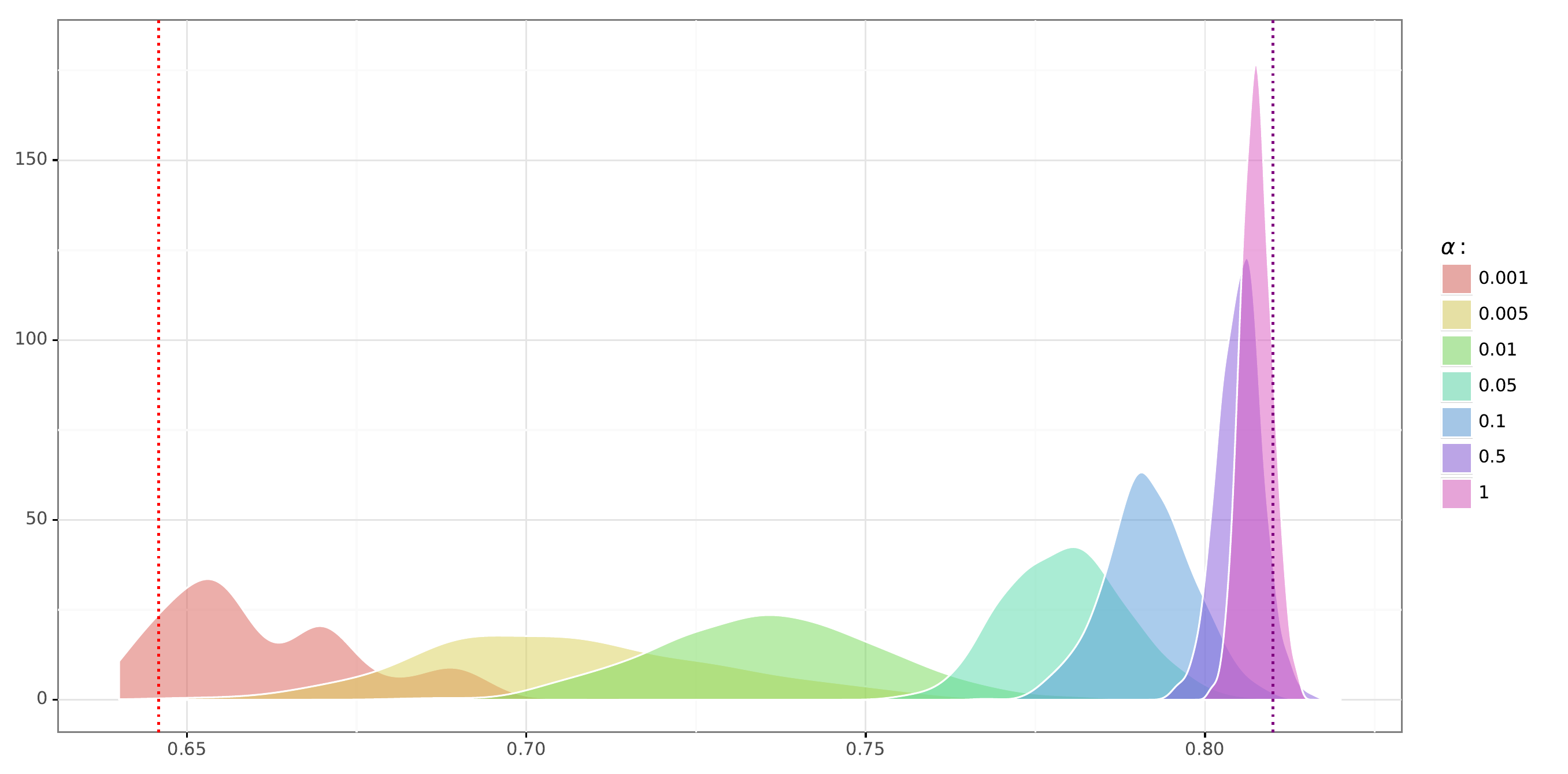}
		\caption{Distribution of $\BMM$ for different levels of $\alpha$ over many draws $\mathbf{p}^{(j)}$, $j=1, \ldots, 1000$, and fixed $\hat{\theta}_1, \ldots, \hat{\theta}_n$; $\med(\hat{\bm{\theta}})$ and $\overline{\theta}$ are shown in dotted red and purple lines, respectively.}
		\label{fig:median_to_mean}
	 \end{center}
\end{figure}

The reduction in variance achieved by the median of weighted means can have a drastic effect on mean squared error when $\hat{\theta}_1, \ldots, \hat{\theta}_n$ come from a distribution with high variance. This can be understood through the following proposition.

\begin{proposition}
Let $\hat{\theta}_1, \ldots, \hat{\theta}_n$ be iid, unbiased estimates of a parameter $\theta \in \R$. Let $\SM=\frac{1}{n}\sum_{i=1}^{n}\hat{\theta}_i$ be the sample mean, and $\Tp$ be any median of weighted means estimator (\ref{eq:step2:median}). Then the mean squared error of $\Tp$ can be bounded by
\begin{align*}
\E[(\Tp-\theta)^2] &\leq \E[(\Tp-\SM)^2] - \V[\SM]\left(1-2\sqrt{\frac{\V[\Tp]}{\V[\SM]}} \right),
\end{align*}
where the expectation is over both the data $\{\hat{\theta}_i\}_{i=1}^n$ and the weights $\{\mathbf{p}^{(j)}\}_{j=1}^J$.
\end{proposition}

\begin{proof}
Decompose the expectation as 
\begin{equation*}
\E[(\Tp-\theta)^2] = \E[(\Tp - \SM +\SM - \theta)^2] = \E\left[(\Tp-\SM)^2\right] + \E\left[(\SM-\theta)^2\right] + 2 \E\left[(\Tp-\SM)(\SM-\theta)\right],
\end{equation*}
and note the cross-term can be written
\begin{align*}
\E\left[(\Tp - \SM)(\SM-\theta)\right] &= \E\left[(\Tp-\theta)(\SM-\theta)\right] + \E\left[(\theta-\SM)(\SM-\theta)\right]\\
&= \E\left[(\Tp-\E[\Tp])(\SM-\theta)\right] - \E\left[(\SM-\theta)^2\right] \\
&= \Cov(\Tp, \SM) - \V[\SM]\\
&= \V[\SM]\left(\rho_{\Tp, \SM} \sqrt{\frac{\V[\Tp]}{\V[\SM]}} - 1\right),
\end{align*}
where $\rho_{\Tp, \SM}$ is the correlation between $\SM$ and $\Tp$. Putting it together,
\begin{equation}
\E\left[(\Tp-\theta)^2\right] = \E[(\SM-\theta)^2] + \E\left[(\Tp-\SM)^2\right] - 2 \V[\SM]\left(1-\rho_{\Tp, \SM} \sqrt{\frac{\V[\Tp]}{\V[\SM]}}\right), \label{eq:MSE_Tp_MSE_SM}
\end{equation}
and since $\SM$ is unbiased, $\E[(\SM-\theta)^2]=\V[\SM]$, so
\begin{align*}
\E\left[(\Tp-\theta)^2\right] &=  \E[(\Tp-\SM)^2] - \V[\SM]\left(1-2\rho_{\Tp, \SM} \cdot \sqrt{\frac{\V[\Tp]}{\V[\SM]}} \right)\\
&\leq  \E[(\Tp-\SM)^2] - \V[\SM]\left(1-2 \sqrt{\frac{\V[\Tp]}{\V[\SM]}} \right) \qedhere.
\end{align*}
\end{proof}

Thus, the mean squared error of $\Tp$ is given by a term measuring the discrepancy between $\Tp$ and $\SM$, minus a term measuring the variance reduction achieved by $\Tp$ with respect to $\SM$. If $\hat{\theta}_1, \ldots, \hat{\theta}_n$ are coming from a distribution with high variance, then generally $\V[\Tp] \ll \V[\SM]$, and the second term in the right-hand side of the bound is negative and very large. If, on the other hand, $\hat{\theta}_1, \ldots, \hat{\theta}_n$ are coming from a distribution with low variance, then $\Tp \approx \SM$, and there should be no significant differences between the mean squared errors of the two estimators.

\section{Bayesian Median of Means} \label{sec:bayesian_median_of_means}

Now consider the median of weighted means estimator obtained by sampling the probabilities in (\ref{eq:step1:sample}) from a $\Dir_n(\alpha, \ldots, \alpha)$ distribution. Recall this scheme is broad enough to interpolate between median and mean, while still being analytically tractable, and is given by
\begin{enumerate}[leftmargin=4\parindent]
\item draw $\mathbf{p}^{(j)} \sim \Dir_n(\alpha, \ldots, \alpha)$ for $j=1, \ldots, J$;
\item compute $Y_j = \sum_{i=1}^{n} p_{i}^{(j)}\hat{\theta}_i$, for $j=1, \ldots, J$; 
\item estimate $\BMM = \widhat{\med}(Y_1, \ldots, Y_j)$.
\end{enumerate}

This estimator is `Bayesian' in the sense that the probabilities $\mathbf{p}^{(j)}$ are being generated according to the Bayesian bootstrap. Indeed, consider weights $\mathbf{w}=(w_1, \ldots, w_n)$ with $\sum_{i=1}^{n} w_i=n$, $w_i\geq 0$ and $\mathbf{p}=\frac{1}{n}\mathbf{w}$, and assume the following underlying model
\begin{align*}
\mathbf{p} &\sim \Dir_n(\gamma, \ldots, \gamma)\\
\mathbf{w} \mid \mathbf{p} &\sim \text{Mult}_n(m, \mathbf{p}),
\end{align*}
so the posterior distribution is
\begin{align*}
\mathbf{p}\mid \mathbf{w} &\sim \Dir_n\left(\gamma + \frac{m}{n}\mathbf{w}\right). 
\end{align*}
With a non-informative prior, $\gamma \to 0$, the posterior distribution becomes $\Dir_n(m/n, \ldots, m/n)$, which amounts to step 1 above with $\alpha=m/n$. In particular, if $m=n$, it is $\Dir_n(1, \ldots, 1)$. 

This gives a posterior on the sums $\sum_{i=1}^{n}p_i^{(j)}\hat{\theta}_i$ for a randomly sampled probability vector $\mathbf{p}^{(j)}$ and fixed $\hat{\theta}_1, \ldots, \hat{\theta}_n$. Summarizing the posterior distribution by minimizing the $l_1$ loss for robustness yields the Bayesian median of means.

Note the usual bootstrap would sample $\mathbf{p}^{(j)}\sim\text{Mult}_n(m, (1/n, \ldots, 1/n))$. This has the same mean as $\Dir_n(m/n, \ldots, m/n)$, and nearly the same variance. The main reason for using the Dirichlet distribution is that it confers additional smoothness to the estimator that are important in establishing theoretical results, in particular asymptotic expansions (see Section \ref{sec:asymptotic_approximation}).

\vspace{5mm}

How can this estimator improve on the sample mean? Proposition \ref{prop:aBMM} below shows that, under some regularity assumption and assuming $J=O(n)$, one can approximate, 
\begin{equation*}
 \BMM = \SM - \frac{1}{3} \frac{\sqrt{s^2_{\hat{\bm{\theta}}}}}{n\alpha+2}\widhat{\skew}(\hat{\bm{\theta}}) + o\left(\frac{1}{n\alpha}\right),
\end{equation*}
where $s^2_{\hat{\bm{\theta}}}=\frac{1}{n}\sum_{i=1}^{n}(\hat{\theta}_i - \SM)^2$ and $\widhat{\skew}(\hat{\bm{\theta}}) = \frac{1}{(s^2_{\hat{\bm{\theta}}})^{3/2}} \cdot \frac{1}{n}\sum_{i=1}^{n}(\hat{\theta_i}- \SM)^3$, so to first-order the Bayesian median of means is a corrected sample mean, with the correction proportional to the sample standard deviation and skewness. This is reminiscent of a shrinkage-type estimator.

Indeed, if $\hat{\theta}_1, \ldots, \hat{\theta}_n$ are coming from a symmetric distribution, then $\BMM \approx \SM$, so consider the case where the underlying distribution is very right-skewed, as in the first plot in Figure \ref{fig:intuition_BMM}. Assume two samples of $n$ points are obtained, as represented in red and blue in the figure. Note the blue sample happens to include a very large, but unlikely, sample point. 

\begin{figure}[htbp]
 \begin{center}
	 {\fontfamily{cmss}\selectfont \textbf{Typical behavior of $\BMM$}} \includegraphics[width=\textwidth]{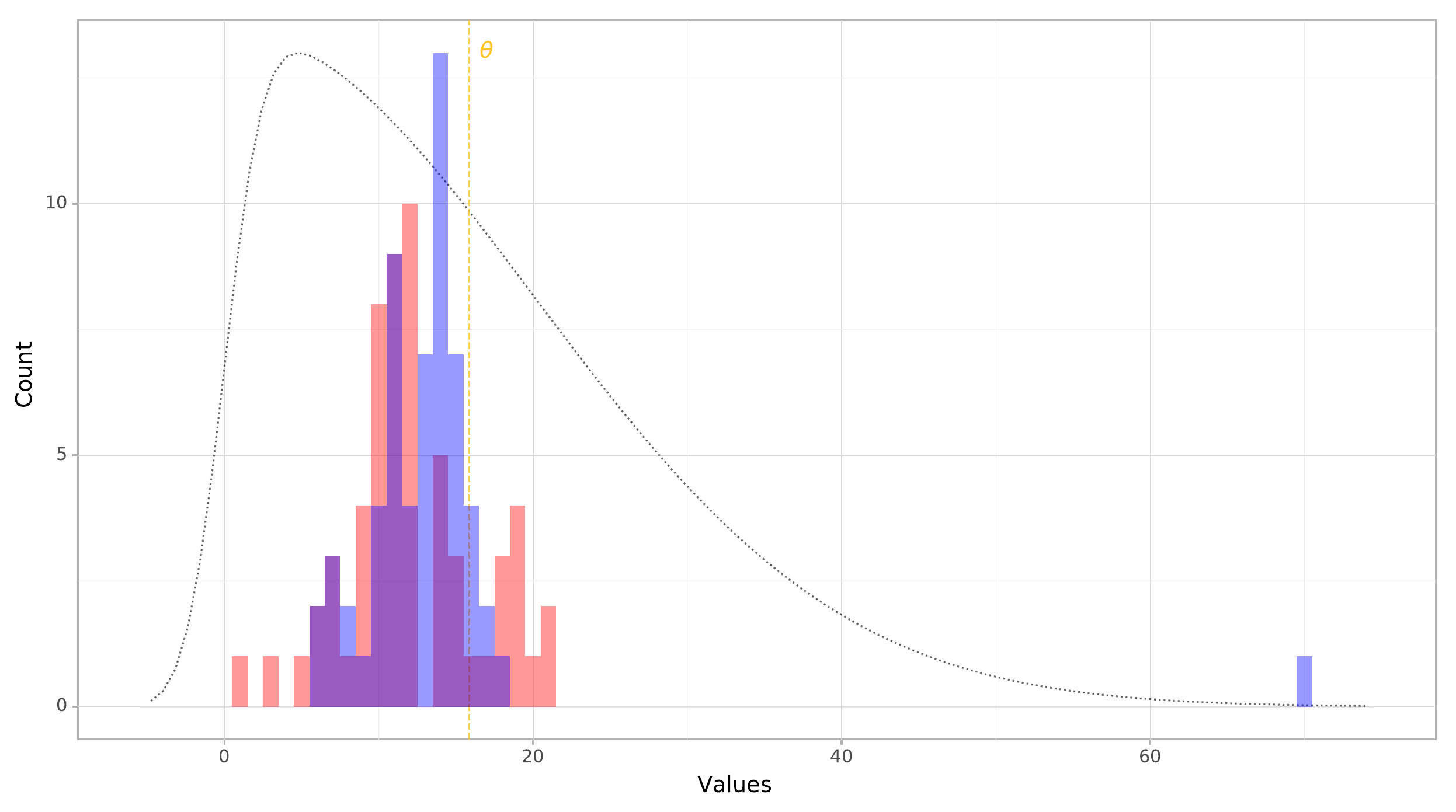}

\includegraphics[width=\textwidth]{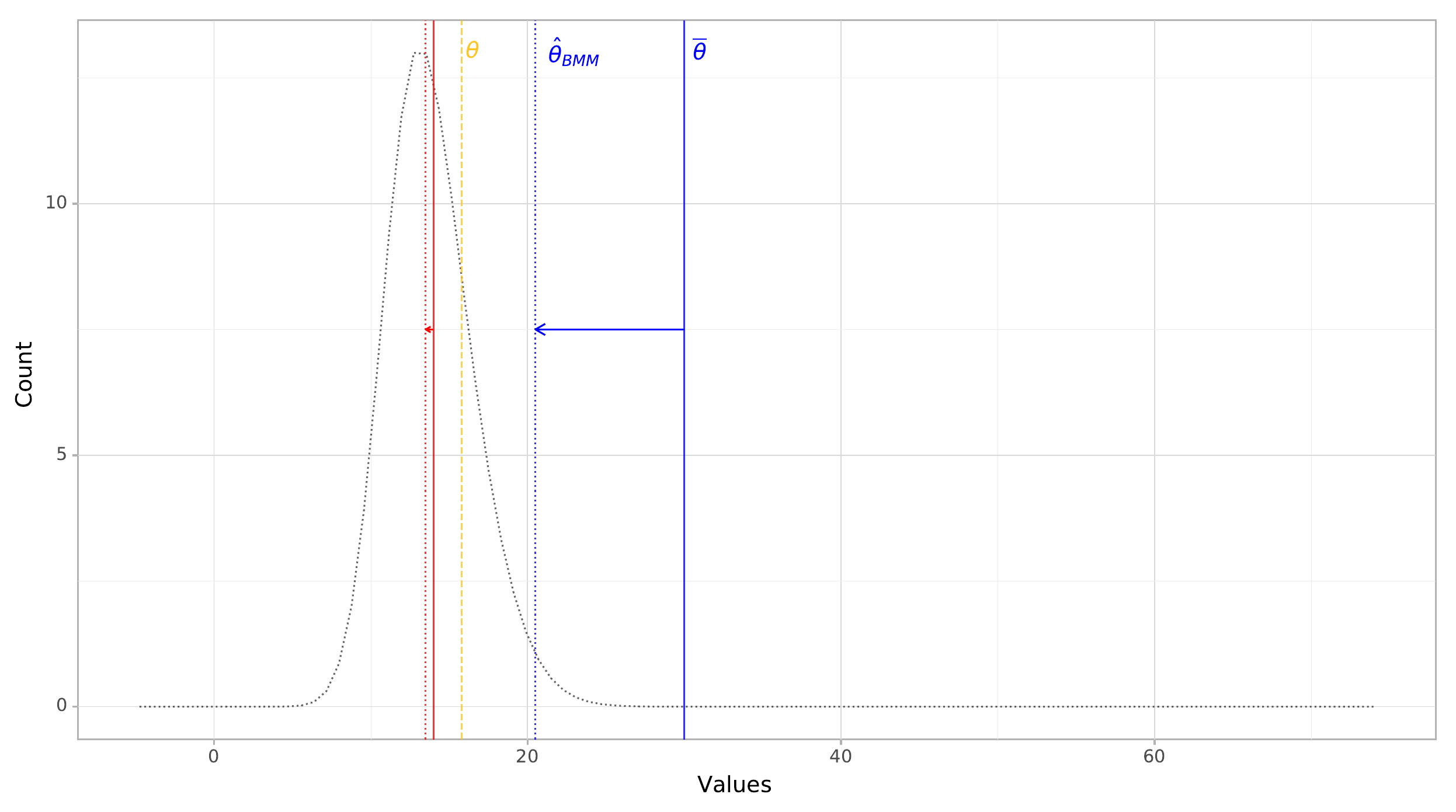}
 \caption{In the upper plot there are two samples, in red and blue, drawn from the underlying distribution of the $\hat{\theta}_i$, represented by the black dotted line; the yellow line shows the mean of the distribution. The plot below contains the distribution of resampled averages $Y_j$, represented by the black dotted line, as well as the sample mean of the blue and red samples indicated by a solid line, and the result of using $\BMM$ indicated by the dotted line.}
 \label{fig:intuition_BMM}
 \end{center}
\end{figure}

The second plot in Figure \ref{fig:intuition_BMM} shows the result of applying both the sample mean and the Bayesian median of means to the red and blue samples. Since the red sample is virtually symmetric, $\widhat{\skew}(\hat{\bm{\theta}})\approx 0$ and so $\SM \approx \BMM$. For the blue sample, the large sample point means the sample mean overestimates $\theta$ by a lot; in this case, both $s^2_{\hat{\bm{\theta}}}$ and $\widhat{\skew}(\hat{\bm{\theta}})$ are large and positive, so $\BMM < \SM$, as represented in Figure \ref{fig:intuition_BMM}. 

This schematic example explains the mechanics behind the Bayesian median of means: it introduces some bias for samples with large sample variance and skewness, and in doing so greatly reduces variance. When the underlying distribution of $\hat{\theta}_i$ has small variance or is symmetric, $\SM$ and $\BMM$ are virtually indistinguishable; however, when the underlying distribution has heavy tails or is very skewed, the estimators disagree (this phenomena was also observed in Example \ref{example:IS_intro_example_MSE}). Hence, the Bayesian median of means matches the sample mean performance for samples coming from approximately Normal distributions, but incurs in corrections once this is no longer the case.

\vspace{5mm}

To establish the theoretical properties of this estimator, Section \ref{sec:cond_moments_and_density_Y} first looks at the distribution of the resampled averages $Y_j$, $j=1, \ldots, J$. Section \ref{sec:theoretical_guarantees} establishes finite-sample guarantees for the Bayesian median of means, as well as asymptotic approximations. Finally, Section \ref{sec:choosing_alpha} considers the issue of picking a value for the hyperparameter $\alpha$.

\subsection{Conditional Moments and Density for $Y$} \label{sec:cond_moments_and_density_Y}

To understand the behavior of the Bayesian median of means, it is first necessary to study $Y_j$. In particular, it will be important to characterize the moments and distribution of $Y$ (the subscript $j$ will be dropped when the meaning is clear). Calculating the first two conditional moments of $Y$ is straightforward:
\begin{equation} \label{eq:mean}
\E[Y|\hat{\bm{\theta}}] = \E\left[\sum_{i=1}^{n}p_i \hat{\theta}_i \mid \hat{\bm{\theta}}\right] = \sum_{i=1}^{n}\E\left[p_i\right] \hat{\theta}_i  = \sum_{i=1}^{n} \left(\frac{\alpha}{n \alpha}\right) \hat{\theta}_i = \frac{1}{n} \sum_{i=1}^{n} \hat{\theta}_i = \SM,
\end{equation}
and
\begin{align}
\V&[Y|\hat{\bm{\theta}}] = \hat{\bm{\theta}}^T \V[\mathbf{p}] \hat{\bm{\theta}}= \sum_{i=1}^{n} \frac{n}{n^2(n \alpha+1)} \hat{\theta}_i^2 - \sum_{i=1}^{n} \sum_{j=1}^{n} \frac{1}{n^2 (n \alpha + 1)} \hat{\theta}_i \hat{\theta}_j \\
&= \frac{1}{(n \alpha + 1)} \left(\frac{1}{n}\sum_{i=1}^{n}\hat{\theta}_i - \left(\frac{1}{n}\sum_{i=1}^{n}\hat{\theta}_i\right)\left(\frac{1}{n}\sum_{i=1}^{n}\hat{\theta}_i\right)\right) = \frac{1}{n \alpha + 1}\left(\overline{\theta^2} - \left(\SM\right)^2\right) \\
&= \frac{1}{n \alpha + 1} s_{\hat{\bm{\theta}}}^2, \label{eq:variance}
\end{align}
where $s_{\hat{\bm{\theta}}}^2 = \frac{1}{n} \sum_{i=1}^{n}(\hat{\theta}_i - \SM)^2$, so
\begin{equation*}
\E[Y^2|\hat{\bm{\theta}}] = \V[Y|\hat{\bm{\theta}}] - \E^2[Y|\hat{\bm{\theta}}] = \frac{\overline{\theta^2} + n \alpha (\SM)^2}{n\alpha+1}.
\end{equation*}

Higher conditional moments can be found in a recursive fashion.

\begin{lemma} 
Let $\mathbf{p} \sim \Dir_n(\alpha, \ldots, \alpha)$ and $\hat{\theta}_1, \ldots, \hat{\theta}_n$ be a set of fixed of estimates. The linear combination of Dirichlet components $Y = \sum_{i=1}^{n} p_i \hat{\theta}_i$ has moments given by the recursion
\begin{equation} \label{eq:recursive_moments}
\E[Y^m | \hat{\bm{\theta}}] = \sum_{k=0}^{m-1} \left( \frac{(n\alpha+k-1)!}{(n\alpha+m-1)!}\frac{(m-1)!}{k!} \E[Y^k | \hat{\bm{\theta}}]\cdot \sum_{i=1}^{n}\alpha \hat{\theta}_i^{m-k}\right).
\end{equation}
\end{lemma}

\begin{proof}
First, recall that if $\mathbf{p}\sim\Dir_n(\alpha, \ldots, \alpha)$, then
\begin{align}
 \E[p_1^{\beta_1} p_2^{\beta_2} \cdots p_n^{\beta_n}] &= \frac{\Gamma(n \alpha)}{(\Gamma(\alpha))^n} \int p_1^{\beta_1+\alpha-1} \cdots p_n^{\beta_n+\alpha-1} dp_1 \cdots dp_n \\
 &= \frac{\Gamma(n \alpha)}{\Gamma(n \alpha + \sum_{i=1}^{n} \beta_i)} \prod_{i=1}^{n} \frac{\Gamma(\beta_i + \alpha)}{\Gamma(\alpha)}. \label{eq:expected_dirichlet_product}
\end{align}

Using $\E[\cdot]$ below to denote the expectation conditional on $\hat{\bm{\theta}}$, write
\begin{align*}
\E\left[\left(\sum_{i=1}^{n}p_i \hat{\theta}_i\right)^m\right] &= \E\left[\left(\sum_{i=1}^{n}p_i \hat{\theta}_i\right)^{m-1}\sum_j p_j \hat{\theta}_j\right] = \sum_j \hat{\theta}_j \cdot \E\left[p_j \left(\sum_{i=1}^{n}p_i\hat{\theta}_i\right)^{m-1}\right]\\
&= \sum_j \hat{\theta}_j \E\left[\sum_{k_1+\cdots+k_n=m-1} {m-1 \choose k_1, \ldots, k_n} \left(\prod_{i=1}^{n}p_i^{k_i+\mathbb{I}_{[i=j]}}\right)\right]\\
&= \sum_{j} \hat{\theta}_j \left(\sum_{k_1+\cdots+k_n=m-1} {m-1 \choose k_1, \ldots, k_n} \left(\prod_{i=1}^{n} \hat{\theta}_i^{k_i}\right) \E\left[\prod_{i=1}^{n} p_i^{k_i + \mathbb{I}_{[i=j]} }\right]\right).
\end{align*}
Using (\ref{eq:expected_dirichlet_product}), this becomes
\begin{align*}
\E\left[\left(\sum_{i=1}^{n}p_i \hat{\theta}_i\right)^m\right] &= \sum_{j} \hat{\theta}_j \left(\sum_{k_1+\cdots+k_n=m-1} {m-1 \choose k_1, \ldots, k_n} \left(\prod_{i=1}^{n} \hat{\theta}_i^{k_i}\right) \frac{\Gamma(n\alpha)}{\Gamma(n\alpha+m)} \prod_{i=1}^{n} \frac{\Gamma(\alpha+k_i+\mathbb{I}_{[i=j]})}{\Gamma(\alpha)}\right)\\
&= \sum_{j} \hat{\theta}_j \left(\sum_{k_1+\cdots+k_n=m-1} {m-1 \choose k_1, \ldots, k_n} \left(\prod_{i=1}^{n} \hat{\theta}_i^{k_i}\right) \frac{\alpha+k_j}{n\alpha+m-1} \E\left[\prod_{i=1}^{n} p_i^{k_i}\right]\right)\\
&= \sum_j \hat{\theta}_j \alpha \frac{1}{n\alpha+m-1}\E\left[\sum_{k_1+\cdots+k_n=m-1}{m-1 \choose k_1, \ldots, k_n} \prod_{i=1}^{n}\hat{\theta}_i^{k_i}p_i^k\right] \\
& \qquad + \sum_j \hat{\theta}_j \frac{1}{n\alpha+m-1} \E\left[\sum_{k_1+\cdots+k_n=m-1} k_j {m-1 \choose k_1, \ldots, k_n} \prod_{i=1}^{n}\hat{\theta}^{k_i}p_i\right].
\end{align*}
The first term in the last equality above is just
\begin{align*}
\sum_j \hat{\theta}_j \alpha \frac{1}{n\alpha+m-1}\E\left[\sum_{k_1+\cdots+k_n=m-1}{m-1 \choose k_1, \ldots, k_n} \prod_{i=1}^{n}\hat{\theta}_i^{k_i}p_i^k\right] = \sum_j \alpha \hat{\theta}_j \frac{1}{n\alpha+m-1}\E[Y^{m-1}],
\end{align*}
while the expectation in the second term is
\begin{align*}
\E\left[\sum_{k_1+\cdots+k_n=m-1} k_j {m-1 \choose k_1, \ldots, k_n} \prod_{i=1}^{n}\hat{\theta}^{k_i}p_i\right] &=  \E\left[(m-1) p_j \hat{\theta}_j \left(\sum_{i=1}^{n}p_i \hat{\theta}_i\right)^{m-2}\right],
\end{align*}
so the second term becomes
\begin{align*}
& \sum_j \hat{\theta}_j \frac{1}{n\alpha+m-1} \E\left[\sum_{k_1+\cdots+k_n=m-1} k_j {m-1 \choose k_1, \ldots, k_n} \prod_{i=1}^{n}\hat{\theta}^{k_i}p_i\right] \\
& \qquad = \sum_j \hat{\theta}_j^2 \frac{m-1}{n\alpha+m-1} \E\left[p_j\left(\sum_{i=1}^{n}p_i \hat{\theta}_i\right)^{m-2}\right]\\
&\qquad = \sum_j \hat{\theta}_j^2 \frac{m-1}{n\alpha+m-1} \left(\frac{\alpha}{n\alpha+m-2}\E\left[Y^{m-2}\right] + \frac{m-2}{n\alpha+m-2}\E\left[p_j\left(\sum_{i=1}^{n}p_i \hat{\theta}_i\right)^{m-3}\right]\right),
\end{align*}
where the last equality follows by applying the argument above with $m-1$ instead of $m$. Putting it all together, 
\begin{align*}
\E\left[\left(\sum_{i=1}^{n}p_i \hat{\theta}_i\right)^m\right] &= \sum_j \alpha \hat{\theta}_j \frac{1}{n\alpha+m-1}\E\left[Y^{m-1}\right] \\
& \qquad + \sum_j \alpha \hat{\theta}^2_j \frac{m-1}{(n\alpha+m-1)(n\alpha+m-2)}\E\left[Y^{m-2}\right] \\
& \qquad + \sum_j \hat{\theta}^2_j \frac{(m-1)(m-2)}{(n\alpha+m-1)(n\alpha+m-2)}\E\left[p_j\left(\sum_{i=1}^{n}p_i \hat{\theta}_i\right)^{m-3}\right].
\end{align*}
Proceeding with the inductive argument, this gives
\begin{equation*}
\E\left[\left(\sum_{i=1}^{n}p_i \hat{\theta}_i\right)^m\right] = \sum_{k=0}^{m-1} \left(\sum_{j=1}^{n}\alpha \hat{\theta}_j^{m-k}\right) \frac{(n\alpha+k-1)!}{(n\alpha+m-1)!} \frac{(m-1)!}{k!} \E[Y^k],
\end{equation*}
as desired.
\end{proof}

For example, the lemma above gives
\begin{align}
\E[Y | \hat{\bm{\theta}}] &= \frac{1}{n} \sum_{i=1}^{n} \hat{\theta}_i = \SM \label{eq:first_moment} \\
\E[Y^2 | \hat{\bm{\theta}}] &= \frac{1}{n\alpha+1} \overline{\theta^2} + \frac{n\alpha}{n\alpha+1} \left(\overline{\theta}\right)^2 \label{eq:second_moment}\\
\E[Y^3 | \hat{\bm{\theta}}] &= \frac{2}{(n\alpha+2)(n\alpha+1)} \overline{\theta^3} + \frac{3 n \alpha}{(n\alpha+2)(n\alpha+1)}\overline{\theta^2}\overline{\theta} + \frac{n^2 \alpha^2}{(n\alpha+2)(n\alpha+1)} (\overline{\theta})^3, \label{eq:third_moment}
\end{align}
and it is not hard to compute higher-order moments as needed. From this, one can also obtain the unconditional moments: using the Law of Iterated Expectations,
\begin{equation} \label{eq:unconditional_mean}
 \E[Y] = \E[\E[Y | \hat{\theta}]] = \E\left[\frac{1}{n}\sum_{i=1}^n \hat{\theta}_i \right] = \theta,
\end{equation}
so $Y$ is unbiased, and by the Law of Total Variance,
\begin{align} 
\V[Y] &= \E[\V[Y | \hat{\theta}]] + \V[\E[Y | \hat{\theta}]] = \E\left[\frac{1}{n\alpha+1}s^2_{\hat{\theta}}\right]+\V\left[\frac{1}{n}\sum_{i=1}^{n}\hat{\theta}_i\right] \\
	&= \frac{1}{n\alpha+1}\E\left[\frac{1}{n}\sum_{i=1}^n (\hat{\theta}_i -\overline{\theta})^2\right] + \frac{\sigma^2}{n} = \frac{n-1}{n(n\alpha+1)}\sigma^2 + \frac{\sigma^2}{n} \\
	&= \frac{\sigma^2}{n}\left(\frac{n-1}{n\alpha+1}+1\right)= \frac{\sigma^2}{n} \frac{n(\alpha+1)}{n\alpha+1}. \label{eq:unconditional_variance}
\end{align}
Note in particular that $\V[Y] \stackrel{\alpha \to 0}{\longrightarrow} \sigma^2 = \V[\hat{\theta}_1]$ and $\V[Y] \stackrel{\alpha \to \infty}{\longrightarrow} \sigma^2/n= \V[\SM]$, as expected.

Consider now the conditional distribution of $Y$. Since it is a random mean of the $\hat{\theta}_i$, it admits an asymptotic Normal distribution concentrated around $\SM$.

\begin{proposition} \label{prop:bayesian_bootstrap_CLT}
Let $\hat{\theta}_1, \ldots, \hat{\theta} \iid [\theta, \sigma^2]$, suppose $\mathbf{p} \sim \Dir_n(\alpha, \ldots, \alpha)$, with $\alpha>0$, and $Y=\sum_{i=1}^{n}p_i \hat{\theta}_i$. If $\V[\hat{\theta}_i]=\sigma^2<\infty$, then, for almost all sequences $\hat{\theta}_1, \hat{\theta}_2, \ldots$, the conditional distribution of $\sqrt{n}(Y-\SM)$ converges to $N(0, \sigma^2/\alpha)$.
\end{proposition}

\begin{proof}
Since $\mathbf{p}\sim\Dir_n(\alpha, \ldots, \alpha)$, each coordinate $p_i$ can be written $p_i \stackrel{d}{=} G_i/\sum_{j=1}^{n}G_j$, where $G_j \iid \text{Gamma}(\alpha, 1)$. Thus, 
\begin{equation*}
 \sqrt{n}(Y - \SM) = \sqrt{n} \sum_{i=1}^{n} p_i (\hat{\theta}_i-\SM) = \frac{\frac{1}{\alpha\sqrt{n}} \sum_{i=1}^{n} G_i (\hat{\theta}_i - \SM)}{\frac{1}{\alpha n} \sum_{j=1}^{n}G_j}.
\end{equation*}
By the Strong Law of Large Numbers, $\frac{1}{\alpha n}\sum_{j=1}^{n}G_j \stackrel{as}{\longrightarrow} 1$. Conditioned on $\hat{\theta}_1, \ldots, \hat{\theta}_n$, $\sum_{i=1}^{n} G_i (\hat{\theta}_i - \SM)$ is a weighted sum of the $G_i$. Note
\begin{align*}
 \E\left[\sum_{i=1}^{n} G_i (\hat{\theta}_i - \SM) \mid \hat{\bm{\theta}}\right] &= \alpha \sum_{i=1}^{n} (\hat{\theta}_i - \SM) = 0. \\
 \V\left[\sum_{i=1}^{n} G_i (\hat{\theta}_i - \SM) \mid \hat{\bm{\theta}} \right] &= \sum_{i=1}^{n} (\hat{\theta}_i - \SM)^2 \V[G_i] = \alpha \sum_{i=1}^{n} (\hat{\theta}_i - \SM)^2.
\end{align*}
To obtain a Central Limit Theorem, it suffices to see Lindeberg's condition is implied by
\begin{equation*}
 \max_{1\leq i \leq n} \frac{(\hat{\theta}_i - \SM)^2}{\sum_{j=1}^{n}(\hat{\theta}_j - \SM)^2} =  \max_{1\leq i \leq n} \frac{\frac{1}{n} (\hat{\theta}_i - \SM)^2}{\frac{1}{n} \sum_{j=1}^{n}(\hat{\theta}_j - \SM)^2}  \stackrel{as}{\longrightarrow} 0,
\end{equation*}
and so $\frac{1}{\alpha \sqrt{n}}\sum_{i=1}^{n} G_i (\hat{\theta_i}-\SM) \Rightarrow N(0, \frac{\sigma^2}{\alpha})$ conditioned on $\hat{\theta}_1, \hat{\theta}_2, \ldots$, almost surely. An application of Slutsky's theorem then yields the proposition.
\end{proof}

Computing the exact distribution of $Y_j$, however, is more intricate. It can be done using the theory of Dirichlet means first developed by Cifarelli and Regazzini (\cite{cifarelli1990distribution}, \cite{Cifarelli_1994}). Through a connection between the Stieltjes transform of the distribution function of $Y$ and the Laplace transform of a related Gamma process, known as the Markov-Krein identity, they were able to obtain the cumulative distribution function for $Y$ in \cite{Cifarelli_1994}. 

\begin{proposition} \label{prop:density}
Consider $\mathbf{p}\sim \Dir_n(\alpha, \ldots, \alpha)$, and take $Y=\sum_{i=1}^{n}p_i\hat{\theta}_i$, with $\hat{\theta}_i$ fixed for $i=1, \ldots, n$. The cumulative distribution function of $Y$, denoted $F_{Y|\hat{\bm{\theta}}}$, is supported on $[\min_i \hat{\theta}_i, \max_i \hat{\theta}_i]$, degenerate at points $y=\hat{\theta}_1, \ldots, \hat{\theta}_n$, and absolutely continuous with respect to the Lebesgue measure as long as $\min_i \hat{\theta}_i < \max_i \hat{\theta}_i$.
For $y \neq \hat{\theta}_i$, the probability density function of $Y$ given $\hat{\theta}_1, \ldots, \hat{\theta}_n$, denoted $f_{Y|\hat{\bm{\theta}}}(y)$, can be written as:
\begin{enumerate}[label=(\roman*)]
\item if $\alpha>1/n$,
\begin{equation*}
f_{Y|\hat{\bm{\theta}}}(y) = - \frac{1}{\pi}\lim_{t \nearrow \infty} \int_{-\infty}^{y} \Im\left\{(n\alpha-1)\left(y - s + i/t\right)^{n\alpha-2} \prod_{k=1}^{n}|\hat{\theta}_k-s+i/t|^{-\alpha}\right\} ds;
\end{equation*}
furthermore, if $\alpha<1$,
\begin{equation*}
f_{Y|\hat{\bm{\theta}}}(y) = \frac{n\alpha-1}{\pi} \int_{-\infty}^{y} (y-s)^{n\alpha-2} \left(\prod_{i:\hat{\theta}_i\neq s} |\hat{\theta}_i-s|^{-\alpha}\right) \sin\left(\pi \alpha \sum_{i=1}^{n} \mathbb{I}_{[s\geq \hat{\theta}_{(i)}]} \right) ds,
\end{equation*}
where $\Im\left\{\cdot\right\}$ denotes the imaginary part of the complex number, and $\hat{\theta}_{(i)}$ the $i$-th smallest value in $\{\hat{\theta}_1, \ldots, \hat{\theta}_n\}$;
\item if $\alpha=1/n$, 
\begin{equation*}
f_{Y|\hat{\bm{\theta}}}(y) = \frac{1}{\pi} \sin\left(\pi \alpha \sum_{i=1}^{n}\mathbb{I}_{[y>\hat{\theta}_i]} \right) \prod_{i=1}^{n}|\hat{\theta}_i - y|^{-\alpha};
\end{equation*}
\item if $1/n>\alpha>0$,
\begin{align*}
f_{Y|\hat{\bm{\theta}}}(y) &= \frac{1-n\alpha}{\pi} \lim_{t \nearrow \infty} \int_{-\infty}^{t} \Im\bigg\{(n\alpha-1)(y-s+i/t)^{n\alpha-2} \\
& \hspace{11em} \left(\prod_{k=1}^{n}|\hat{\theta}_k - s + i/t|^{-\alpha}\right)
\left(\prod_{k=1}^{n}|\hat{\theta}_k-y+i/t|^{-\alpha}\right)\bigg\} ds.
\end{align*}
\end{enumerate}
\end{proposition}

\begin{proof}
Degeneracy of the cumulative distribution function is considered in Theorem 1 of \cite{Cifarelli_1994}, while absolute continuity with respect to Lebesgue measure is established in Proposition 4 of \cite{regazzini2002theory}. The density formulas are given as Proposition 9 of \cite{regazzini2002theory}, specialized here to the case where the underlying measure of the Dirichlet process $\tilde{\alpha}$ is supported on $\{\hat{\theta}_1, \ldots, \hat{\theta}_n\}$ with $\tilde{\alpha}(\hat{\theta}_i)=\alpha$ for $i=1, \ldots, n$. The general proof can be found in \cite{regazzini2000}.
\end{proof}

\begin{example}[Behavior at the median] \label{example:density_at_median}
Consider the behavior of $Y|\hat{\bm{\theta}}$ near the median, $m=\med(Y|\hat{\bm{\theta}})$, when $\alpha=1/n$, in which case its density can be given in explicit form. From Proposition \ref{prop:density},
\begin{equation*}
f_{Y|\hat{\bm{\theta}}}(m) = \frac{1}{\pi} \sin\left(\frac{\pi}{n} \sum_{i=1}^{n} \mathbb{I}_{[m>\hat{\theta}_i]} \right) \prod_{i=1}^{n} |\hat{\theta}_i - m|^{-1/n} = C \prod_{i=1}^{n} |\hat{\theta}_i - m|^{-1/n},
\end{equation*}
where $C>0$ is a positive constant. Using the Strong Law of Large Numbers, as $n \to \infty$, 
\begin{equation*}
\prod_{i=1}^{n} |\hat{\theta}_i - m|^{-1/n} = e^{-\frac{1}{n}\sum_{i=1}^{n} \log(|\hat{\theta}_i - m|)} \stackrel{n \to \infty}{\longrightarrow} e^{-\int \log(|x-\theta|)dF_{\hat{\bm{\theta}}}(x)},
\end{equation*}
where $F_{\hat{\bm{\theta}}}$ denotes the law of $\hat{\theta}$, and the last equality uses the fact that $m \stackrel{n \to \infty}{\longrightarrow} \theta$, shown in Proposition \ref{prop:bounds_medY_theta}. Thus,
\begin{equation*}
f_{Y|\hat{\bm{\theta}}}(m) \stackrel{n \to \infty}{\longrightarrow} C e^{-\int \log(|x-\theta|) dF_{\hat{\bm{\theta}}}(x)} > 0,
\end{equation*}
giving a way to specify, at least asymptotically, $f_{Y|\hat{\bm{\theta}}}(m)$. 
\end{example}

While Proposition \ref{prop:density} determines the conditional distribution $Y|\hat{\bm{\theta}}$, directly relating the unconditional distribution of $Y_j$ to $\hat{\bm{\theta}}$ requires different techniques. It can be done, for example, via the Laplace transform of $G_{n\alpha}Y_j$, with $G_{n\alpha}\sim\text{Gamma}(n\alpha)$ independent of $Y_j$, which also uniquely determines the distribution of $Y_j$:
\begin{equation}\label{eq:dirichlet_transform}
\E\left[e^{- \lambda G_{n\alpha}Y_j}\right] = \E\left[\left(1+\lambda Y_j\right)^{-n\alpha}\right] = \left(\E[(1+\lambda \hat{\theta}_i)^{-\alpha}]\right)^n.
\end{equation}
This is proved as Proposition 29 in \cite{pitman2018random}, dating back to \cite{von1941distribution} and \cite{watson1956joint}. The following is one of the few examples where the distribution of both $\hat{\theta}_i$ and $Y_j$ are known.

\begin{example}[Beta distribution] \label{example:dirichlet_mean_of_beta}
Suppose $\hat{\theta}_i \iid \text{Beta}(a, b)$, and $\mathbf{p}^{(j)} \iid \Dir_n(a+b, \ldots, a+b)$, so $\alpha=a+b$. In this case, $\hat{\theta}_i \iid \text{Beta}(a, b)$ is also defined by its generalized Stieltjes transform $\E[(1-\lambda\hat{\theta}_i)^{-(a+b)}] = (1-\lambda)^{-a}$. Hence, (\ref{eq:dirichlet_transform}) gives
\begin{align*}
\E\left[(1+\lambda Y_j)^{-n(a+b)}\right] &= \left(\E\left[(1+\lambda\hat{\theta}_i)^{-n(a+b)}\right]\right)^n = (1+\lambda)^{-na},
\end{align*}
which in turn implies, by the Stieltjes transform characterization, that $Y_j \sim \text{Beta}(na, nb)$. In particular,
\begin{align*}
\E[Y_j] &= \frac{na}{na+nb} = \frac{a}{a+b}, \\
\V[Y_j] &= \frac{ab}{(a+b)^2(na+nb-1)},
\end{align*}
agreeing with equations (\ref{eq:unconditional_mean}) and (\ref{eq:unconditional_variance}) for $\hat{\theta}_i \iid \text{Beta}(a, b)$. Also,
\begin{equation*}
\skew(Y_j)=\frac{2n^{3/2}(b-a)\sqrt{a+b+1/n}}{n^2(a+b+2/n)\sqrt{ab}}\approx \frac{2(b-a)}{\sqrt{n}\sqrt{ab(a+b)}} \stackrel{n \to \infty}{\longrightarrow} 0,
\end{equation*}
so the $Y_j$ are more concentrated and symmetrized around its mean than the original sample $\hat{\theta}_1, \ldots, \hat{\theta}_n$. See Figure \ref{fig:Y_transform}. Further, note the goal is to estimate $\theta=\E[\hat{\theta}_i]=a/(a+b)$, while the Bayesian median of means is estimating
\begin{equation*}
\med(Y_j)  \approx \frac{a - \frac{1}{3n}}{a + b - \frac{2}{3n}},
\end{equation*}
where this approximation for the median of a Beta distribution is valid for $a, b\geq 1/n$ (see \cite{kerman2011closed}). Hence, the bias introduced in using $\BMM$ to estimate $\theta$ is fairly small.
\end{example}

\begin{figure}[!htbp]
 \begin{center}
{\fontfamily{cmss}\selectfont \textbf{Distributions of $\hat{\theta}$ and $Y$}} 

 \includegraphics[width=.8\textwidth]{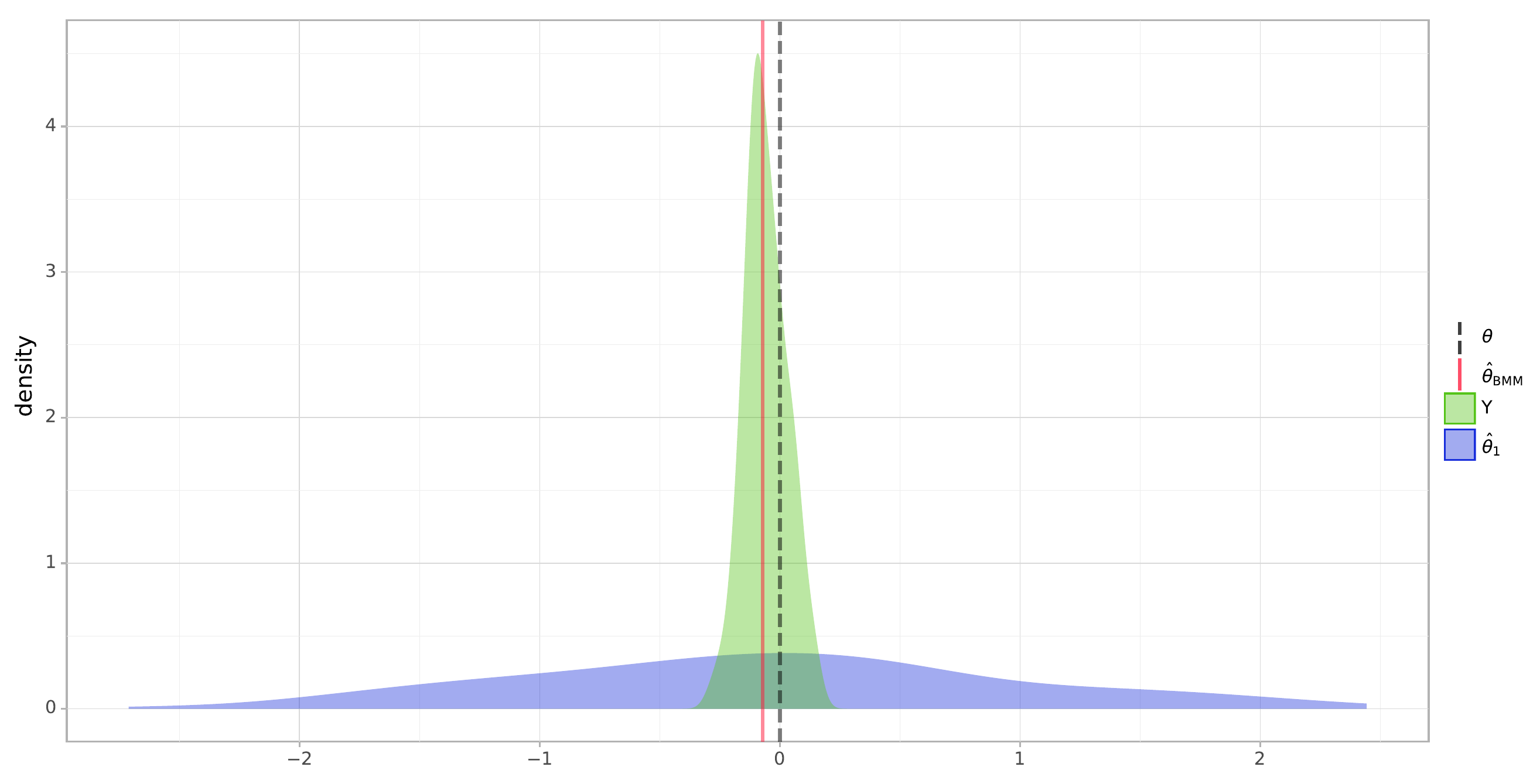}

\includegraphics[width=.8\textwidth]{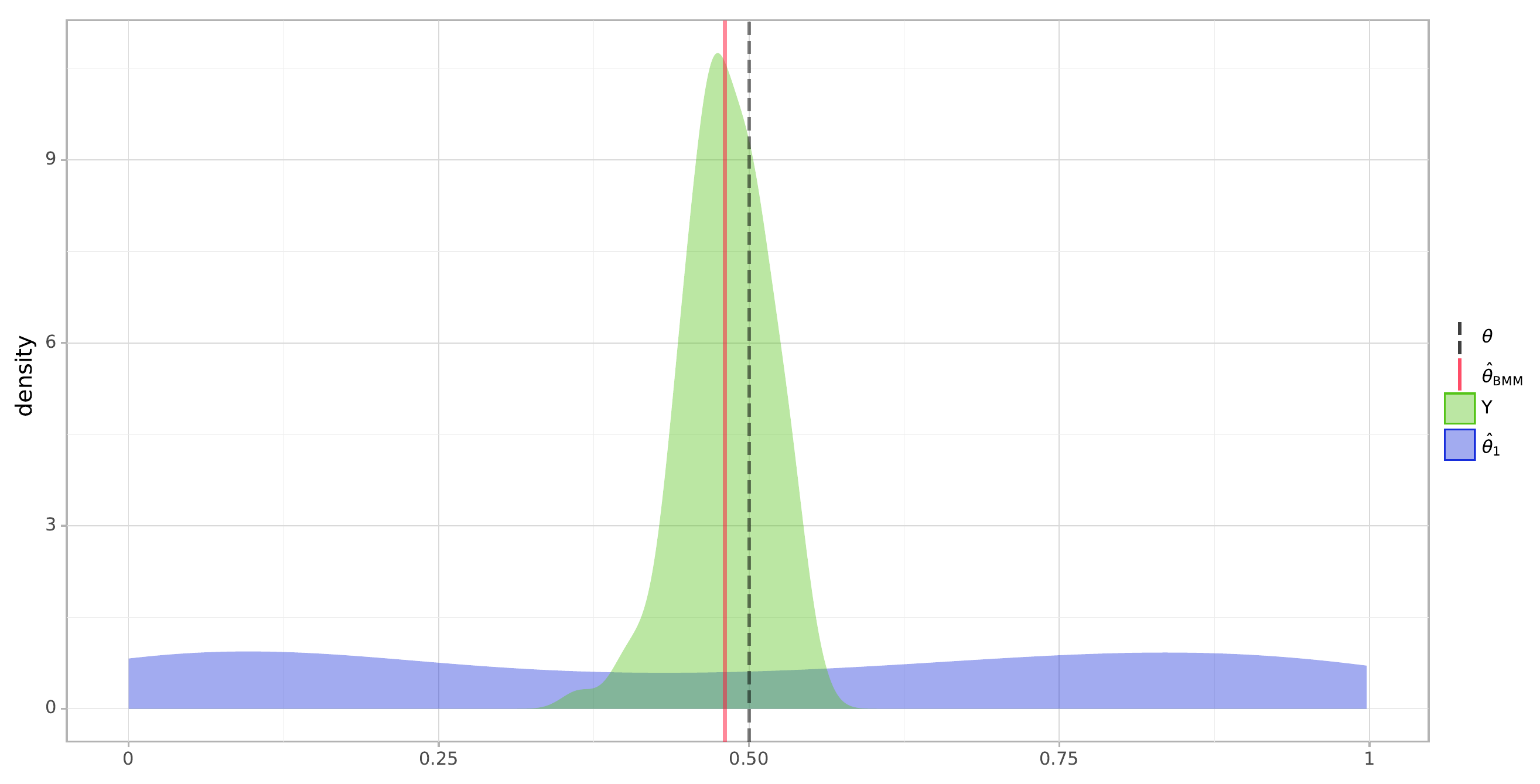}

\includegraphics[width=.8\textwidth]{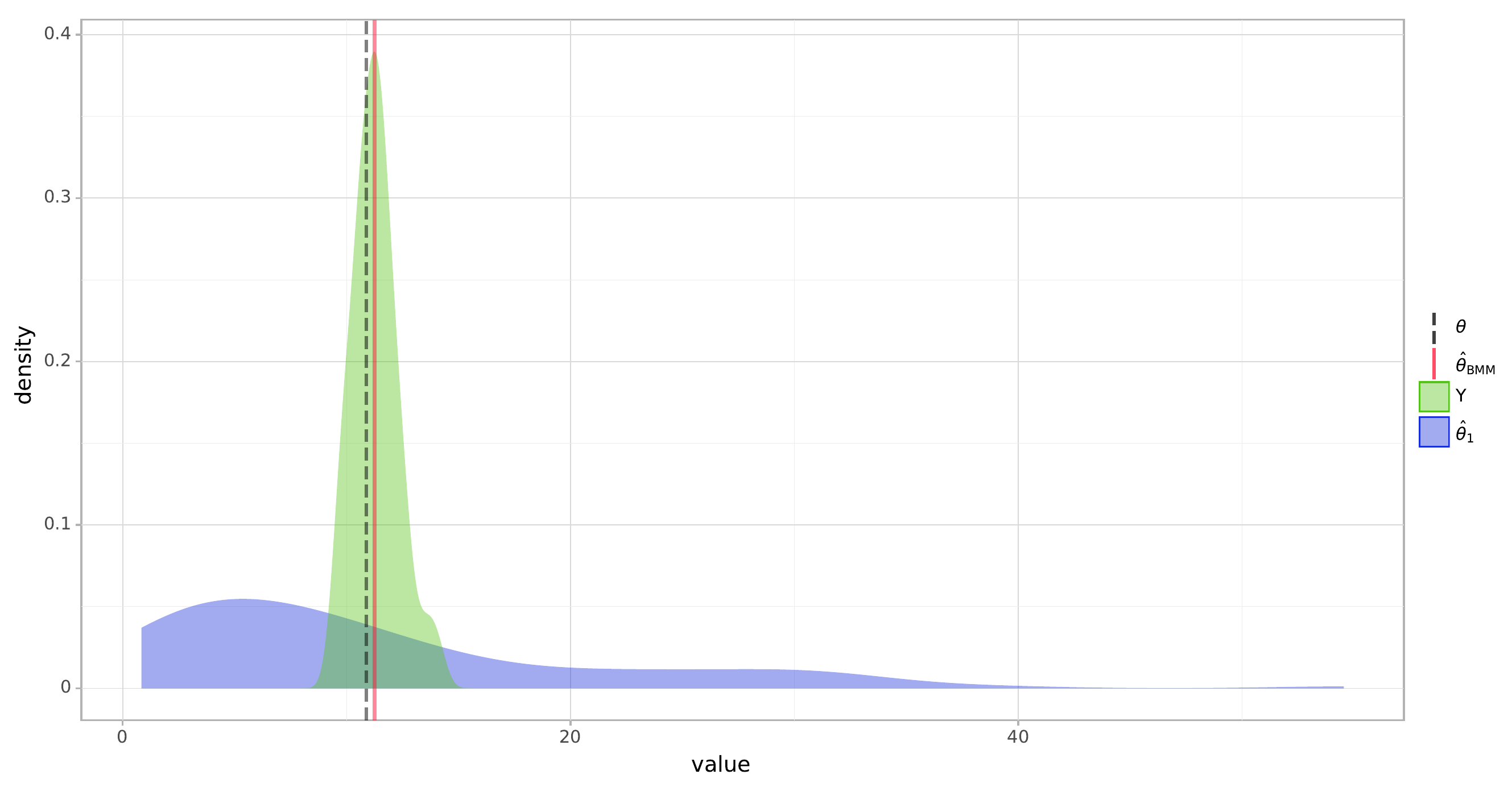}

\caption{In each figure, $\hat{\theta}_1, \ldots, \hat{\theta}_{100}$ is respectively sampled from (i) $N(0, 1)$, (ii) $\text{Beta}(\frac{1}{2}, \frac{1}{2})$ and (iii) $10\cdot \text{Exp}(1)+\text{Weibull}(1, 3)$, shown in blue. The distribution of $Y_1, \ldots, Y_{100}$ is shown in green, with the solid and dashed lines indicating $\BMM$ and $\theta$. Here, $n=J=100$, and $\alpha=1$. Note the distribution of $Y$ is close to $\theta$, even in the presence of multimodality or fat tails, with much smaller variance and skewness.}
 \label{fig:Y_transform}
 \end{center}
\end{figure}

The previous example also sheds light on the effects the Bayesian median of means procedure has on symmetric and unimodal distributions. In particular, if $\hat{\theta}_i$ comes from a symmetric distribution, then $Y$ is also symmetrically distributed (so the median adds no additional bias). 

\begin{proposition}
Let $\hat{\theta}_1, \ldots, \hat{\theta}_n \iid F_{\hat{\bm{\theta}}}$, and assume $F_{\hat{\bm{\theta}}}$ is symmetric, with density $f_{\hat{\bm{\theta}}}$, mean $\theta$ and variance $\sigma^2$. If $\mathbf{p}\sim \text{Dir}(\alpha, \ldots, \alpha)$ and $Y = \sum_{i=1}^{n} p_i \hat{\theta}_i$, then the distribution of $Y$ is also symmetric. If in addition $F_{\hat{\bm{\theta}}}$ is unimodal, then so is the distribution of $Y$. The converse is not true.
\end{proposition}

\begin{proof}
First, assume $\hat{\theta}_1, \ldots, \hat{\theta}_n$ is coming from a symmetric distribution, and without loss of generality take the center of symmetry to be $0$. Then, for any constant $p_i$, $p_i \hat{\theta}_i \stackrel{d}{=}-p_i\hat{\theta}_i$, and because $\hat{\theta}_1, \ldots, \hat{\theta}_n$ are iid, $\sum_{i=1}^{n}p_i \hat{\theta}_i \stackrel{d}{=} -\sum_{i=1}^{n}p_i \hat{\theta}_i$. Thus,
\begin{align*}
P[Y \leq y] &= P\left[\sum_{i=1}^{n} \hat{\theta}_i p_i \leq y\right] = \E\left[P\left[\sum_{i=1}^{n} \hat{\theta}_i p_i \leq y \mid \mathbf{p}\right]\right] = \E\left[P\left[\sum_{i=1}^{n} \hat{\theta}_i p_i > -y \mid \mathbf{p}\right]\right] = P[Y > -y],
\end{align*}
so $Y$ is also symmetric. 

Now, assume additionally that $F_{\hat{\theta}}(t)$ is unimodal. Conditional on $\mathbf{p}$, the sum of independent, symmetric and unimodal random variables is still symmetric and unimodal (see \cite{purkayastha1998simple}), so $Y$ itself is also symmetric and unimodal. The unconditional case follows by taking expectations as above.

On the other hand, note that by Example \ref{example:dirichlet_mean_of_beta}, a multimodal distribution such as $\hat{\theta}_i \iid \text{Beta}(\frac{1}{2}, \frac{1}{2})$ leads to $Y \sim \text{Beta}(\frac{n}{2}, \frac{n}{2})$, which is unimodal if $n>2$.
\end{proof}

\subsection{Theoretical Guarantees} \label{sec:theoretical_guarantees}

This section establishes the main theoretical results in the paper. Section \ref{sec:approx_BMM_with_median} characterizes the convergence of $\BMM$ to $\med(Y|\hat{\theta}_1, \ldots, \hat{\theta}_n)$, as $J\to\infty$ and for fixed $n$. Section \ref{sec:concentration_bounds} then bounds the error in estimating $\theta$ using $\med(Y|\hat{\theta}_1, \ldots, \hat{\theta}_n)$ in finite samples. Section \ref{sec:asymptotic_approximation} considers asymptotic guarantees.

\subsubsection{Approximating $\widehat{\theta}_{\text{BMM}}$ with $\med(Y\mid\widehat{\theta})$} \label{sec:approx_BMM_with_median}

The limiting distribution of $\BMM$ as $J\to\infty$ can be found via the Central Limit Theorem for medians and Proposition \ref{prop:density}.

\begin{proposition} \label{prop:medianCLT}
Let $\hat{\theta}_1, \ldots, \hat{\theta}_n \in \R$ be fixed, and consider $Y_j = \sum_{i=1}^{n} p^{(j)}_i \hat{\theta}_i$, where $\mathbf{p}^{(j)}\iid \Dir_n(\alpha, \ldots, \alpha)$ for $j=1, \ldots, J$. If $\BMM^{(J)} = \widhat{\med}(Y_1, \ldots, Y_J | \hat{\bm{\theta}})$ is the sample median and $m=\med(Y|\hat{\bm{\theta}})$ is the population median, with $m \neq \hat{\theta}_i$, $i=1, \ldots, n$ and $f_{Y|\hat{\bm{\theta}}}(m)>0$, then, as $J \to \infty$,
\begin{equation*}
 \BMM^{(J)} \Longrightarrow N\left(m, \frac{1}{4 J f^2_{Y|\hat{\bm{\theta}}}(m)}\right),
\end{equation*}
with $f_{Y|\hat{\bm{\theta}}}(y)$ given in Proposition \ref{prop:density}.
\end{proposition}

\begin{proof}
Since conditional on $\hat{\theta}_1, \ldots, \hat{\theta}_n$, $\left\{Y_j\right\}_{j=1}^J$ are iid, defining $W_J = \frac{1}{J} \sum_{j=1}^{J} \mathbb{I}_{[Y_j \leq m]}$, the Central Limit Theorem gives $\sqrt{J}(W_J - 1/2) \Rightarrow N(0, 1/4)$. Considering the generalized inverse $F_{Y|\hat{\bm{\theta}}}^{-1}(t)$ and recalling $\frac{d}{dt}F_{Y|\hat{\bm{\theta}}}^{-1}(t)=\frac{1}{f_{Y|\hat{\bm{\theta}}}(F^{-1}_{Y|\hat{\bm{\theta}}}(t))}$, a straightforward application of the Delta Method yields $\sqrt{J}(\BMM^{(J)}-m) \Rightarrow N(0, 1/(4f^2_{Y|\hat{\bm{\theta}}}(m)))$, where $f_{Y|\hat{\bm{\theta}}}(y)$ is defined in Proposition \ref{prop:density}.
\end{proof}

It is possible to obtain better control over the fluctuations between $\BMM$ and $\med(Y|\hat{\bm{\theta}})$ via finite-sample concentration bounds. In particular, this suggests how to set $J$ as a function of $n$ to obtain concentration around the median for any small $t=O(1/n)$.

\begin{proposition} \label{prop:bound_BMM_m}
Consider $\hat{\theta}_1, \ldots, \hat{\theta}_n \in \R$, and let $Y_j = \sum_{i=1}^{n} p^{(j)}_i \hat{\theta}_i$, where $\mathbf{p}^{(j)} \iid \Dir_n(\alpha, \ldots, \alpha)$ for $j=1, \ldots, J$. Let $\BMM=\widhat{\med}(Y_1, \ldots, Y_J | \hat{\bm{\theta}})$ and assume $m=\med(Y|\hat{\bm{\theta}})$ is unique. For $t$ small enough that $\min_{[m, m+t]} f_{Y|\hat{\bm{\theta}}}(\xi) > C > 0$ for some $C=C(\hat{\theta}_1, \ldots, \hat{\theta}_n)$, it holds that
 \begin{equation*}
  P_{\mathbf{p}}\left[|\BMM - m| > t \mid \hat{\bm{\theta}} \right] \leq 2 e^{-2Jt^2C^2},
 \end{equation*}
so $\BMM$ concentrates exponentially fast.
\end{proposition}

\begin{proof}
To establish concentration, split
\begin{equation} \label{eq:bound_BMM_m}
P[|\BMM-m|\geq t] = P[\BMM-m\geq t] + P[\BMM-m\leq -t],
\end{equation}
and bound each term in turn. For the first, note that conditionally on $\hat{\bm{\theta}}$ the $Y_1, \ldots, Y_J$ are independent, so the sample median $\BMM$ is bigger than $m+t$ if at least half of the $J$ points is so. Thus,
\begin{align*}
P[\BMM \geq m + t] &= P\left[\sum_{j=1}^{J} \mathbb{I}_{[Y_j \geq m+t]} \geq \frac{J}{2} \right] = P\left[\text{Bin}(J, q_+(t)) \geq \frac{J}{2}\right],
\end{align*}
where
$q_+(t) = P_{\mathbf{p}}\left[Y \geq m+t \mid \hat{\bm{\theta}}\right] < \frac{1}{2}$. The Hoeffding bound gives
\begin{equation*}
P\left[\text{Bin}(J, q_+(t)) \geq J(q_+(t) + q_+(t)\left(\frac{1}{2q_+(t)}-1\right))\right] \leq e^{-2J\left(\frac{1}{2}-q_+(t)\right)^2}.
\end{equation*}
Note the density of $Y$, $f_{Y|\hat{\bm{\theta}}}(y)$ is strictly positive in $[m, m+t]$ since the support of $Y|\hat{\bm{\theta}}$ is the convex hull of the $\hat{\theta}_i$. Thus, there exists $C>0$ such that $\min_{[m , m+t]}f_{Y|\hat{\bm{\theta}}}(\xi) > C$, and
\begin{equation*}
\frac{1}{2} - q_+(t) = P\left[Y \geq m \mid \hat{\bm{\theta}}\right] - P\left[Y \geq m+t \mid \hat{\bm{\theta}}\right] = \int_{m}^{m+t} f_{Y |\hat{\bm{\theta}}}(\xi) d \xi \geq t \cdot \min_{[m, m+t]} f_{Y|\hat{\bm{\theta}}}(\xi) > t C,
\end{equation*}
so it is possible to conclude
\begin{equation} \label{eq:BMM_med_bound}
P\left[\BMM^{(J)} \geq m+t \mid \hat{\bm{\theta}}\right] \leq e^{-2J \left(\frac{1}{2}-q_+(t)\right)^2} \leq e^{-2Jt^2C^2}.
\end{equation}
Bounding the other term in (\ref{eq:bound_BMM_m}) similarly yields the proposition.
\end{proof}

In some cases, it is possible to get an asymptotic description of the constant $C$ above.

\begin{example}[Value of $C$]
Recall from Example \ref{example:density_at_median} that when $\alpha=1/n$, if $F_{\hat{\theta}}$ denotes the cdf of $\hat{\theta}_i$ and $\tilde{C}=\frac{1}{\pi} \sin\left(\frac{\pi}{n} \sum_{i=1}^{n} \mathbb{I}_{[m>\hat{\theta}_i]} \right)$, then, as $n \to \infty$, the density at the median $m$ converges to
\begin{equation*}
	f_{Y | \hat{\bm{\theta}}}(m) \stackrel{n \to \infty}{\longrightarrow} \tilde{C} e^{- \int \log (|x-\theta| dF_{\hat{\theta}}(x))} = \frac{1}{\pi} e^{- \int \log (|x-\theta| dF_{\hat{\theta}}(x))},
\end{equation*}
so the bound in (\ref{eq:BMM_med_bound}) becomes
\begin{equation*}
P\left[\BMM^{(J)} \geq m+t\right] \leq e^{-\frac{2}{\pi^2}Jt^2 \cdot e^{-2\int \log(|x-\theta|) dF_{\hat{\theta}(x)}}}. \qedhere
\end{equation*}
\end{example}

The bound in Proposition \ref{prop:bound_BMM_m} works well for small values of $t$. For example, it ensures that $J=O(n^2)$ suffices to guarantee that, with constant probability, $|\BMM-m|$ is $O(\frac{1}{n})$ as $n \to \infty$. For large values of $t$, a different bound can yield better results.

\begin{proposition} \label{prop:bound_BMM_m_large}
Let $Y_1, \ldots, Y_J$ be iid given $\hat{\bm{\theta}}=(\hat{\theta}_1, \ldots, \hat{\theta}_n)$, and assume $m=\med(Y|\hat{\bm{\theta}})$ is unique. For $t>2\sqrt{\V[Y|\hat{\bm{\theta}}]}$,
\begin{equation*}
P\left[|\BMM-m| > t\right] \leq \left(4 \sqrt{\V[Y|\hat{\bm{\theta}}}]\right)^{J/2} \cdot \frac{1}{t^{J/2}}.
\end{equation*}
\end{proposition}

\begin{proof}
From the definition of conditional median and Jensen's inequality,
\begin{equation*}
\E\left[|Y-m|\right] \leq \E\left[|Y-\E[Y]|\right] \leq \sqrt{\E\left[(Y-\E[Y])^2\right]} = \sqrt{\V[Y|\hat{\bm{\theta}}]},
\end{equation*}
so, from Markov's inequality,
\begin{equation*}
P\left[|Y-m| \geq t \mid \hat{\bm{\theta}}\right] \leq \frac{\sqrt{\V[Y|\hat{\bm{\theta}}]}}{t}.
\end{equation*}

Now, if the sample median deviates from the conditional median, $m$, by more than $t$ then at least $J/2$ of the $Y_j$ are more that than $t$ apart from $m$. Since $Y_1, \ldots, Y_J$ are conditionally iid, this implies
\begin{equation*}
P\left[|\BMM - m| \geq t \mid \hat{\bm{\theta}}\right] \leq P\left[\text{Bin}\left(J, \frac{1}{t}\sqrt{\V[Y|\hat{\bm{\theta}}]}\right) \geq J/2\right].
\end{equation*}
To bound this probability, recall the Chernoff bound for $X \sim \text{Bin}(J, p)$ yields, for $0 \leq a \leq J(1-p)$,
\begin{align*}
P\left[X \geq Jp + a\right] \leq \inf_{t \geq 0} e^{-at -Jpt +J \ln (1+p(e^t-1))} = e^{-(Jp+a)\ln \frac{Jp+a}{Jp} - (J-Jp-a) \ln \frac{J-Jp-a}{J-Jp}}.
\end{align*}
Thus, for $t>2 \sqrt{\V[Y|\hat{\bm{\theta}}]}$,
\begin{align*}
P\left[|\BMM - m| \geq t | \hat{\bm{\theta}}\right] &\leq P\left[\text{Bin}\left(J, \frac{1}{t}\sqrt{\V[Y|\hat{\bm{\theta}}]}\right) \geq \frac{J}{t} \sqrt{\V[Y|\hat{\bm{\theta}}]}+\left(\frac{J}{2}-\frac{J}{t}\sqrt{\V[Y|\hat{\bm{\theta}}]}\right) \right]\\
&\leq e^{-\frac{J}{2} \ln\left(\frac{J/2}{(J/t)\sqrt{\V[Y|\hat{\bm{\theta}}]}}\right)- \frac{J}{2} \ln \left(\frac{J/2}{J(1-(1/t)\sqrt{\V[Y|\hat{\bm{\theta}}]})}\right)}\\
&\leq e^{-\frac{J}{2}\ln \left(\frac{t}{2\sqrt{\V[Y|\hat{\bm{\theta}}]}}\right) - \frac{J}{2}\ln\left(\frac{1}{2}\right)} = e^{-\frac{J}{2} \ln\left(\frac{t}{4\sqrt{\V[Y|\hat{\bm{\theta}}]}}\right)} \\
&=\left(4 \sqrt{\V[Y|\hat{\bm{\theta}}]}\right)^{J/2} \frac{1}{t^{J/2}},
\end{align*}
which gives suitable control over $|\BMM - m|$ for large $t$.
\end{proof}

Both Propositions \ref{prop:bound_BMM_m} and \ref{prop:bound_BMM_m_large} are useful in quantifying how big $J$ should be to control $|\BMM - m|$. An alternative perspective on how large $J$ should be comes from analyzing the asymptotic bias and variance incurred by $\BMM$ in estimating $m$. Both decrease as $O(\frac{1}{J})$, as the next proposition shows.

\begin{proposition} \label{prop:BMM_m_bias_and_variance}
Let $Y_1, \ldots, Y_J$ be conditionally iid and $m = \med(Y|\hat{\bm{\theta}})$ unique, and define $\BMM = \widhat{\med}(Y_1, \ldots, Y_J|\hat{\bm{\theta}})$. Then:
\begin{equation*}
 \E[\BMM] = m + O\left(\frac{1}{J}\right),\qquad \V[\BMM] = O\left(\frac{1}{J}\right).
\end{equation*}
In particular both the bias of $\BMM$ (in estimating $m$) and the variance decrease as $O(1/J)$.
\end{proposition}

\begin{proof}
Assume for simplicity $J=2k+1$ so $\BMM = Y_{(k)}$. Thus,
\begin{align} \label{eq:E_BMM_r}
 \E[\BMM^r] %
	 &= \int_{-\infty}^{\infty} t^r \frac{(2k+1)!}{(k!)^2} f_{Y|\hat{\bm{\theta}}}(t) \left[F_{Y|\hat{\bm{\theta}}}(t)\left(1-F_{Y|\hat{\bm{\theta}}}(t)\right)\right]^k dt \\
	 &= \frac{(2k+1)!}{(k!)^2} \int_{-\infty}^{\infty} e^k\left(\log F_{Y|\hat{\bm{\theta}}}(t) - \log(1-F_{Y|\hat{\bm{\theta}}}(t))\right) t^r f_{Y|\hat{\bm{\theta}}}(t) dt.
\end{align}
The idea will be to approximate this integral via a Laplace approximation. First, using Stirling's formula, 
\begin{equation*}
 \frac{(2k+1)!}{(k!)^2} = 2k \left(1+\frac{1}{2k}\right) \frac{1}{\sqrt{\pi k} } 2^{2^k}\left(1 + O\left(\frac{1}{k}\right)\right) = 2^{2k+1} \sqrt{\frac{k}{\pi}} \left(1+O\left(\frac{1}{k}\right)\right).
\end{equation*}
Since $f_{Y|\hat{\bm{\theta}}}(m)>0$, $\log F_{Y|\hat{\bm{\theta}}}(t) - \log\left(1-F_{Y|\hat{\bm{\theta}}}(t)\right)$ is maximized at the conditional median, $m$. Taking $r=1$, the Laplace approximation yields
\begin{equation*}
 \int_{-\infty}^{\infty} e^k\left(\log F_{Y|\hat{\bm{\theta}}}(t) - \log(1-F_{Y|\hat{\bm{\theta}}}(t))\right) t^r f_{Y|\hat{\bm{\theta}}}(t) dt = 2^{-(2k+1)}m \sqrt{\frac{\pi}{k}} \left(1+O\left(\frac{1}{k}\right)\right).
\end{equation*}
Putting it together, 
\begin{equation*}
 \E[\BMM] = 2^{2k+1}\sqrt{\frac{k}{\pi}} \left(1+O\left(\frac{1}{k}\right)\right) 2^{-(2k+1)} m \sqrt{\frac{\pi}{k}} \left(1+O\left(\frac{1}{k}\right)\right) = m \left(1+O\left(\frac{1}{k}\right)\right),
\end{equation*}
so $\E[\BMM]-m = O(1/k)$, and the bias of $\BMM$ in estimating $m$ decreases as $O(1/J)$. In fact, a more careful approximation yields
\begin{equation*}
 \E\left[\BMM\right] = m - \frac{f'_{Y|\hat{\bm{\theta}}}(m)}{8J(f_{Y|\hat{\bm{\theta}}}(m))^3} + O\left(\frac{1}{J^2}\right).
\end{equation*}

Using $r=2$ in (\ref{eq:E_BMM_r}) leds to the following expression for the variance:
\begin{equation*}
 \V[\BMM] = \frac{1}{4J(f_{Y|\hat{\bm{\theta}}}(m))^2} + O\left(\frac{1}{J^2}\right).
\end{equation*}
Note this matches the asymptotic variance found in Proposition \ref{prop:medianCLT}. Hence, the variance of $\BMM$ also decreases as $O\left(\frac{1}{J}\right)$.
\end{proof}

In terms of variance, note that $\V[\SM]=\sigma^2/n$ while $\V[\BMM]=1/(4J(f_{Y|\hat{\bm{\theta}}}(m))^2) + O(1/J^2)$. Suppose $J=n$ Dirichlets are sampled to form $\BMM$. Then, as $n\to\infty$, $\V[\SM]>\V[\BMM]$ if $4\sigma^2 f_{Y|\hat{\bm{\theta}}}^2(m)>1$. The larger the tails of the underlying distribution of $\hat{\theta}_i$ the more $\sigma^2$ grows while $f_{Y|\hat{\bm{\theta}}}(m)$ stays the same, thereby making $\BMM$ more attractive. Of course, variance is only half of the picture. Since the bias in using $\BMM$ to estimate $m$ is of order $O(1/J)$, the next subsection considers how far apart $m$ and $\theta$ can be.

\subsubsection{Concentration bounds for $|m - \theta|$} \label{sec:concentration_bounds}

Since $\BMM$ concentrates around $m=\med(Y|\hat{\theta}_1, \ldots, \hat{\theta}_n)$, it is important to characterize how $m$ and $\theta$ differ. The results in this section show that when $\V[\hat{\theta}_i]=\sigma^2$ is small, $m$ and $\theta$ are close and so are $\BMM$ and $\SM$. 

The following mean-median inequality will prove to be extremely valuable. It guarantees that the bias in estimating the median instead of the mean is bounded by the square root of the variance. 

\begin{proposition} \label{prop:mean-median}
If $X$ is a random variable with finite variance $\sigma^2$, then
\begin{equation*}
 |\med(X)-\E[X]| \leq \sigma,
\end{equation*}
so the distance between mean and median is at most a standard deviation. Similarly, conditional on random variables $\hat{\theta}_1, \ldots, \hat{\theta}_n$, almost surely,
\begin{equation*}
|\med(X | \hat{\bm{\theta}}) - \E[X | \hat{\bm{\theta}}]| \leq \sqrt{\V[X| \hat{\bm{\theta}}]}.
\end{equation*}
\end{proposition}

\begin{proof}
Using Jensen's inequality and the fact that the median minimizes the $L_1$ loss, 
\begin{align*}
|\med(X) - \E[X]| &= |\E[X-\med(X)]| \leq \E[|X-\med(X)|] \leq \E[|X-c|],
\end{align*}
for any $c \in \R$. Taking $c=\E[X]$,
\begin{align*}
 |\med(X) - \E[X]| &\leq \E[|X-\E[X]|] \leq \sqrt{\E[(X-\E[X])^2]} = \sigma.
\end{align*}
The conditional result follows analogously by using the conditional Jensen's inequality.
\end{proof}

This inequality can often be strengthened. For example, one could also have taken $c=0$ in the proposition above to obtain $|\med(X) - \E[X]| \leq \E[|X|]$, which sometimes yields stronger results. For unimodal distributions, the upper bound can be tightened to $\sqrt{0.6}\sigma$ (see \cite{basu1997mean}). If $X$ concentrates exponentially around the mean or median, it is also possible to obtain better results, as the next proposition shows.

\begin{proposition}
Let $X$ be a random variable with mean $\theta$ and median $m$. If there exist $a, b >0$ such that $P[|X-\med(X)|>t] \leq a e^{-\frac{t^2}{b}}$ or $P[|X-\E[X]|>t]\leq \frac{a}{2} e^{-\frac{4t^2}{b}}$, then
\begin{equation*}
 |\med(X) - \E[X]| \leq \min \left(\sqrt{ab}, a\sqrt{\pi b}/2\right).
\end{equation*}
\end{proposition}

\begin{proof} First, note $P[|X-\E[X]|]\leq \frac{a}{2}e^{-\frac{4t^2}{b}}$ implies $P[|X-\med(X)|>t] \leq a e^{-\frac{t^2}{b}}$. Indeed, consider two cases: (i) $t \geq 2 |\E[X]-\med(X)|$, and (ii) $t < 2 |\E[X]-\med(X)|$. For (i), note 
\begin{align*}
P[|X - \med(X)| \geq t] &\leq P[|X-\med(X)|\geq t/2+|\E[X]-\med(X)|] \\
&\leq P[|X-\E[X]| \geq t/2] \leq \frac{a}{2}e^{-\frac{t^2}{b}},
\end{align*}
using the fact that $|X-\med(X)|\leq |X-\E[X]|+|\E[X]-\med(X)|$. For (ii), by the definition of median,
\begin{align*}
\frac{1}{2} &\leq P\left[|X-\med(X)|\geq 0\right] \leq P\left[|X-\E[X]|\geq |\E[X]-\med(X)|\right] \\
&\leq a e^{-\frac{(\E[X]-\med(X))^2}{b}} \leq ae^{- \frac{t^2}{4b}},
\end{align*}
which implies $2ae^{-\frac{t^2}{4b}}\geq 1$, so the bound must hold. This proves exponential mean concentration implies exponential median concentration.	

Now, proceeding as in Proposition \ref{prop:mean-median},
 \begin{align*}
  |\med(X) - \E[X]| &\leq \E[|X-\med(X)|] = \int_{0}^{\infty} P[|X-\med(X)|>t] dt \\
  &\leq \int_{0}^{\infty} ae^{-t^2/b} dt =a\sqrt{\pi b}/2.
 \end{align*}
On the other hand,
\begin{align*}
 \V[X] &= \V[X-\med(X)] \leq \E[(X-\med(X))^2] \\
 &= \int_{0}^{\infty} P[(X-\med(X))^2>t] dt \\
 &\leq \int_{0}^{\infty} ae^{-t/b} dt = ab,
\end{align*}
so, using Proposition \ref{prop:mean-median}, 
\begin{equation*}
 |\med(X) - \E[X]| \leq \sqrt{\V[X]} \leq \sqrt{ab}.
\end{equation*}
Hence, $|\med(X) - \E[X]|\leq \min(\sqrt{ab}, a\sqrt{\pi b}/2)$.
\end{proof}

Note that taking expectation on the result of Proposition \ref{prop:mean-median} gives the following bound on the bias of the conditional median:
\begin{equation*}
|\E[X]-\E_{\hat{\bm{\theta}}}[\med(X|\hat{\bm{\theta}})]| \leq \E_{\hat{\bm{\theta}}}[|\E[X|\hat{\bm{\theta}}]-\med(X|\hat{\theta})|] \leq \E_{\hat{\bm{\theta}}}\left[\sqrt{\V[X|\hat{\bm{\theta}}]}\right] \leq \V[X].
\end{equation*}

These results provide useful bounds for the Bayesian median of means estimator.

\begin{proposition} \label{prop:bounds_medY_theta}
 Suppose $\hat{\theta}_1, \ldots, \hat{\theta}_n \iid [\theta, \sigma^2]$, $\mathbf{p} \sim \Dir_n(\alpha, \ldots, \alpha)$ and take $Y= \sum_{i=1}^{n} p_i \hat{\theta}_i$. Then,
 \begin{equation*}
  |\med(Y)- \theta| \leq \sqrt{\frac{\sigma^2}{n} \frac{n(\alpha+1)}{n\alpha+1}},
 \end{equation*}
and, conditioning on $\hat{\theta}_1, \ldots, \hat{\theta}_n$, almost surely
\begin{equation*}
 |\med(Y|\hat{\bm{\theta}}) - \SM| \leq \sqrt{\frac{1}{n\alpha+1}s^2_{\hat{\bm{\theta}}}}.
\end{equation*}
In particular, as $n \to \infty$, the unconditional median converges to $\theta$ at least as fast as $O(1/\sqrt{n})$ and the conditional median converges to $\SM$ at the same speed.
\end{proposition}

\begin{proof}
From Proposition \ref{prop:mean-median} and (\ref{eq:unconditional_variance}),
\begin{equation*}
 |\med(Y)-\theta| = |\med(Y)-\E[Y]| \leq \sqrt{\V[Y]} =\sqrt{\frac{\sigma^2}{n} \frac{n(\alpha+1)}{n\alpha+1}},
\end{equation*}
while Proposition \ref{prop:mean-median} and (\ref{eq:variance}) imply
\begin{equation*}
 |\med(Y|\hat{\bm{\theta}})-\SM| = |\med(Y|\hat{\bm{\theta}})-\E[Y|\hat{\bm{\theta}}]| \leq \sqrt{\V[Y|\hat{\bm{\theta}}]} =\sqrt{\frac{1}{n\alpha+1}s^2_{\hat{\bm{\theta}}}}. \qedhere
\end{equation*}
\end{proof}

The results above highlights the impact $\alpha$ has on the difference between $\med(Y|\hat{\bm{\theta}})$ and $\theta$. For instance, if $\alpha = O(n)$, the bound implies a maximum bias of order $O(1/n)$. In fact, any value of $O(n^{\kappa})$, for $\kappa>0$, implies the bias in estimating the median instead of the mean is asymptotically negligible relative to the variance of $\BMM$. This, of course, comes at the expense of variance reduction, since larger $\alpha$ mean $\BMM$ becomes closer to $\SM$. In any case, they are enough to establish $\med(Y|\hat{\bm{\theta}})$ and $\SM$ are asymptotically equivalent.

\begin{corollary}
If $\hat{\theta}_1, \ldots, \hat{\theta}_n \iid [\theta, \sigma^2]$, $\mathbf{p} \sim \Dir_n(\alpha, \ldots, \alpha)$ and $Y= \sum_{i=1}^{n} p_i \hat{\theta}_i$, then $\med(Y|\hat{\bm{\theta}})-\SM \stackrel{L_2}{\longrightarrow} 0$, and $\med(Y|\hat{\bm{\theta}}) \stackrel{L_2}{\longrightarrow} \theta$. In particular, almost surely, $\med(Y|\hat{\bm{\theta}})$ is asymptotically unbiased.
\end{corollary}

\begin{proof}
 Note from Proposition \ref{prop:bounds_medY_theta} that
\begin{equation*}
 \E\left[(\med(Y|\hat{\bm{\theta}})-\SM)^2\right] \leq \frac{1}{n\alpha+1} \frac{n-1}{n} \sigma^2 \stackrel{n \to \infty}{\longrightarrow}0,
\end{equation*}
so in particular the $L_2$ norm of $\med(Y|\hat{\bm{\theta}})$ and $\SM$ goes to zero as $n \to \infty$, the more so the smaller $\sigma^2$ is. Furthermore, by the triangle inequality,
\begin{equation*}
\sqrt{\E\left[(\med(Y|\hat{\bm{\theta}})-\theta)^2\right]} \leq \sqrt{\E\left[(|\med(Y|\hat{\bm{\theta}})-\SM|)^2\right]} + \sqrt{\E\left[(|\SM-\theta|)^2\right]} \stackrel{n \to \infty}{\longrightarrow} 0.
\end{equation*}
In particular, since $L_2$ convergence implies $L_1$ convergence,
\begin{equation*}
 \E[\med(Y|\hat{\bm{\theta}})-\theta] \leq \E[|\med(Y|\hat{\bm{\theta}})-\theta|] \stackrel{n \to \infty}{\longrightarrow} 0,
\end{equation*}
so $\med(Y|\hat{\bm{\theta}})$ is asymptotically unbiased.
\end{proof}

Note Proposition \ref{prop:bound_BMM_m} bounds the distance between $\BMM$ and $m=\med(Y|\hat{\bm{\theta}})$ and Proposition \ref{prop:bounds_medY_theta} bounds the distance between $m$ and $\theta$. Putting them together, it is possible to upper bound the finite-sample (unconditional) bias of the Bayesian median of means.

\begin{proposition} \label{prop:bias_bound}
Let $\hat{\theta}_1, \ldots, \hat{\theta}_n \iid [\theta, \sigma^2]$, $\mathbf{p}^{(j)} \iid \Dir_n(\alpha, \ldots, \alpha)$ and $Y_j= \sum_{i=1}^{n} p^{(j)}_i \hat{\theta}_i$ for $j=1, \ldots, J$, with $J>2$ and unique $m=\med(Y|\hat{\bm{\theta}})$. Consider $t_0 = 4\sqrt{s^2_{\hat{\bm{\theta}}}/\alpha}$, and assume there exists $C_{t_0}(\hat{\bm{\theta}})>0$ such that $\min_{[m, m+t_0]}f(\xi) > C_{t_0}(\hat{\bm{\theta}})$. Then
\begin{equation*}
|\E[\BMM]-\theta| \leq \sqrt{\frac{\pi}{2J}} \frac{1}{\tilde{C}} + \frac{2}{J-2}\frac{4 n^{-J/4}}{\sqrt{\alpha}} \sqrt{\sigma^2} + \sqrt{\frac{1}{n\alpha+1} \sigma^2},
\end{equation*}
where $\tilde{C}=\sqrt{\E[1/C_{t_0}(\hat{\bm{\theta}})]}$.
\end{proposition}

\begin{proof}
First, decompose the bias as
\begin{equation} \label{eq:decompose_bias}
|\E[\BMM]-\theta| = |\E[\BMM - \SM]| \leq \E[|\BMM - \SM|] \leq \E[|\BMM-m|] + \E[|m-\SM|].
\end{equation}
To bound the first term, use Propositions \ref{prop:bound_BMM_m} and \ref{prop:bound_BMM_m_large} to get
\begin{align*}
\E[|\BMM-m| \mid \hat{\bm{\theta}}] &= \int_0^{t_0} P\left[|\BMM-m|\geq t\right] dt + \int_{t_0}^{\infty} P\left[|\BMM-m|\geq t\right] dt\\
&\leq 2 \int_0^{t_0}e^{-2JC_{t_0}(\hat{\bm{\theta}})t^2} dt + \left(4 \sqrt{\V[Y|\hat{\bm{\theta}}]}\right)^{J/2} \int_{t_0}^{\infty} t^{-J/2} dt \\
&\leq \sqrt{\frac{\pi}{2JC_{t_0}(\hat{\bm{\theta}})}} + \left(4 \sqrt{\V[Y|\hat{\bm{\theta}}]}\right)^{J/2} \frac{2}{J-2} t_0^{1-J/2} \\
&\leq \sqrt{\frac{\pi}{2JC_{t_0}(\hat{\bm{\theta}})}} + \frac{2^{J+1}}{J-2} \left(\sqrt{\frac{1}{n\alpha}s^2_{\hat{\bm{\theta}}}}\right)^{J/2} \left(4 \sqrt{\frac{1}{\alpha} s^2_{\hat{\bm{\theta}}}}\right)^{1-J/2}\\
&= \sqrt{\frac{\pi}{2JC_{t_0}(\hat{\bm{\theta}})}} + \frac{8}{J-2} n^{-J/4} \alpha^{-1/2} \sqrt{s^2_{\hat{\bm{\theta}}}} \\
&= \sqrt{\frac{\pi}{2JC_{t_0}(\hat{\bm{\theta}})}} + \frac{2}{J-2} \frac{1}{n^{J/4}} t_0.
\end{align*}
Thus, by the Law of Iterated Expectations and Jensen,
\begin{equation*}
\E[|\BMM-m|] \leq \sqrt{\frac{\pi}{2J} \E\left[\frac{1}{C_{t_0}(\hat{\bm{\theta}})}\right]} + \frac{8}{J-2} \frac{1}{n^{J/4}} \sqrt{\frac{n-1}{n\alpha}\sigma^2}.
\end{equation*}
For the second term on the left-hand side of (\ref{eq:decompose_bias}), note
\begin{equation*}
\E[|m-\theta|] \leq \sqrt{\frac{1}{n\alpha+1} \frac{n-1}{n} \sigma^2} \leq \sqrt{\frac{1}{n\alpha+1}\sigma^2},
\end{equation*}
which yields the bound.
\end{proof}

Note to have the bias go to zero it is necessary to take both $J, n \to \infty$. Increasing $J$ concentrates $\BMM$ around the conditional median, and increasing $n$ (or $\alpha$) lessens the bias in estimating the median instead of the mean. Indeed, as $\alpha \to \infty$, $\BMM$ converges to $\SM$ and $\med(Y\mid \hat{\bm{\theta}})=\theta$, but at the expense of possibly higher variance.

The decomposition in (\ref{eq:decompose_bias}) also readily gives a bound for the $L_1$ error of the procedure.

\begin{corollary}
Consider the setting of Proposition \ref{prop:bias_bound}. Then, 
\begin{equation*}
\E\left[|\BMM - \theta|\right] \leq \sqrt{\frac{\pi}{2J}} \frac{1}{\tilde{C}} + \frac{2}{J-2}\frac{4 n^{-J/4}}{\sqrt{\alpha}} \sqrt{\sigma^2} + \sqrt{\frac{1}{n\alpha+1} \sigma^2} + \sqrt{\frac{\sigma^2}{n}}.
\end{equation*}
In particular, if $J=O(n)$, then $\BMM$ is consistent.
\end{corollary}

\begin{proof}
 Note that
 \begin{equation*}
 \E\left[|\BMM-\theta|\right] \leq \E\left[|\BMM - m|\right] + \E[|m - \SM|] + \E[|\SM-\theta|].
 \end{equation*}
Proposition \ref{prop:bias_bound} gives a bound on the first two terms, while the second can be controlled via Chebyshev: $\E\left[|\SM-\theta|\right] \leq \sqrt{\E(\SM-\theta)} = \sqrt{\sigma^2/n}$. Consistency follows from $L_1$ convergence.
\end{proof}

The bound on the $L_1$ error above, while holding quite generally, is not very useful, particularly in terms of comparing $\BMM$ with $\SM$, since $\E[|\SM-\theta|]$ is needed to upper bound $\E[|\BMM-\theta|]$. To get a better understanding of when and how $\BMM$ outperforms the sample mean, it is necessary to investigate further asymptotic properties of this estimator using the results of Section \ref{sec:cond_moments_and_density_Y}.

\subsubsection{Asymptotic approximation} \label{sec:asymptotic_approximation}

Proposition \ref{prop:BMM_m_bias_and_variance} showed the variance of the Bayesian median of means is of order $O(1/J)$, and so is the bias in estimating $\med(Y|\hat{\bm{\theta}})$ with $\BMM$. This section develops an asymptotic expansion of $\med(Y|\hat{\bm{\theta}})$ to show the bias in estimating $\theta$ with $\med(Y|\hat{\bm{\theta}})$ is of order $O(1/(n\alpha))$. This implies, when $J=n$ and $\alpha=1$, the variance of $\BMM$ is of order $O(1/n)$ while the squared bias is only $O(1/n^2)$. The asymptotic approximation also gives a deterministic, approximate algorithm for $\BMM$.

\begin{proposition}\label{prop:aBMM}
Let $\hat{\theta}_1, \ldots, \hat{\theta}_n \iid [\theta, \sigma^2]$ with $\E[|\hat{\theta}_i|^3]<\infty$, and denote by $\SM$ the sample mean. If $Y=\sum_{i=1}^{n}p_i\hat{\theta}_i$ for $p\sim\Dir(\alpha, \ldots, \alpha)$, then almost surely with respect to the sampled $\hat{\theta}_1, \ldots, \hat{\theta}_n$,
\begin{equation} \label{eq:aBMM_with_remainder}
\med(Y|\hat{\bm{\theta}}) = \SM - \frac{1}{3}\frac{\sqrt{s^2_{\hat{\bm{\theta}}}}}{n \alpha + 2} \widhat{\skew}(\hat{\bm{\theta}})+ o\left(\frac{1}{n \alpha}\right).
\end{equation}
\end{proposition}

\begin{proof}
Recall from Proposition \ref{prop:bayesian_bootstrap_CLT} that, conditional on $\hat{\bm{\theta}} = (\hat{\theta}_1, \ldots, \hat{\theta}_n)$, 
\begin{equation*}
\sqrt{n}\tilde{Y} = \frac{\sqrt{n}(Y-\SM)}{\sqrt{\frac{1}{n\alpha+1}s^2_{\hat{\bm{\theta}}}}} \Longrightarrow N(0, 1),
\end{equation*}
where $\tilde{Y}$ is the standardized version of $Y$.
Since $\E[\hat{\theta}_i^3]<\infty$, the Edgeworth expansion of \cite{weng1989second} implies that, almost surely with respect to the empirical distribution of $\hat{\theta}_i$,
\begin{equation}\label{eq:bayesian_edgeworth_expansion}
F_{\tilde{Y}|\hat{\bm{\theta}}}(\tilde{y}) = \Phi(\tilde{y}) - \frac{1}{3} \frac{\sqrt{n\alpha+1}}{n\alpha+2} \cdot \widhat{\skew}(\hat{\bm{\theta}}) \cdot (\tilde{y}^2-1) \varphi(\tilde{y}) + o\left(\frac{1}{\sqrt{n\alpha}}\right).
\end{equation}
where $\Phi(\tilde{y})$ and $\varphi(\tilde{y})$ denote the cumulative distribution and probability density functions, respectively, of a standardized Normal random variable, and $\widhat{\skew}(\hat{\bm{\theta}})$ is the sample skewness of $\hat{\theta}_1, \ldots, \hat{\theta}_n$, that is, 
\begin{equation*} 
\widhat{\skew}(\hat{\bm{\theta}}) = \frac{\frac{1}{n}\sum_{i=1}^{n}(\hat{\theta}_i - \SM)^3}{\left(\frac{1}{n} \sum_{i=1}^{n} (\hat{\theta}_i - \SM)^2\right)^{3/2}}.
\end{equation*}
The median of a standardized Normal is zero, so, from Proposition \ref{prop:bounds_medY_theta}, it suffices to consider the median of $\tilde{Y}$ to be $\tilde{m} = O(1/\sqrt{n\alpha})$. In this case, a Taylor expansion yields
\begin{equation*}
\Phi(0) = \Phi(\tilde{m}) + \varphi(\tilde{m}) (0-\tilde{m}) + O\left(\frac{1}{n}\right),
\end{equation*}
and so plugging in $\tilde{y}=\tilde{m}$ in (\ref{eq:bayesian_edgeworth_expansion}) gives
\begin{align*}
\frac{1}{2} &= \Phi(\tilde{m}) - (\tilde{m}^2-1)\varphi(\tilde{m}) \cdot \frac{1}{3} \frac{\sqrt{n\alpha+1}}{n\alpha+2} \cdot\widhat{\skew}(\hat{\bm{\theta}}) + o\left(\frac{1}{\sqrt{n\alpha}}\right)\\
&= \frac{1}{2} + \tilde{m} \varphi(\tilde{m}) + \varphi(\tilde{m}) \frac{1}{3} \frac{\sqrt{n\alpha+1}}{n\alpha+2} \cdot\widhat{\skew}(\hat{\bm{\theta}}) + o\left(\frac{1}{\sqrt{n\alpha}}\right),
\end{align*}
so
\begin{equation*}
\tilde{m} = -  \frac{1}{3} \frac{\sqrt{n\alpha+1}}{n\alpha+2} \cdot\widhat{\skew}(\hat{\bm{\theta}}) + o\left(\frac{1}{\sqrt{n\alpha}}\right).
\end{equation*}
Since $\med(Y|\hat{\bm{\theta}}) = m  = \SM + \sqrt{\frac{1}{n\alpha+1}s^2_{\hat{\bm{\theta}}}} \tilde{m}$, this implies
\begin{equation*}
\med(Y|\hat{\bm{\theta}}) = \SM - \frac{1}{3} \frac{1}{n\alpha+2} \sqrt{s^2_{\hat{\bm{\theta}}}}  \cdot\widhat{\skew}(\hat{\bm{\theta}}) + o\left(\frac{1}{n\alpha}\right),
\end{equation*}
so the conditional median is, to first-order, a variance and skewness correction applied to the sample mean.
\end{proof}

Recall that, as $J \to \infty$, $\med(Y|\hat{\bm{\theta}})$ becomes well-approximated by $\BMM$, and increasing the value of $J$ is generally easy since it only involves sampling Dirichlets. In this sense, the previous proposition illuminates important aspects of the Bayesian median of means estimator. Several key remarks are collected below.

\vspace{5mm}

\noindent
\textbf{Relationsip between $\BMM$ and $\SM$.} Formula (\ref{eq:aBMM_with_remainder}) synthesizes many previous results. For example, it makes it clear that as $\alpha \to \infty$, $\BMM$ is converging to $\SM$ and so, in particular, $\BMM$ is asymptotically unbiased. To first-order, the two estimators differ when the sample variance or the sample skewness are very large, in which case $\BMM$ applies a correction that reduces variance at the expense of bias. In cases where $\SM$ is known to be optimal, for instance when $\hat{\theta}_i\iid N(\theta, \sigma^2)$, $\BMM$ and $\SM$ are virtually the same. Also, while $\SM$ performs worse the higher $\sigma^2$ is, $\BMM$ does not have its asymptotic variance affected by $\sigma^2$, but its bias roughly depends on $\sigma^2 \cdot \skew(\hat{\bm{\theta}})$. Thus, the ideal setting for $\BMM$ relative to $\SM$ is one with high variance, low skewness --- precisely those of Example \ref{example:ARE}.

\vspace{5mm}

\noindent
\textbf{How $\alpha$ affects convergence.} The values of $n$ and $\alpha$ are directly related. If $\alpha$ is much smaller than $1$, the bias of the estimator can be relatively large, even asymptotically. For instance, if $\alpha=1/n$ the Bayesian median of means ceases to be asymptotically unbiased, as there is a constant correction to the sample mean, possibly incurring high variance. Values of $\alpha$ much larger than 1 suggest more data is available than is actually the case, and thus the estimator relies more on the sample mean. Proposition \ref{prop:bayesian_bootstrap_CLT} also holds when $\alpha \to \infty$, since a $\text{Gamma}(\alpha)$ can be thought of as the convolution of $\alpha$ $\text{Exponential}$ random variables, in which case the Central Limit Theorem yields $\sqrt{\alpha}(Y-\SM) \mid \hat{\bm{\theta}} \Rightarrow N(0, s^2_{\hat{\bm{\theta}}}/n)$, so Proposition \ref{prop:aBMM} still carries through for $\alpha \to \infty$ with fixed $n$.

Example \ref{example:compare_BMM_aBMM} compares the behavior of $\BMM$ and $\SM$ for varying values of $n$ and $\alpha$ for a Skewnormal simulation. The larger $n$ or $\alpha$, the more the estimators look alike.

\vspace{5mm}

\noindent
\textbf{Cornish-Fisher expansion.} The proof of Proposition \ref{prop:aBMM} can be thought of as the development of a Cornish-Fisher expansion for the median of $Y|\hat{\bm{\theta}}$, which provides a full asymptotic expansion of the quantiles of $F_{Y|\hat{\bm{\theta}}}(y)$ in terms of its cumulants. For example, using (\ref{eq:third_moment}), the skewness of $Y|\hat{\bm{\theta}}$ can be written
\begin{align*}
 \widhat{\skew}(Y|\hat{\bm{\theta}}) &= \frac{\E[Y^3|\hat{\bm{\theta}}] - 3 \E[Y|\hat{\bm{\theta}}]\V[Y|\hat{\bm{\theta}}] - (\E[Y|\hat{\bm{\theta}}])^3}{(\V[Y|\hat{\bm{\theta}}])^{3/2}} \\
 &= \frac{\frac{2}{(n\alpha+2)(n\alpha+1)}\left[\overline{\theta^3} - 3 \SM \left(\overline{\theta^2} - \left(\SM\right)^2\right)^2 - \SM^3\right]}{\left(\frac{1}{n\alpha+1}s^2_{\hat{\bm{\theta}}}\right)^{3/2}} \\
 &= \frac{2 (n\alpha+1)^{1/2}}{n\alpha+2} \cdot \widhat{\skew}(\hat{\bm{\theta}}),
\end{align*}
which appears as part of the coefficient of the first factor in (\ref{eq:bayesian_edgeworth_expansion}). Section \ref{sec:cond_moments_and_density_Y} established the density of $Y|\hat{\bm{\theta}}$ and a recursive formula for higher moments, which is what is needed to develop further terms and establish better approximations. In particular, a formal Cornish-Fisher expansion implies the remainder in (\ref{eq:aBMM_with_remainder}) is actually $O(\frac{1}{n^2 \alpha^2})$.

\vspace{5mm}

\noindent
\textbf{Bias of $\BMM$ is asymptotically negligible.} The motivation for the Bayesian median of means is to improve on the mean squared error of the sample mean by trading off some variance for bias. For this approach to work, it is paramount that the squared bias introduced is not larger than the reduction in variance. The following corollary provides an asymptotic assurance.

\begin{corollary}
Let $\hat{\theta}_1, \ldots, \hat{\theta}_n \iid [\theta, \sigma^2]$, with $\E[|\hat{\theta}_i|^3]<\infty$, and consider $Y_j = \sum_{i=1}^{n} p^{(j)}_i \hat{\theta}_i$, where $p^{(j)}\iid \Dir(\alpha, \ldots, \alpha)$, $j=1, \ldots, J$. Take and $J = n$ and $\alpha=1$. If $\med(Y|\hat{\bm{\theta}})$ is unique, then almost surely
\begin{equation*}
\E[\BMM] = \theta + O\left(\frac{1}{n}\right), \qquad \V[\BMM] = O\left(\frac{1}{n}\right).
\end{equation*}
\end{corollary}

\begin{proof}
 If $\alpha=1$ and $J = O(n)$, then Proposition \ref{prop:aBMM} implies the bias between $\med(Y|\hat{\bm{\theta}})$ and $\theta$ is of order $O\left(\frac{1}{n}\right)$, while Proposition \ref{prop:BMM_m_bias_and_variance} imply the bias between $\BMM$ and $\med(Y|\hat{\bm{\theta}})$ is of order $O\left(\frac{1}{n}\right)$, so the overall bias of $\BMM$ is $O(\frac{1}{n})$. From Proposition \ref{prop:BMM_m_bias_and_variance}, the variance is of order $O\left(\frac{1}{n}\right)$.
\end{proof}

The crucial implication of the above corollary is that, in terms of mean squared error, the squared bias incurred by the Bayesian median of means is of order $O(\frac{1}{n^2})$, and so it is negligible with respect to the variance, which is of order $O(\frac{1}{n})$. Put another way, for $n$ large enough, the expected decrease in variance of the Bayesian median of means is sure to make its mean squared error smaller than that of the sample mean. Determining how large $n$ should be, however, for the asymptotic regime to be sufficiently accurate depends on the underlying distribution of $\hat{\theta}_i$; in particular, note the MLE generally exhibits a similar behavior of trading-off some bias for variance, but it does so by requiring the distribution to be specified.

\vspace{5mm}

\noindent
\textbf{Deterministic approximation to $\BMM$.} Proposition \ref{prop:aBMM} also suggests a new, robust mean estimator when $\E[|\hat{\theta}_i|^3]<\infty$: the \emph{approximate Bayesian median of means}, defined as
\begin{equation} \label{eq:aBMM}
\aBMM = \SM - \frac{1}{3} \frac{\sqrt{s^2_{\hat{\bm{\theta}}}}}{n\alpha+2} \widhat{\skew}(\hat{\bm{\theta}}).
\end{equation}
Note $\aBMM$ is deterministic and simpler to compute than the original $\BMM$, since it doesn't require the sampling of auxiliary Dirichlet random variables. It only has one hyperparameter, namely $\alpha$, and Section \ref{sec:choosing_alpha} considers how to set it. Example \ref{example:compare_BMM_aBMM} shows that $\aBMM$ is accurate in replicating the performance of $\BMM$ for a Skewnormal example; further comparisons are provided in Section \ref{sec:comparison_with_sample_mean_and_median_of_means}. Better approximations are possible by using additional terms of the Cornish-Fisher expansion, but at the expense of more computation and assumptions on the distribution of $\hat{\theta}_1, \ldots, \hat{\theta}_n$.

\begin{example}[Approximate BMM] \label{example:compare_BMM_aBMM} Let $\hat{\theta}_1, \ldots, \hat{\theta}_n$ be a sample of $n=1000$ iid samples from a Skewnormal distribution, with location $\xi=0$, scale $\omega=20$ and shape $\beta=10$. The sample mean is $\SM \approx 15.625$ and the true mean is $\theta \approx 15.878$ (with variance $147.87$ and skewness $0.956$). 
	
Figure \ref{fig:aBMM} illustrates the result of applying both the approximate and the exact Bayesian median of means procedures to such data, with $J=n=1000$ and varying values of $\alpha$. The violin plot in the figure, drawn from $100$ simulations, shows the distribution of $\BMM$ is approximately a scaled Normal, in accordance with Proposition $\ref{prop:medianCLT}$, and that $\aBMM$ and $\BMM$ agree with each other the bigger $\alpha$ is.

Figure \ref{fig:aBMM_change_n} considers the same setting, but with fixed $\alpha=1$ and varying values of $n$. Note the sampled $\hat{\theta}_1, \ldots, \hat{\theta}_n$ change with $n$, which is why there are variations in each case. Overall, it is clear that $\aBMM$ provides a decent approximation to $\BMM$, and they are generally similar to $\SM$, especially as $n$ increases, in which case all estimators approach the value of $\theta$. For these simulations, $J$ is kept fixed at $1000$.
\end{example}

\begin{figure}[htbp]
\begin{center}
{\fontfamily{cmss}\selectfont \textbf{$\aBMM$ vs $\BMM$, different $\alpha$}}
\includegraphics[width=\textwidth]{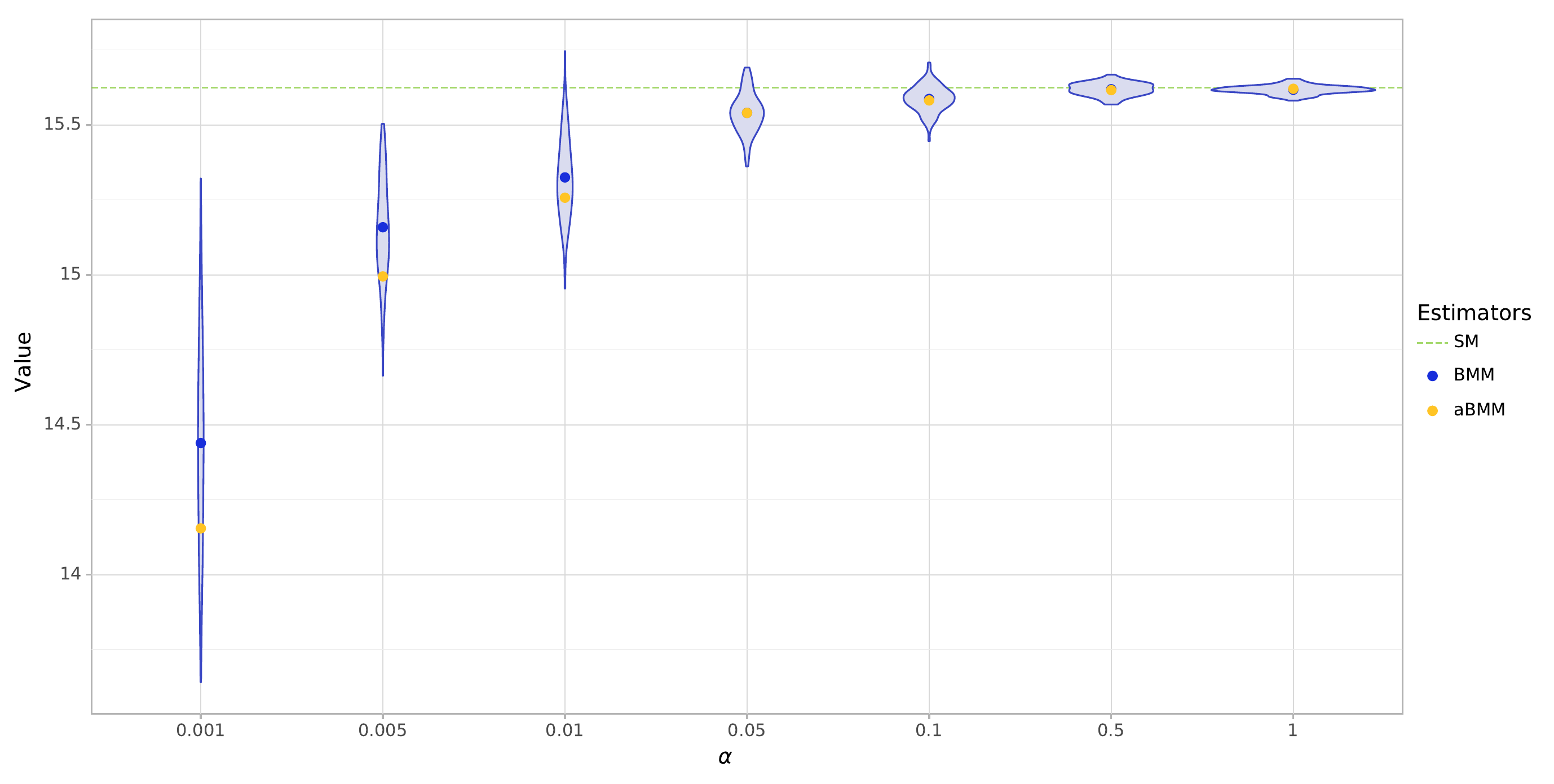}
\caption{Violin plot for $\BMM$ for different values of $\alpha$, for fixed $\hat{\theta}_1, \ldots, \hat{\theta}_{1000}$ sampled iid from a Skewnormal distribution, and 100 simulations. The blue dot is the mean of the distribution; the yellow dot is the deterministic approximation to $\BMM$, and the green line is the sample mean.}
\label{fig:aBMM}

\vspace{15mm}

{\fontfamily{cmss}\selectfont \textbf{$\aBMM$ vs $\BMM$, different $n$}}
\includegraphics[width=\textwidth]{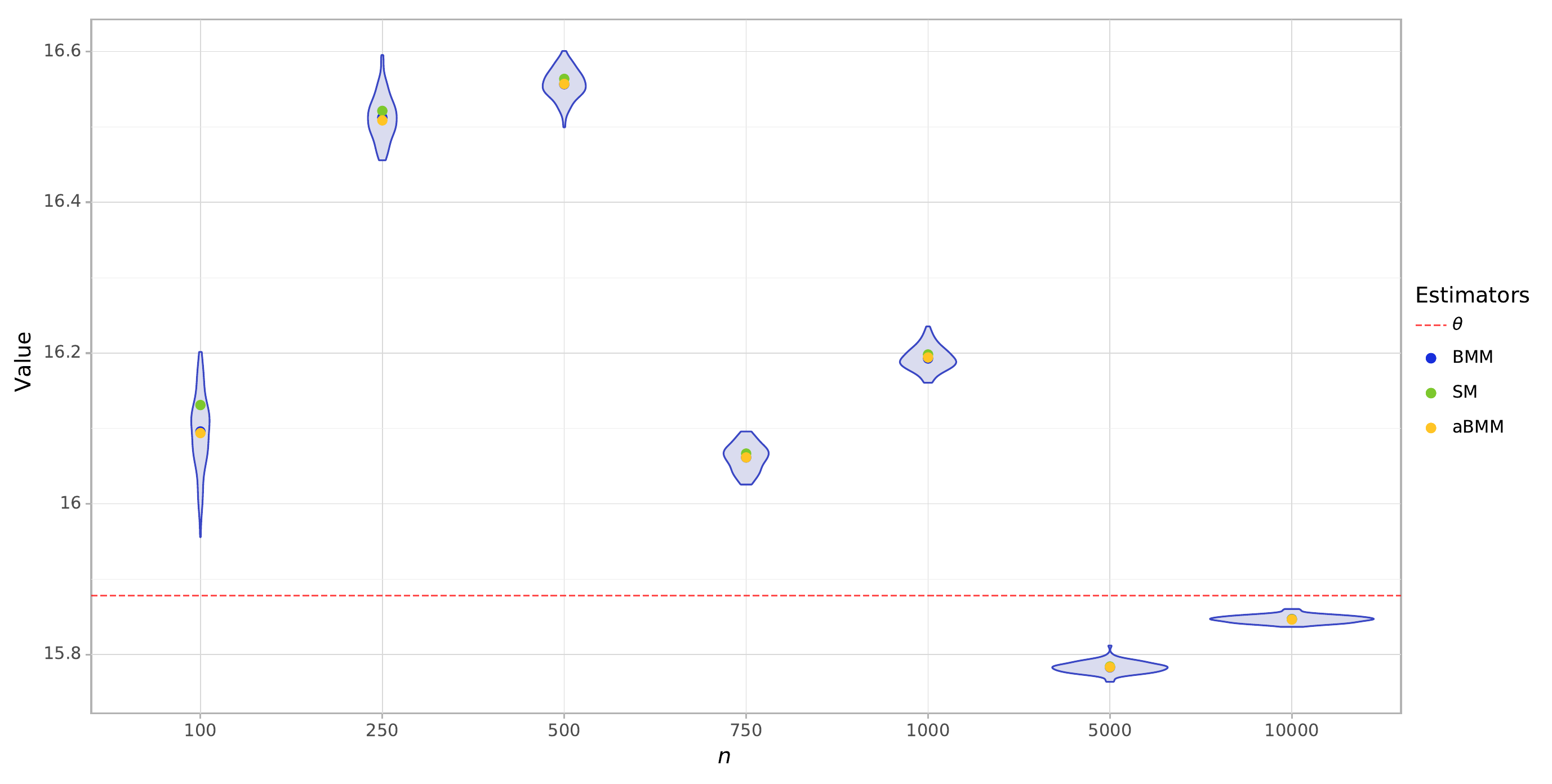}
\caption{Violin plot for $\BMM$ for different values of $n$, for fixed $\hat{\theta}_1, \ldots, \hat{\theta}_n$ sampled iid from a Skewnormal distribution, and 100 simulations. The yellow dot is the deterministic approximation to $\BMM$, the blue dot is the mean of the $\BMM$ distribution; the green dot is the sample mean; and the red line is the true value of the parameter.}
\label{fig:aBMM_change_n}
 \end{center}
\end{figure}

\subsection{Choosing $\alpha$} \label{sec:choosing_alpha}

The Bayesian median of means algorithm is fully specified, except for the choice of hyperparameter $\alpha$. This section discussesr how to set $\alpha$, from picking it independently of the data to more data-driven choices. For its computational simplicity and statistical properties, $\alpha=1$ is taken to be the recommended value.

\subsubsection{Setting $\alpha=1$}

Recall large values of $\alpha$ approximate $\BMM$ to $\SM$, so they induce less bias but more variance. Ideally, one would set $\alpha$ large enough so as to have minimum bias while still keeping the variance in control. From this regard, recall $\BMM = \widhat{\med}(Y_1, \ldots, Y_J)$, where
\begin{align*}
\V[Y_1] &= \E[\V[Y_1|\hat{\bm{\theta}}]] + \V[\E[Y_1|\hat{\bm{\theta}}]] = \E\left[\frac{1}{n\alpha+1}s^2_{\hat{\bm{\theta}}}\right]+\V[\SM]\\
&= \frac{1}{n\alpha+1} \frac{n-1}{n} \sigma^2 + \frac{\sigma^2}{n}.
\end{align*}
The second term above is not affected by $\alpha$, so one could pick $\alpha$ to have the first term of the same asymptotic order of the second. This amounts to setting $\alpha=1-2/n$. For sufficiently large $n$, it makes sense to simply set $\alpha=1$.

Computationally, the choice $\alpha=1$ is also advantageous, since in this case sampling $(p_1, \ldots, p_n) \sim \Dir_n(1, \ldots, 1)$ can be done by sampling $n-1$ Uniform random variables, $U_1, \ldots, U_{n-1}$, and ordering them. Then, set $U_0=0$, $U_{n}=1$ and take $p_{i}=U_{(i)}-U_{(i-1)}$ for $i=1, \ldots, n$. For extremely large values of $n$, sorting is more expensive than sampling Gammas, the usual way of obtaining a Dirichlet draw; still, for the case $\alpha=1$, Exponentials can be sampled instead of Gammas, which are much faster.

Other canonical, data-independent choices for $\alpha$ are $\alpha=4$, so the first two terms of the Edgeworth expansion (\ref{eq:bayesian_edgeworth_expansion}) match that of the regular bootstrap, and $\alpha=0.5$, which amounts to the half-sampled bootstrap scheme considered in \cite{friedman2007bagging}.

\subsubsection{Picking $\alpha$ via prior information}

One of the advantages of the Bayesian median of means is that prior information can be easily incorporated. If the user is confident the underlying distribution is close to Normal, a higher $\alpha$ should be set; if they expect a distribution with high or infinite variance, such as the importance sampling estimator in Section \ref{sec:introduction}, then lower values of $\alpha$ are better. In particular, $\alpha$ can be attributed its own prior, say $\alpha \sim \text{Beta}(\beta_1, \beta_2)$, and have its parameters determined via an Empirical Bayes approach. More generally, the prior might depend on sample quantities such as standard deviation or skewness.

Prior information can also help set $\alpha$ through asymptotic considerations. Recall the bias of $\BMM$ is of order $O(1/J)+O(1/n\alpha)$, so, for instance, $\alpha=1/\sqrt{n}$ and $J=n\alpha$ imply the bias squared is of the same order as the variance, which might be desirable if it is known beforehand that variance is a bigger concern than bias. 

\subsubsection{Picking $\alpha$ adaptively}

It is also possible to pick $\alpha$ depending on the sampled values $\hat{\theta}_1, \ldots, \hat{\theta}_n$ by using cross-validation. That is, for a given choice of $\alpha$, split $\hat{\theta}_1, \ldots, \hat{\theta}_n$ into $k$ folds, and use $\BMM$ with $k-1$ folds worth of data to estimate, say, the sample mean of the unseen fold as a proxy for $\theta$. Average the errors over the folds to obtain an error estimate, and pick the $\alpha$ that yields lowest error. Similarly, if $\alpha$ is set much smaller than $1$, most samples $\mathbf{p}^{(1)}, \ldots, \mathbf{p}^{(J)}$ will contain coordinates very close to zero. Setting them to zero amounts to not using some of the $\hat{\theta}_i$ in creating $Y_j$, so these $\hat{\theta}_i$ can be thought of as out-of-bag samples, and the $\alpha$ that best predicts the $\hat{\theta}_i$ not used is selected. Note, however, that the number of folds now becomes another hyperparameter to be determined. Furthermore, if the sample is highly skewed or with large variance then cross-validation is expected to fail, since there might be severe mismatches between the folds.

\section{Empirical Results} \label{sec:empirical_results}

This section considers the empirical behavior of the Bayesian median of means in a variety of settings. The results will generally be compared against the sample mean, a standard nonparametric location estimator, using mean squared error loss. By default, the Bayesian median of means will use $\alpha=1$ and $J=n$. The full procedure is given in Algorithm \ref{alg:bmm} below. Note this is readily parallelizable. All the code to generate the figures and examples can be found at \url{https://github.com/paulo-o/bmm}.

The examples below range from low to high-variance distributions, coming from real and simulated data. Section \ref{sec:comparison_with_sample_mean_and_median_of_means} compares the behavior of the Bayesian median of means, $\BMM$, against four other candidates: the sample mean, $\SM$; the approximate Bayesian median of means, $\aBMM$; the classical median of means, $\MM$; and the sample median, $\widhat{\med}(\hat{\bm{\theta}})$. Section \ref{sec:confidence intervals} considers the issue of developing confidence intervals for $\BMM$ in an efficient manner, while Section \ref{sec:applications} applies $\BMM$ to procedures such as importance sampling, cross-validation and bagging. Note in some of these cases the estimators $\hat{\theta}_i$ are no longer independent, as was assumed throughout the paper, but the conditional results still hold, and it is interesting to see how the Bayesian median of means fares in this context.

\vspace{5mm}

\begin{algorithm}
\caption{Bayesian Median of Means}\label{alg:bmm}
\begin{algorithmic}[1]
\Procedure{BMM}{$\{\hat{\theta}_i\}_{i=1}^n$, $J=n$, $\alpha=1$}
\For{$j=1, \ldots, J$}
\State draw $\mathbf{p}^{(j)} \sim \Dir_n(\alpha, \ldots, \alpha)$
\State set $Y_j = \sum_{i=1}^{n} p^{(j)}_i \hat{\theta}_i$
\EndFor
\State \textbf{return} $\BMM = \widhat{\med}(Y_1, \ldots, Y_J)$
\EndProcedure
\end{algorithmic}
\end{algorithm}

\subsection{Comparison with sample mean and median of means}
\label{sec:comparison_with_sample_mean_and_median_of_means}

Recall both $\BMM$ and $\MM$ can be thought of as interpolating between the sample mean and the sample median. The simulations below explore how $\BMM$, $\aBMM$ and $\MM$ fare in settings that favor these extremes. Since the the Bayesian median of mean has its hyperparameter set at $\alpha=1$, a fair comparison would have the median of means with as few groups $g$ as possible. Since $g=1$ and $g=2$ lead to the sample mean, $g=3$ is used in the examples below.

\begin{example}[Skewnormal]
Suppose $\hat{\theta}_i \iid \text{SkewN}(\xi, \sigma, \lambda)$, with location parameter $\xi=0$, scale $\sigma=1000$ and shape $\lambda=0$ (so this is a Normal distribution). The boxplot of $\SM$, $\aBMM$, $\BMM$, $\MM$ and $\widhat{\med}(\hat{\theta}_1, \ldots, \hat{\theta}_n)$, evaluated over $1000$ simulations, is shown on the left side of Figure \ref{fig:compare_estimators_skewnormal}. The numbers in the figure correspond to the mean squared error of each estimator. As expected, all estimators are unbiased, but the median has greater variance, contributing to worse performance. The median of means, being closer to the median than the Bayesian median of means, shows slightly degraded results. Both the Bayesian median of means and its approximation, as well as the sample mean, exhibit similar performance.

On the other hand, consider increasing the skewness of the distribution to $\lambda=40$. Now, the underlying distribution of the $\hat{\theta}_i$ is no longer symmetric, so the median becomes severely biased, as shown on the right side of Figure \ref{fig:compare_estimators_skewnormal}. Note, however, that $\BMM$ and $\aBMM$ incur minimal bias and still have a similar performance to $\SM$.
\end{example}

\begin{figure}[htbp]
\centering {\fontfamily{cmss}\selectfont\text{\textbf{Boxplots for Skewnormal example}}}
 \begin{center} \includegraphics[width=.49\textwidth]{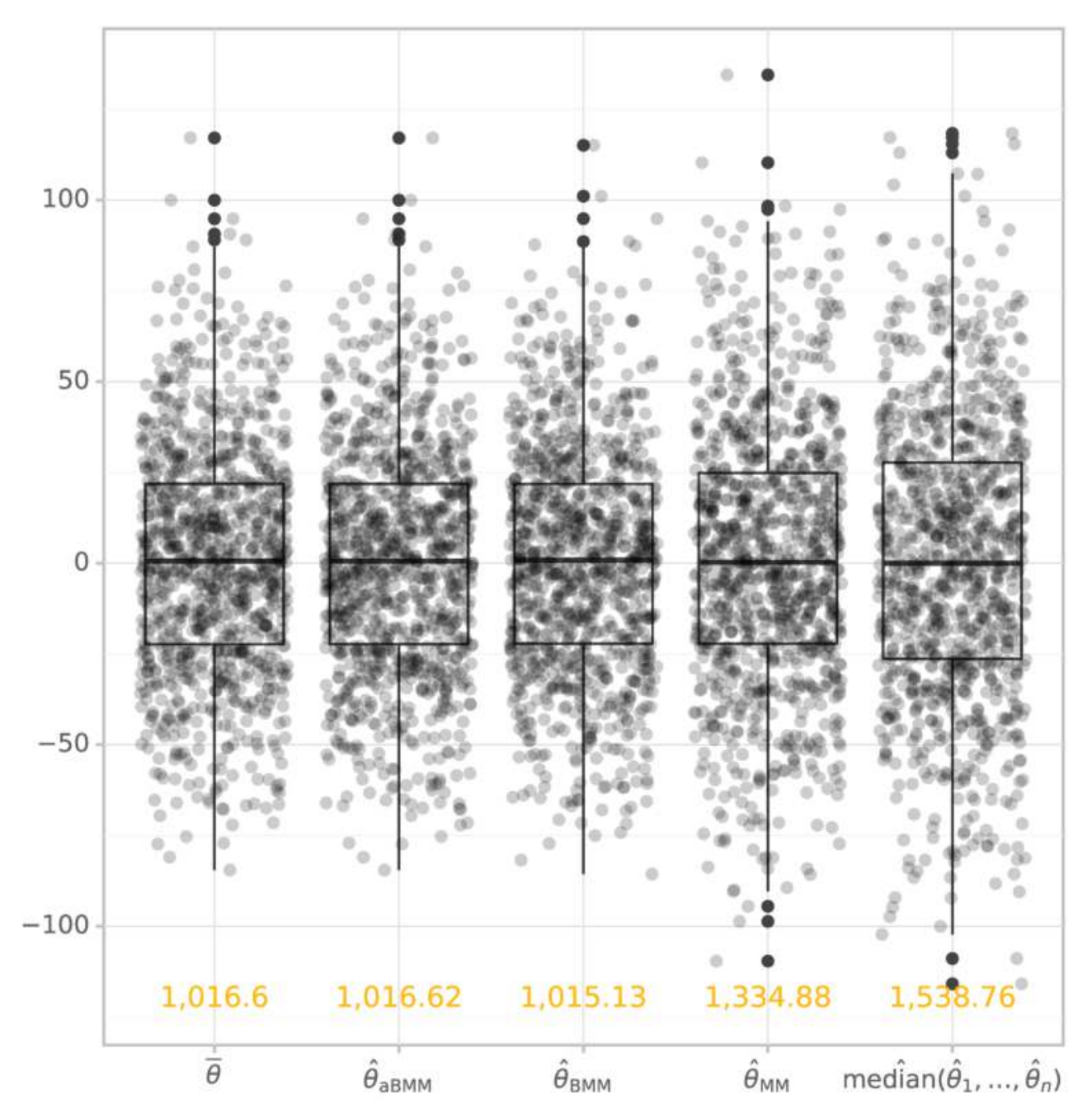}
\includegraphics[width=.49\textwidth]{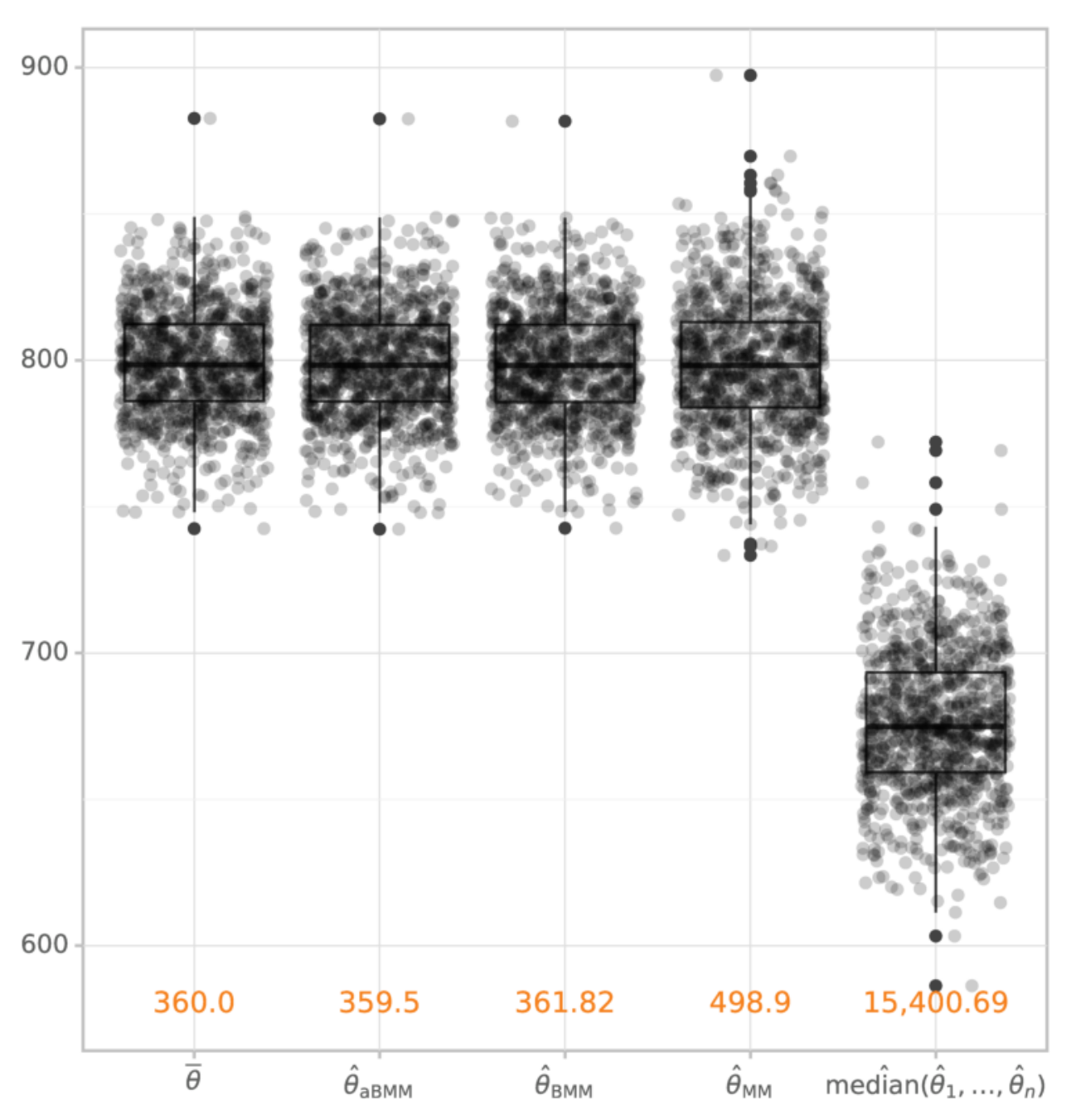}
\caption{Boxplot for different estimators in the Skewnormal example. The left figure has zero skew; the right figure has high skew. The numbers indicate the mean squared error of each estimator.}
\label{fig:compare_estimators_skewnormal}
 \end{center}
\end{figure}

\begin{example}[Pareto]
Let $\hat{\theta}_i \iid \text{Pareto}(\xi, \sigma, \lambda)$ with location parameter $\xi=0$, scale $\sigma=1000$ and two possible shapes $\lambda=4$ (lower skewness) and $\lambda = 2.5$ (higher skewness). In both cases, displayed in Figure \ref{fig:compare_estimators_pareto}, there is enough bias in using the median that it performs far worse than any other estimator. The sample mean works well when the skewness is small, but it is still comparable to both $\BMM$ and $\aBMM$. Once the skewness and variance increase, the sample mean performs worse than either, particularly because it gives weights to extreme sample points. The median of means exhibits slightly worse performance than both the exact and approximate Bayesian median of means.
\end{example}

\begin{figure}[htbp]
\centering {\fontfamily{cmss}\selectfont\text{\textbf{Boxplots for Pareto example}}}
 \begin{center} \includegraphics[width=.49\textwidth]{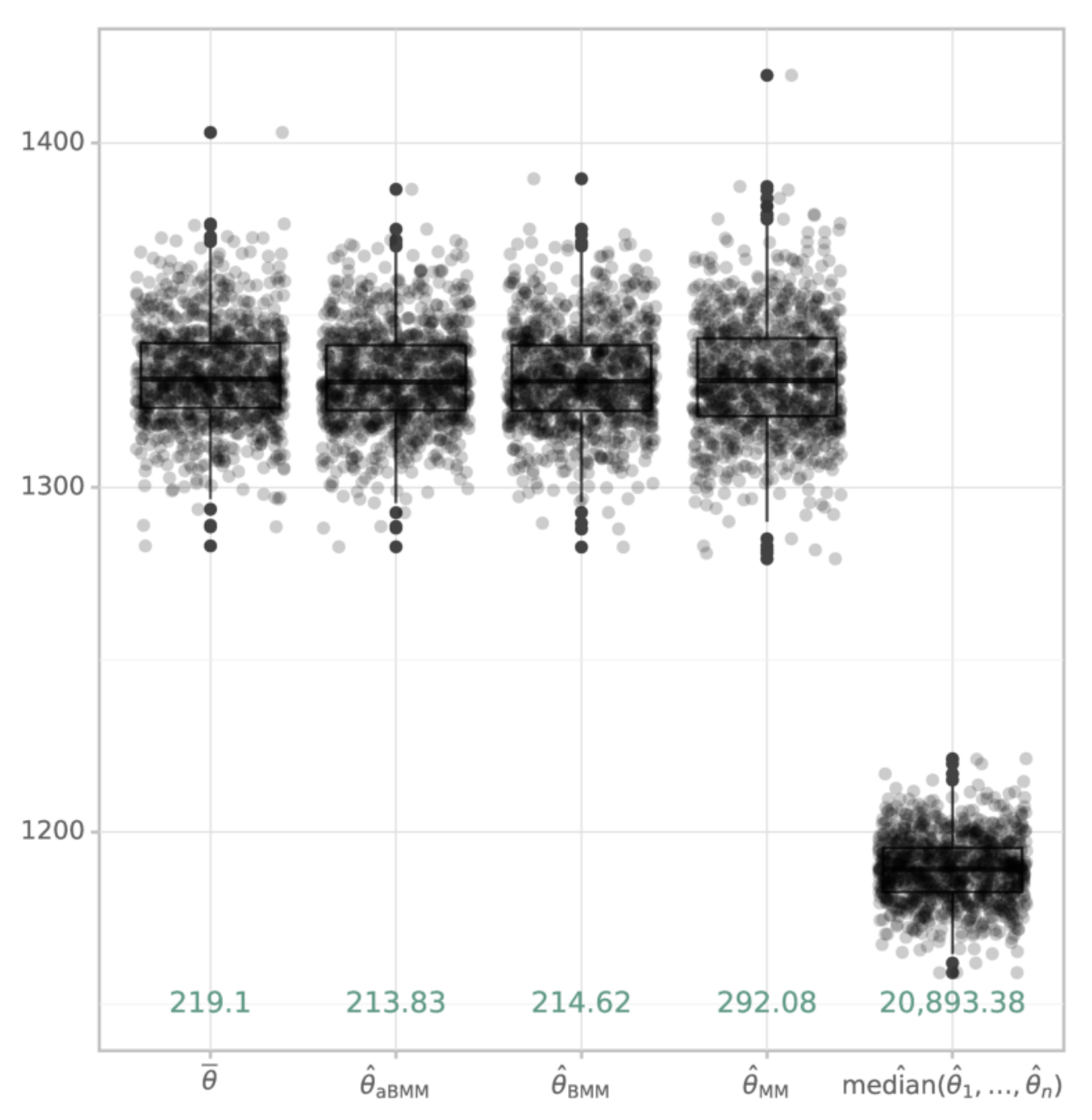}
\includegraphics[width=.49\textwidth]{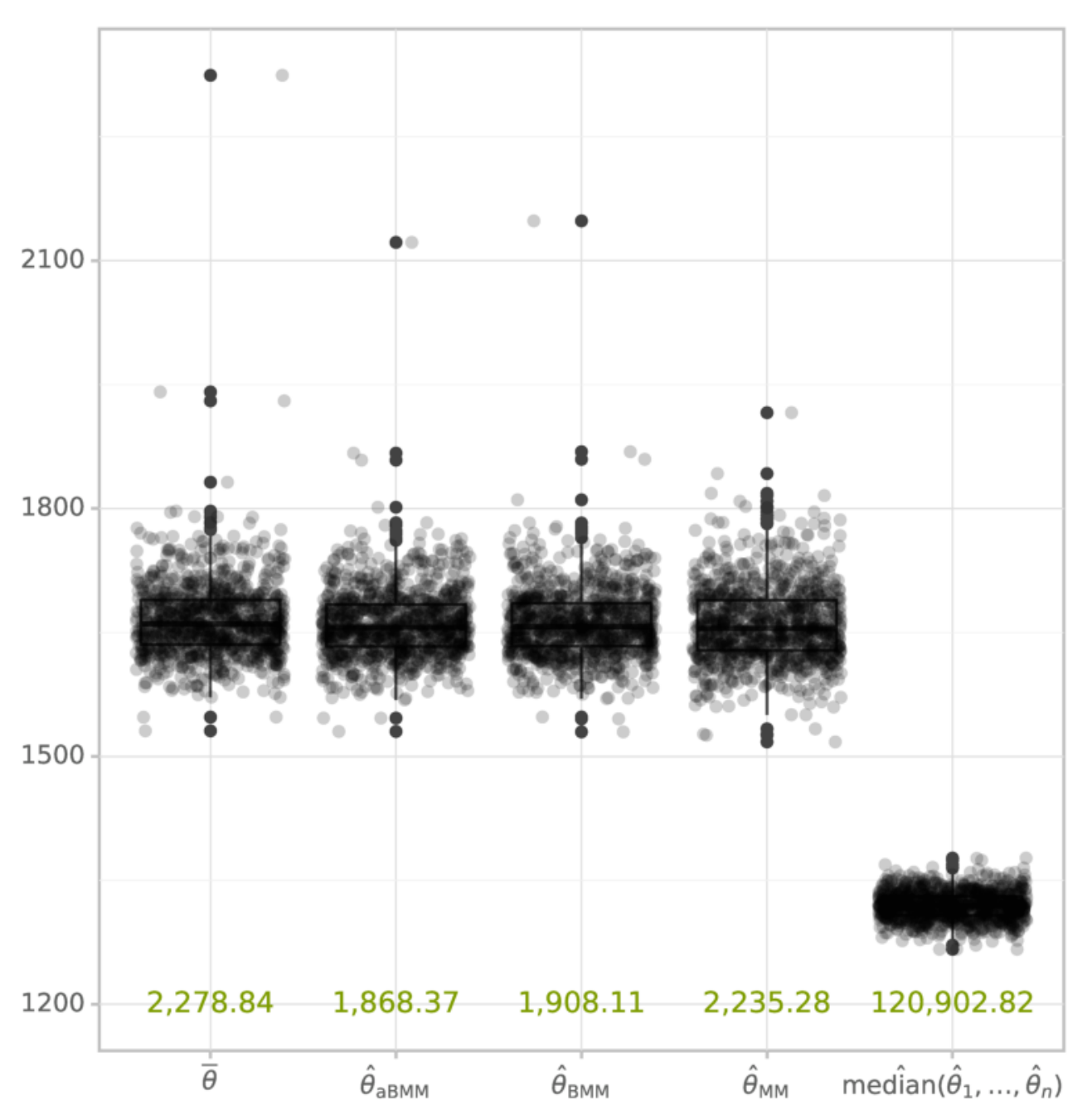}
\caption{Boxplot for different estimators in the Pareto example. The left figure has low skewness; the right figure is higher skewed. The numbers indicate the mean squared error of each estimator.}
\label{fig:compare_estimators_pareto}
 \end{center}
\end{figure}

\begin{example}[Discrete]
Assume now a distribution that is particularly favorable to the sample median relative to the sample mean: suppose $\hat{\theta}_i$ is distributed as
\begin{equation*}
\begin{cases} 
	n^2 \sigma &\mbox{with probability } \frac{1}{2n^p} \\ 
	0		   &\mbox{with probability } 1-\frac{1}{n^p} \\ 
	-n^2 \sigma &\mbox{with probability } \frac{1}{2n^p}. 
\end{cases}
\end{equation*}
Note $\E[\hat{\theta}_i]=0$ and the distribution is symmetric, so the median is unbiased. On the other hand, the variance is $\V[\hat{\theta}_i]=n^{4-p} \sigma^2$. For almost all samples drawn $\hat{\theta}_i=0$, which is the true value of the parameter. Rarely, however, $\hat{\theta}_i$ attains a large value that is capable of throwing the sample mean off, while the sample median remains immune. 

Figure \ref{fig:compare_estimators_discrete} shows the result of the different estimators when $n=1000$, $\sigma=30$ and $p=1$ (high probability of extremes) or $p=1.5$ (lower). Note in both cases the sample median attains zero mean squared error, while the sample mean behaves poorly. When $p$ is bigger, the extreme values seldom survive the median operation, so the median of means achieves good performance, and less so the Bayesian median of means, as it doesn't give weight zero to any single $\hat{\theta}_i$. When the probability of extreme events increase, the median of means becomes relatively more susceptible to extreme measurements, and displays worse results. Also, because the underlying distribution is symmetric, $\widhat{\skew}(\hat{\bm{\theta}})\approx0$, so $\aBMM$ is virtually the same as the sample mean. In this case, the approximation fails, as the neglected higher-order terms in the expansion become consequential.
\end{example}

\begin{figure}[htbp]
\centering {\fontfamily{cmss}\selectfont\text{\textbf{Boxplots for discrete example}}}
 \begin{center} \includegraphics[width=.49\textwidth]{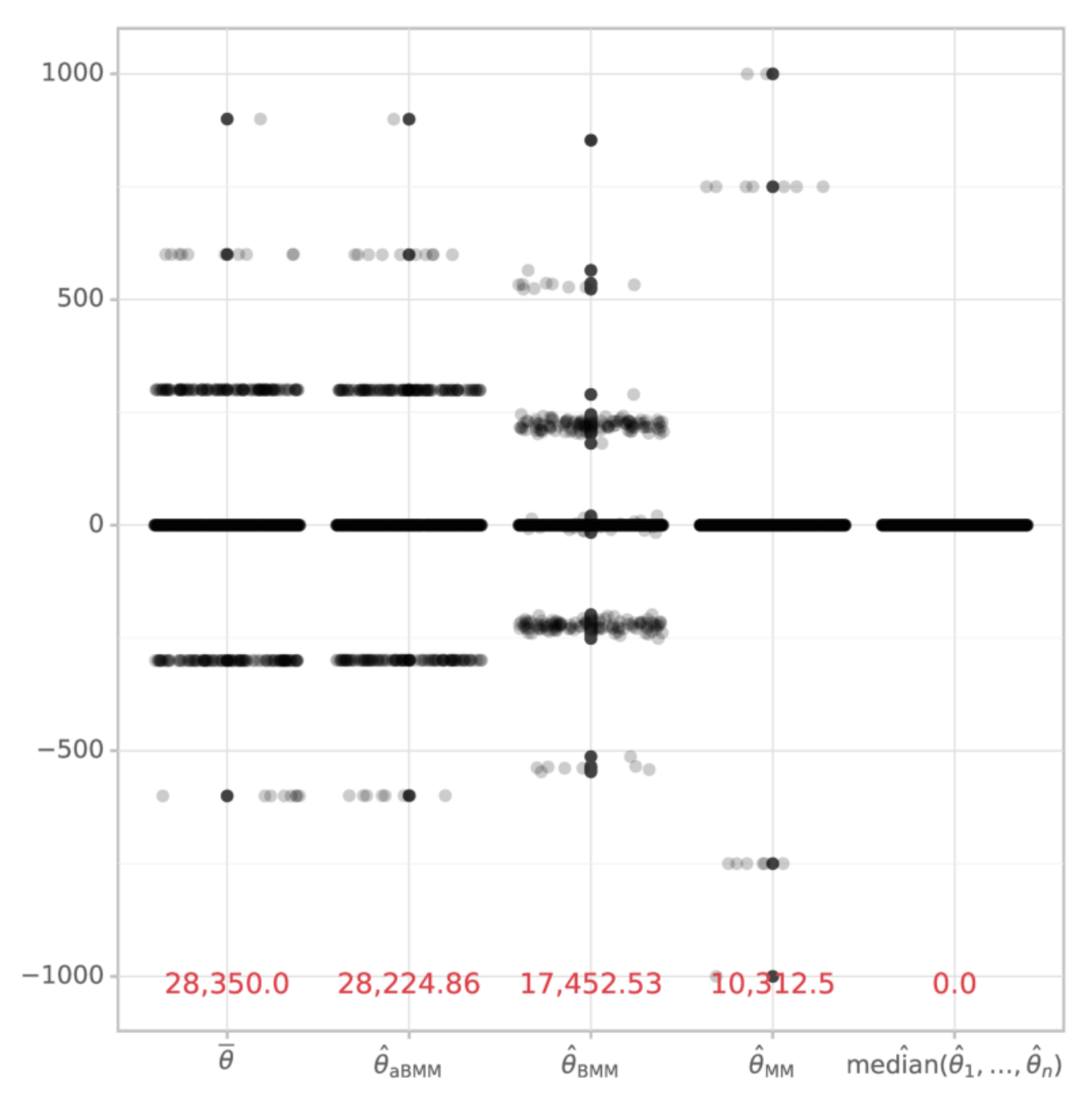}
\includegraphics[width=.49\textwidth]{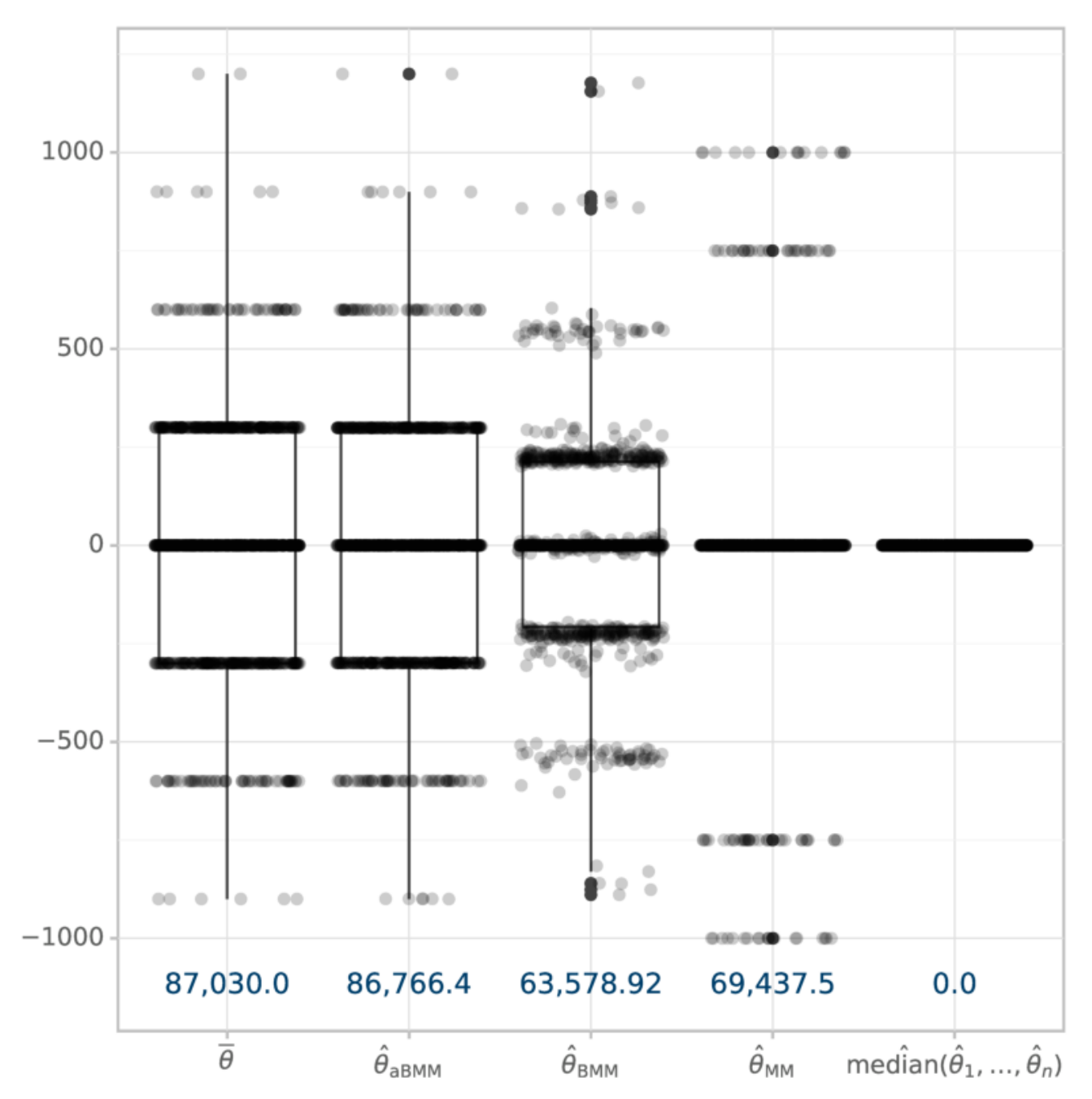}
\caption{Boxplot for different estimators in the discrete example. The left figure has low probability of extreme events; the right figure's probability is higher. The numbers indicate the mean squared error of each estimator.}
\label{fig:compare_estimators_discrete}
 \end{center}
\end{figure}

The behavior of the estimators in this last example can be understood as follows. While the median of means is much better protected against extremes because of the median operation, it is less efficient in the sense that it underutilizes the data, as the final estimator has only used $n/g$ datapoints, where $g$ is the number of groups. The Bayesian median of means is somewhere in between the two. Indeed, write
\begin{align*}
 \SM = \widhat{\med} (\SM, \ldots, \SM), \qquad
 \BMM = \widhat{\med} (Y_1, \ldots, Y_J), \qquad
 \MM = \widhat{\med}(\overline{\theta}_1, \ldots, \overline{\theta}_g).
\end{align*}
In each case, $\E[\SM]=\E[Y_1]=\E[\overline{\theta}_1]$, but the variances are very different: 
\begin{equation*}
 \V[\SM] = \frac{\sigma^2}{n}, \qquad \V[Y_1] = \frac{\sigma^2}{n} \frac{n(\alpha+1)}{n\alpha+1}, \qquad \V[\overline{\theta}_1] = \frac{\sigma^2g}{n}.
\end{equation*}
It is clear $\V[Y_1] \geq \V[\SM]$, while $\V[\overline{\theta}_1]\geq\V[Y_1]$ if $n\alpha \geq (\alpha+1)n/g$. With $\alpha=1$, this is always the case (since $g\geq2$), and the inequality becomes more pronounced as $\alpha \to \infty$, that is, as $\BMM \to \SM$.

However, while the terms in $\BMM$ have less variance than those of $\MM$, the Bayesian median of means is less immune to extremes. Indeed, the blocks $\overline{\theta}_i$ are completely independent of each other, while the $Y_j$ are dependent, as they all rely on the same values of $\hat{\theta}_1, \ldots, \hat{\theta}_n$. Still, this dependence is weaker than that of the sample mean, where all terms are equal to $\overline{\theta}$. In particular, the sample mean has an asymptotic breakdown point of zero, so it suffices to have one outlier to severely alter it, while the median of means has an asymptotic breakdown point equal or bigger to $g/2n$, since at least $g/2$ outliers are necessary, but not sufficient, to damage it. The Bayesian median of means sits somewhere in-between as $\alpha$ varies, though the notion of breakdown point is harder to quantify in this case.

\vspace{5mm}

Consider now how the estimators behave as $n$ changes. Recall that as $n \to \infty$, both $\BMM$ and $\SM$ should become indistinguishable. In the examples below, $n=0, 10, 20, 30, \ldots, 2000$, and $100$ simulations are used for each $n$ to obtain the mean squared error. As before, the median of means estimators uses $g=3$ groups. All the figures below have the $y$-axis in $\log_{10}$ scale.

\begin{example}[Exponential-$t$ distribution] Let $\hat{\theta}_1, \ldots, \hat{\theta}_n \iid  \text{Expo}(\lambda)+t_{2.5}(0, \sigma)$, where $\sigma=\frac{1}{\lambda}=30$, so the $\hat{\theta}_i$ are not symmetric and have a distribution with relatively heavy tails. Figure \ref{fig:exponential_t} shows the decaying MSE for each estimator, except the median. Note the approximate and exact Bayesian median of means have virtually the same performance, while the sample mean exhibits consistently bigger MSE. The median of means displays the worst performance in this scenario because it relies too much on the median, which has MSE close to $17$ for large $n$ due to the bias.
\end{example}

\begin{figure}[htbp]
 \begin{center}
{\fontfamily{cmss}\selectfont\text{\textbf{MSE for Exponential-$t$ example}}}  \includegraphics[width=1\textwidth]{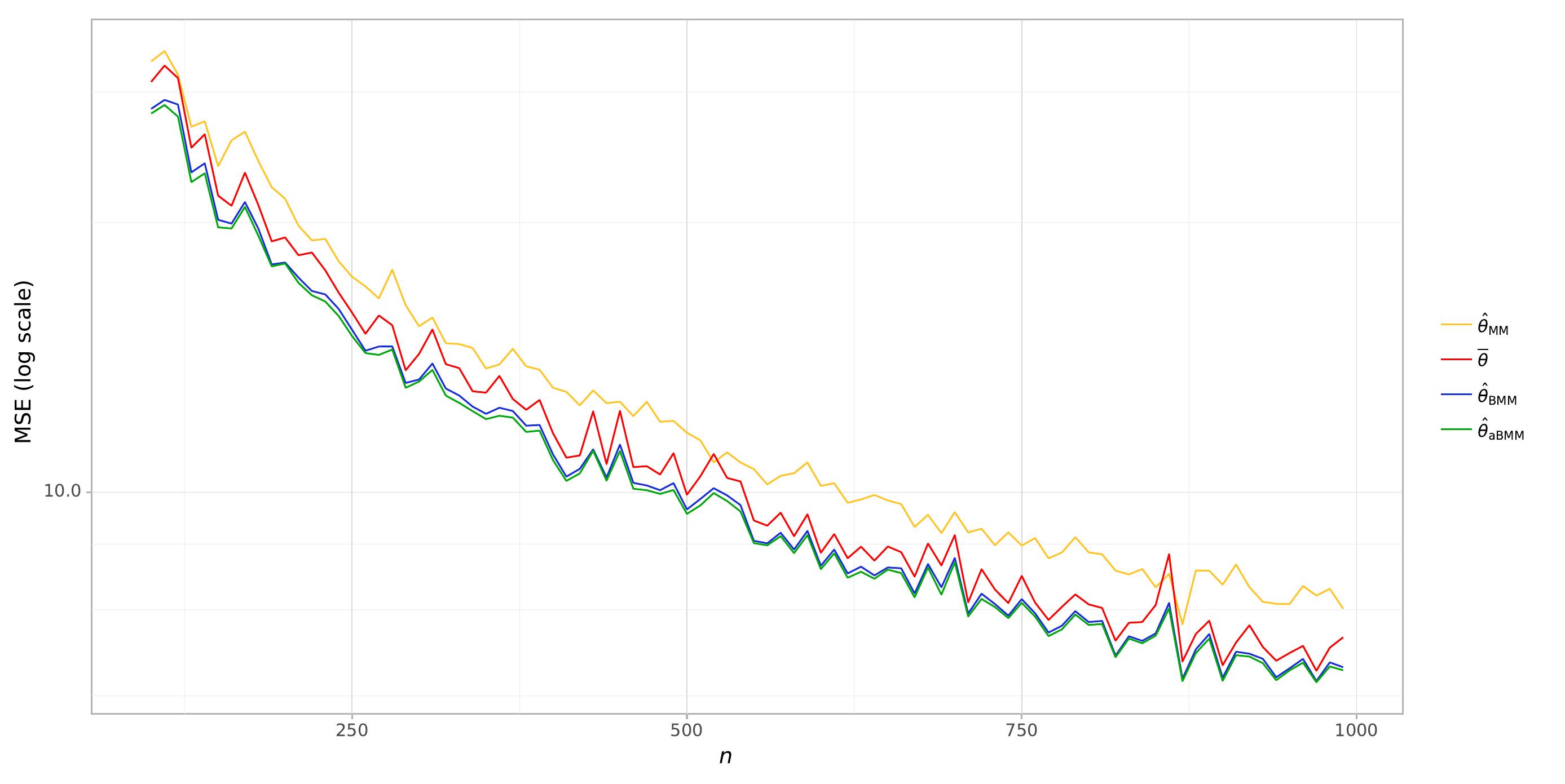}
\caption{Performance of $\SM$, $\BMM$, $\aBMM$ and $\MM$ in terms of MSE for Exponential-$t$ example. The $y$-axis is in the logarithmic scale. The MSE of $\hat{\med}(\hat{\theta}_1, \ldots, \hat{\theta}_n)$ stays around 17 as $n$ grows and is not shown.}
\label{fig:exponential_t}
 \end{center}
\end{figure}

\begin{example}[Lognormal]
Suppose $\hat{\theta}_i \iid \text{Lognormal}(\theta, \sigma^2)$, with $\mu=4$ and $\sigma=1$ (so the shape parameter is 1, the location is 0 and the scale is $e^4$). Recall the Lognormal distribution is in the exponential family, so the sample mean is the MLE. Figure \ref{fig:lognormal} shows the median of means has consistently worse performance. The Bayesian median of means and the sample mean are very close to each other, although the Bayesian median of means is generally better, in particular for small $n$. The approximate Bayesian median of means is not shown because it is virtually identical to its exact counterpart. The median is not shown because its MSE revolves around $1250$ and does not diminish with $n$ (the bias squared in using the median is $(e^{4.5}-e^{4})^2 \approx 1254$).
\end{example}

\begin{figure}[htbp]
 \begin{center}
	 {\fontfamily{cmss}\selectfont\text{\textbf{MSE for Lognormal example}}}  \includegraphics[width=1\textwidth]{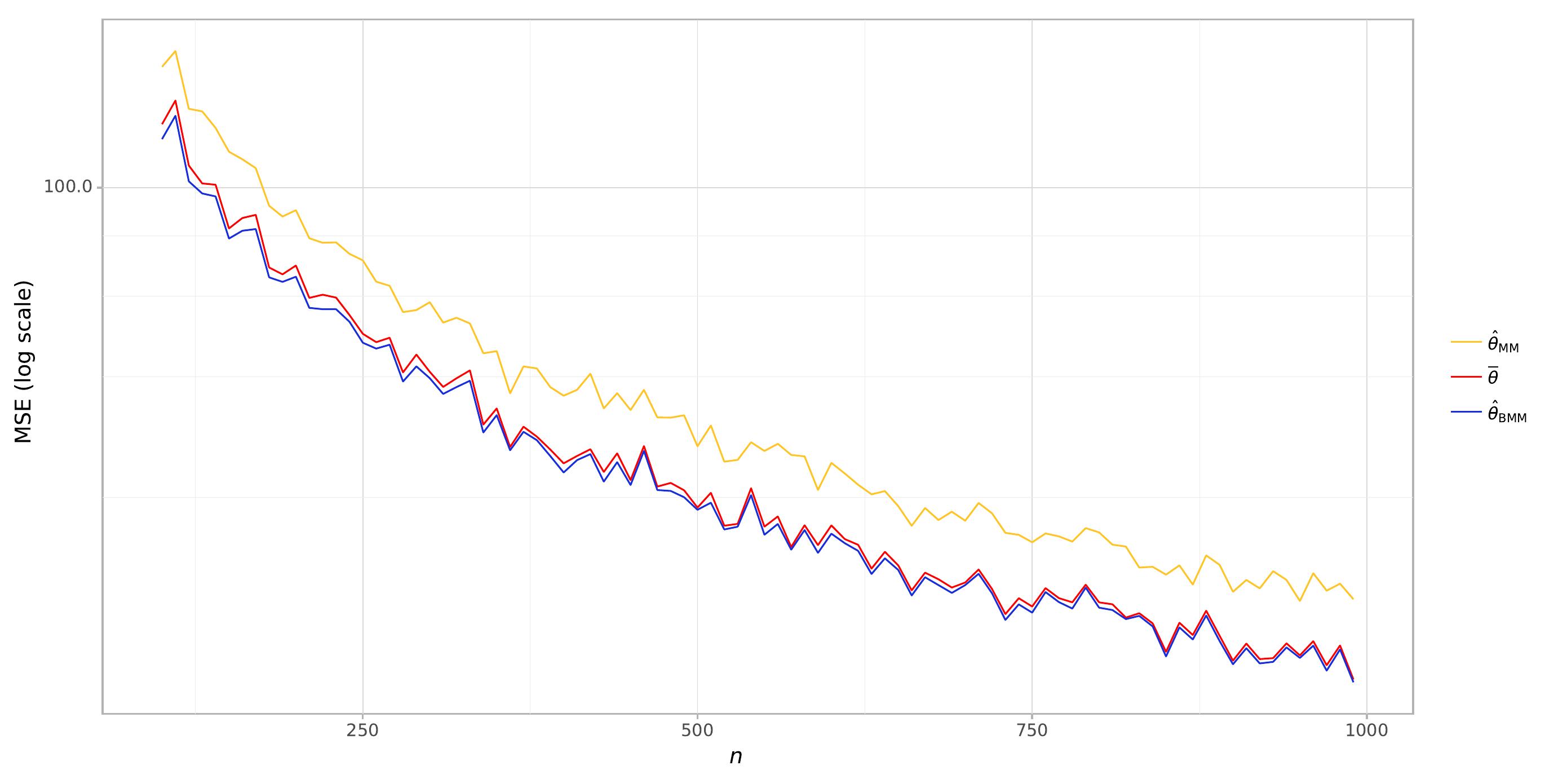}
\caption{Performance of $\SM$, $\BMM$ and $\MM$ in terms of MSE for Lognormal example. The $y$-axis is in the logarithmic scale. The curve for $\aBMM$ is not shown since it is virtually the same as that of $\BMM$; also not shown is $\hat{\med}(\hat{\theta}_1, \ldots, \hat{\theta}_n)$ which has MSE around $1250$ and not decreasing with $n$.}
\label{fig:lognormal}
 \end{center}
\end{figure}

\begin{example}[Maximal bias] For a given variance level, Proposition \ref{prop:mean-median} ensures that the maximum distance between mean and median is the standard deviation. This can be achieved by sampling $\hat{\theta}_i$ as 
\begin{equation*}
 \begin{cases} \sigma &\mbox{with probability } \frac{1}{2} + \varepsilon  \\ -\sigma  &\mbox{with probability } \frac{1}{2} - \varepsilon, \end{cases}
\end{equation*}
for sufficiently small $\varepsilon$, so $\theta\approx0$ while $\med(\hat{\theta}_1)=\sigma$. Note this represents the worst case for any median-based procedure, since the bias is maximal. Figure \ref{fig:worst_case} shows the result of applying the sample mean, the Bayesian median of means and the median of means to this distribution. For small $n$ the sample mean dominates by a small amount, but that gap quickly disappears as $n$ increases.
\begin{figure}[htbp]
 \begin{center}
	 {\fontfamily{cmss}\selectfont\text{\textbf{MSE for maximal bias example}}}  \includegraphics[width=1\textwidth]{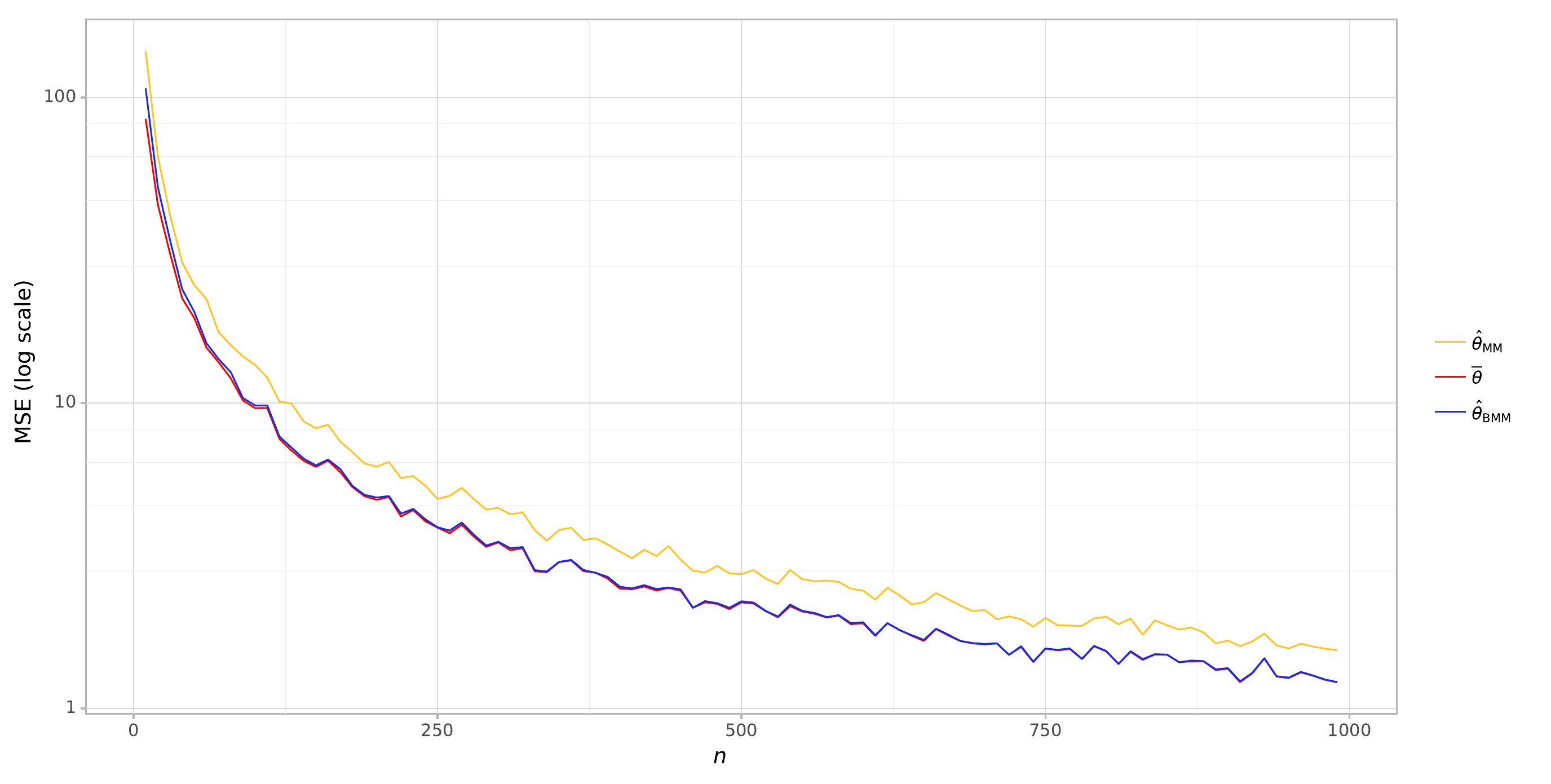}
\caption{Performance of $\SM$, $\BMM$ and $\MM$ in terms of MSE for the maximal bias example. The $y$-axis is in the logarithmic scale. The sample mean has better MSE, but the gap quickly disappears. The curve for $\aBMM$ is not shown since it is virtually the same as that of $\BMM$; also not shown is $\hat{\med}(\hat{\theta}_1, \ldots, \hat{\theta}_n)$ which has MSE around $880$}
\label{fig:worst_case}
 \end{center}
\end{figure}
\end{example}

\subsection{Confidence intervals} \label{sec:confidence intervals}

Constructing confidence intervals for $\BMM$ or $\aBMM$ is not so straightforward since the samples $Y_1, \ldots, Y_J$ are not independent, and confidence intervals for the median are usually based on asymptotic considerations. An alternative is to use the bootstrap. That is, sample $\hat{\theta}_1^{(b)}, \ldots, \hat{\theta}_n^{(b)}$ with replacement from $\{\hat{\theta}_i\}_{i=1}^n$, $b=1, \ldots, B$ times. For each bootstrap sample $b$, find the Bayesian median of means estimator, $\BMM^{(b)}$, and use the empirical distribution to generate confidence intervals for $\theta$ under the assumption that $P[\BMM- \theta \leq t] \approx P[\BMM^{(b)} - \BMM\leq t]$. The examples below show that this procedure yields coverage near the prescribed levels.

Unfortunately, it might be hard or unfeasible to obtain $B$ bootstrap samples for the Bayesian median of means since it involves $O(Bn^2)$ operations. A computational shortcut is to fix the Dirichlet draws $\mathbf{p}^{(1)}, \ldots, \mathbf{p}^{(n)}$ for all $b=1, \ldots, B$ samples, reducing the complexity to $O(Bn)$, which is considerably faster and similar to bootstrapping the sample mean. One would not expect this to significantly change the estimates, since the $\hat{\theta}_i^{(b)}$ are being sampled independently with replacement, but it does add correlation across the samples. In the examples below, the effect of fixing the Dirichlet draws is negligible, while the computational speedup is considerable. Also, both percentile and BCa intervals were analyzed, and the difference was again minor, likely due to the stabilizing effect of the Dirichlet averages. See Algorithm \ref{alg:ci_bmm} for the full description of the percentile bootstrap used.

\begin{algorithm}
\caption{Confidence Interval for Bayesian Median of Means}\label{alg:ci_bmm}
\begin{algorithmic}[1]
\Procedure{CI\_BMM}{$\{\hat{\theta}_i\}_{i=1}^n$, $\tilde{\alpha}=0.05$, $J=n$, $\alpha=1$}
\For{$j=1, \ldots, J$}
\State draw $\mathbf{p}^{(j)} \sim \Dir_n(\alpha, \ldots, \alpha)$
\EndFor
\For{$b=1, \ldots, B$}
\State sample $\hat{\theta}^{(b)}_1, \ldots, \hat{\theta}^{(b)}_1$ with replacement from $\{\hat{\theta}_i\}_{i=1}^n$
\State let $Y_j^{(b)} = \sum_{i=1}^{n} p^{(j)}_i\hat{\theta}_i^{(b)}$
\State let $\BMM^{(b)} = \widhat{\med}(Y_1^{(b)}, \ldots, Y_J^{(b)})$
\EndFor
\State set $(l_{\tilde{\alpha}}, u_{\tilde{\alpha}})$ to be the $\tilde{\alpha}/2$ and $1-\tilde{\alpha}/2$ quantiles of the empirical distribution of $\{\BMM^{(b)}\}_{b=1}^B$
\State \textbf{return} $(l_{\tilde{\alpha}}, u_{\tilde{\alpha}})$
\EndProcedure
\end{algorithmic}
\end{algorithm}

\begin{example}[CIs for Exponential, Pareto and Normal distributions] Figure \ref{fig:ci} shows the confidence intervals at the $95\%$ level for $\BMM$ when (i) $\hat{\theta}_i \iid \text{Expo}(1/3)+5$; (ii) $\hat{\theta}_i \iid \text{Pareto}(4, 10)$; and (iii) $\hat{\theta}_i \iid N(0, 1)$. The $x$-axis refers to $1000$ different draws of $\{\hat{\theta}_i\}_{i=1}^n$, the black dots represent $\BMM$ and the grey lines are the intervals. The dotted line in red is the actual value of $\theta$, and the red crosses are instances where the interval doesn't cover. Note the coverage over these $1000$ draws is close to the nominal level of $95\%$.
\begin{figure}[htbp]
 \begin{center}
{\fontfamily{cmss}\selectfont \textbf{Confidence intervals for $\BMM$}}
\vspace{5mm}
 \includegraphics[width=.78\textwidth]{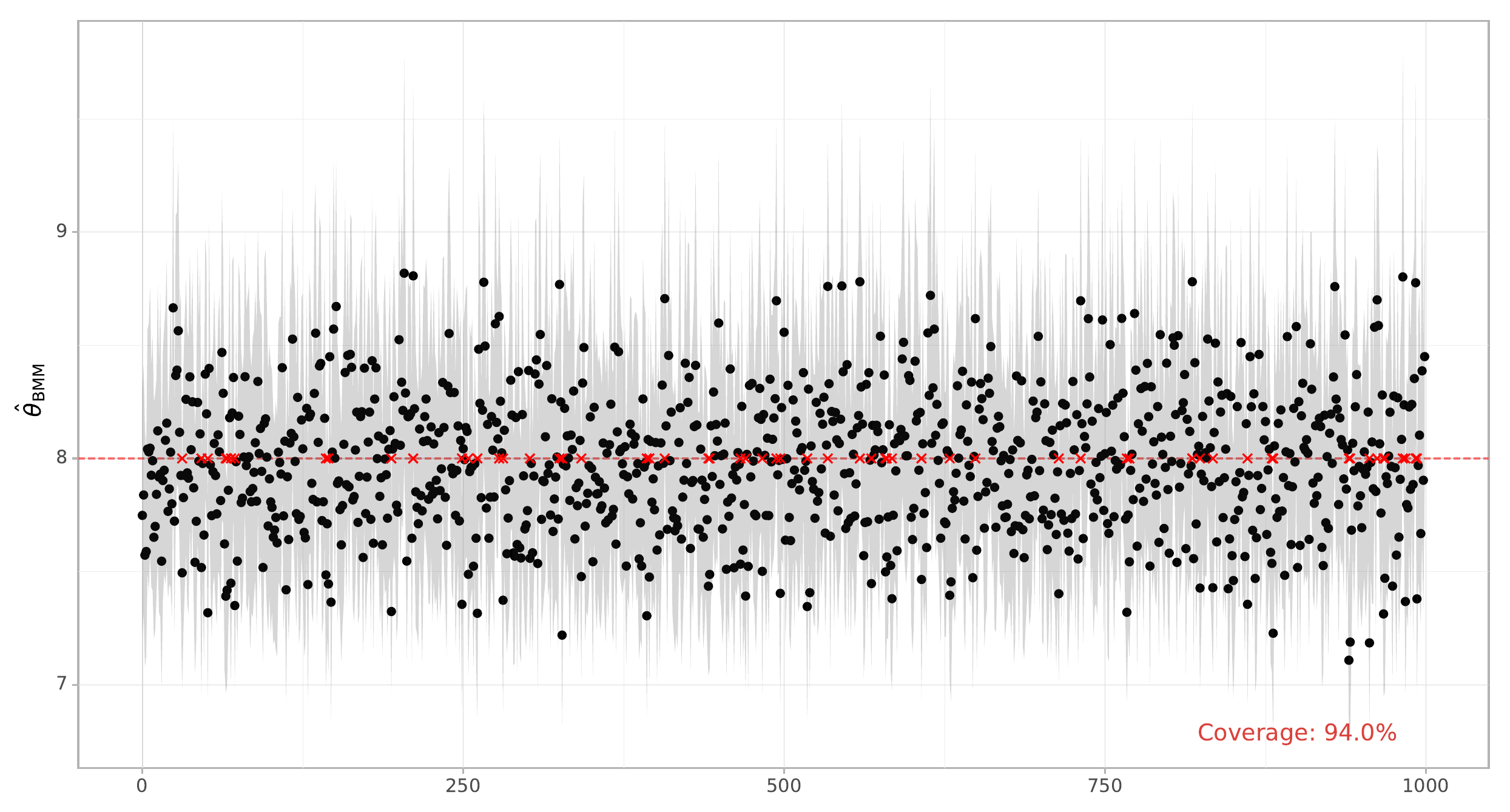}
\includegraphics[width=.78\textwidth]{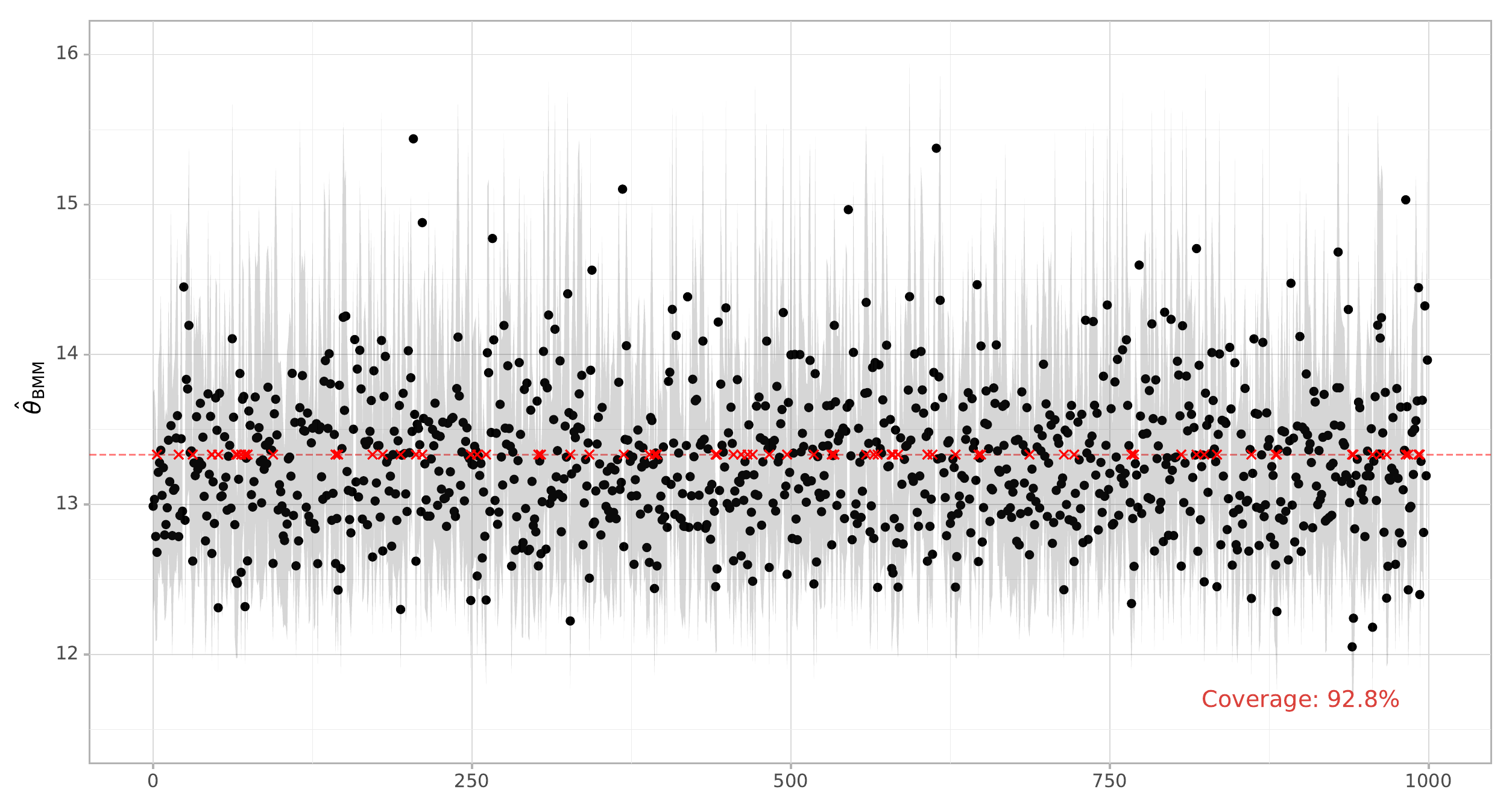}
\includegraphics[width=.8\textwidth]{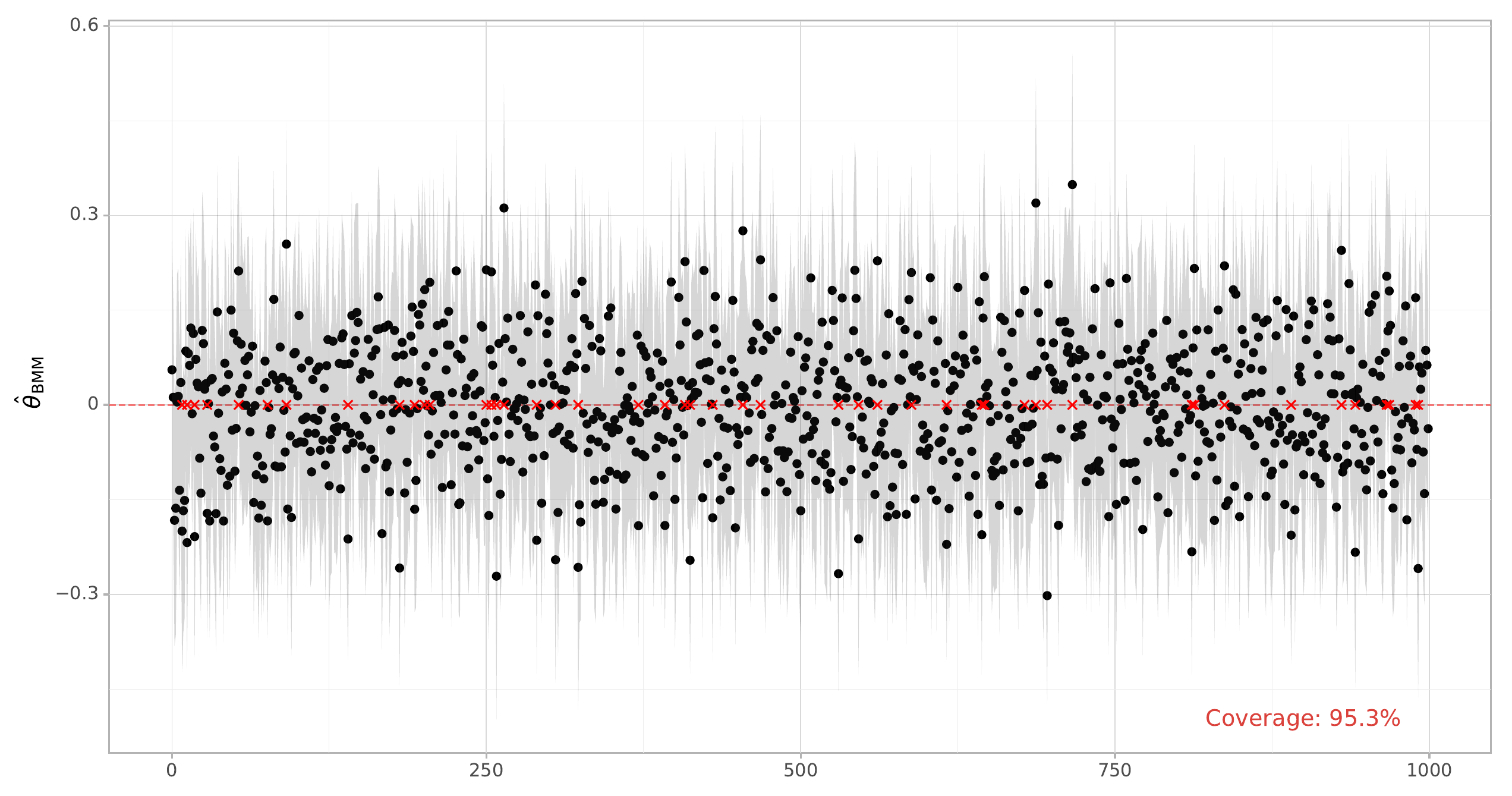}
\caption{Bootstrapped confidence interval and coverages for $\BMM$ and (i) $\hat{\theta}_i \iid \text{Expo}(1/3)+5$; (ii) $\hat{\theta}_i \iid \text{Pareto}(4, 10)$; (iii) $\hat{\theta}_i \iid N(0, 1)$. The red line indicates $\theta$; crosses indicate when CI does not cover.}
\label{fig:ci}
\end{center}
\end{figure}
\end{example}

\subsection{Applications} \label{sec:applications}

Many statistical procedures are based on the idea of aggregating multiple unbiased estimators. This section considers the performance of the Bayesian median of means in importance sampling, bagging and cross-validation. The takeaway is that as long as the estimators being aggregated are low variance and low skewness, the Bayesian median of means performs well and comparable to the sample mean. However, once the estimators exhibit high variance then the sample mean will frequently underperform, and using the Bayesian median of means yields significant gains in mean squared error. 

\subsubsection{Importance sampling}

A class of problems that depend on the aggregation of unbiased estimates with potentially huge variance is importance sampling, as previously illustrated in Example \ref{example:IS_intro_example_MSE}. As another application of importance sampling, consider estimating the number of Fibonacci permutations.

A \emph{Fibonacci permutation} refers to a permutation of $m$ objects such that element $i$ is restricted to positions $i-1$, $i$ or $i+1$ ($1$ can be placed in position $1$ or $2$, and $m$ can placed in position $m-1$ or $m$). For example, if $m=7$, then
\begin{equation*}
[2, 1, 3, 4, 5, 7, 6], \quad [1, 3, 2, 4, 6, 5, 7], \quad [2, 1, 4, 3, 5, 7, 6], \quad [1, 2, 4, 3, 5, 6, 7], \quad [2, 1, 4, 3, 6, 5, 7]
\end{equation*}
are examples of Fibonacci permutations, while $[3, 1, 2, 4, 5, 7, 6]$, $[1, 2, 5, 4, 3, 6, 7]$ and $[2, 1, 4, 3, 6, 7, 5]$ are not. How to efficiently estimate the number of Fibonacci permutations for arbitrary $m$?

On the one hand, it is not hard to see that the number of such Fibonacci permutations is given by $\text{Fib}(m+1)$, where $\text{Fib}(m)$ denotes the $m$-th element in the Fibonacci sequence, $0, 1, 1, 2, 3, 5, 8, \ldots$. Indeed, let $P_m$ be the number of Fibonacci permutations of $m$ elements. Object $1$ can be placed in positions $1$ or $2$; if in position $1$, then there are $P_{m-1}$ possibilities left; if in position $2$, then position $1$ must necessarily have object $2$, and thus there are $P_{m-2}$ possibilities left. Hence, $P_m = P_{m-1} + P_{m-2}$ and $P_1=1$ give the Fibonacci recursion, so $P_m = \text{Fib}(m+1)$.

On the other hand, importance sampling can also be used for this task. The fact that the solution is known means it is possible to faithfully investigate whether using the Bayesian median of means yields any improvement to the method. Also, there are many other similar problems that do not admit closed-form solutions (see \cite{diacoms2001statistical}), so this serves as a benchmark. 

To use importance sampling, let $Z_m$ denote the number of  Fibonacci permutations of $m$ objects, which is the target of the exercise, and take $p(x) = \frac{1}{Z_m}$ to be the Uniform distribution, which is hard to sample from. Consider an alternative distribution that is easier to simulate: (i) the first object's position is to be sampled uniformly at random from $\left\{1, 2\right\}$; (ii) if the first object went to position $1$, then the second object's position is to sampled uniformly at random from $\left\{2, 3\right\}$; if the first object went to position $2$, then the second object must necessarily go to position $1$; (iii-) keep picking among available alternatives for each object's position until all objects are placed. This defines a sequence of conditional distributions for $X_i = (X_{i}^{(1)}, X_{i}^{(2)}, \ldots, X_{i}^{(m)})$:
\begin{equation*}
X_{i}^{(1)} \sim \text{Unif}\left\{1, 2\right\}, \quad X_{i}^{(2)} \sim \begin{cases} \text{Unif}\left\{2, 3\right\} &\mbox{if } X_{i}^{(1)}=1 \\ 1  &\mbox{if } X_{i}^{(1)}=2 \end{cases}, \quad X_{i}^{(3)} \sim \begin{cases} \text{Unif}\left\{3, 4\right\} &\mbox{if } X_{i}^{(2)}=2 \\ 2  &\mbox{if } X_{i}^{(2)}=3 \end{cases}, \quad \ldots.
\end{equation*}
Denote this distribution by
\begin{equation*}
q(X_i)=q_1(X_{i}^{(1)})q_2(X_{i}^{(2)}|X_{i}^{(1)})q_3(X_{i}^{(3)}|X_{i}^{(2)}) \cdots q_n(X_{i}^{(m)}|X_{i}^{(m-1)}),
\end{equation*}
and note that if $\eta_{i, j}$ is the number of neighboring positions available to place $X_i^{(j)}$, (so $\eta_{i, 1}=2$ for all $i$; $\eta_{i, 2}=2$ if $X_i^{(1)}=1$ and $\eta_{i, 2}=1$ if $X_i^{(1)}=2$, etc.), then
\begin{equation*}
\frac{1}{q(X_i)} = \eta_{i}^{(1)} \cdot \eta_{i}^{(2)} \cdots \eta_{i}^{(m)},
\end{equation*}
that is, $q(X_i)$ is the reciprocal of the product of the available neighbors for each position when constructed sequentially.

Now, since it is easy to sample from $q$, $Z_m$ can be estimated using the usual importance sampling estimator:
\begin{align*}
Z_m &= \E_p[Z_m] = \E_q\left[Z_m \frac{p(X)}{q(X)}\right] = \E\left[ \eta^{(1)} \cdot \eta^{(2)} \cdots \eta^{(m)} \right] \approx \frac{1}{n} \sum_{i=1}^{n}  \eta_{i}^{(1)} \cdot \eta_{i}^{(2)} \cdots \eta_{i}^{(m)},
\end{align*}
where $ \eta_{i}$ refers to the neighbors available in constructing $X_i$. This can be done using Algorithm \ref{alg:fib} below. 

\begin{algorithm}
\caption{Importance sampling Fibonacci estimate}\label{alg:fib}
\begin{algorithmic}[1]
\Procedure{IS\_Fib\_est}{$m$}
\State \textbf{set} $p=(-1, -1, \ldots, -1)$ where $\text{length}(p)=m$
\State \textbf{initialize} $\text{prod\_available\_neighbors}=1$, $i=0$
\While{$i<m$}
\If{$i=0$}: \State $\text{available\_neighbors}=\{0, 1\}$
\ElsIf{$i=m-1$}
\State $\text{available\_neighbors}=\{x \in \left\{m-2, m-1\right\} \ : \ p[x]=-1\}$
\Else
\State $\text{available\_neighbors}=\{x \in \left\{i-1, i, i+1 \right\} \ : \ p[x]=-1\}$
\EndIf
\State $\text{prod\_available\_neighbors}  = \text{length}(\text{available\_neighbors}) * \text{prod\_available\_neighbors}$
\State \text{chosen\_neighbor} = \text{sample}(\text{available\_neighbors})
\State $p[\text{chosen\_neighbor}] = i$
\If{$\text{chosen\_neighbor} \neq i$}: 
\State $p[i] = \text{chosen\_neighbor}$
\State $i = i+1$
\EndIf
$i = i+1$
\EndWhile
\State \textbf{return} $\text{prod\_available\_neighbors}$
\EndProcedure
\end{algorithmic}
\end{algorithm}

Figure \ref{fig:fibonacci_permutations_hist} shows histograms for $\hat{\theta}_i$, for $i=1, \ldots, 1000$, when $m=20, 50, 80, 100$. Note the problem becomes harder as $m$ increases, in particular due to the presence of extremely large but unlikely values of $\hat{\theta}_i$, corresponding to rare permutations (so $q(X_i)$ is very small, and thus $1/q(X_i)$ is very big). These large weights are essential for unbiasedness, but add substantial variance to the estimates. 

\begin{figure}[!ht]
\centering
{\fontfamily{cmss}\selectfont\text{\textbf{Histogram of individual importance sampling estimates}}}
\vspace{5mm}

\begin{subfigure}{.35\textwidth}
  \centering
  \includegraphics[width=\linewidth]{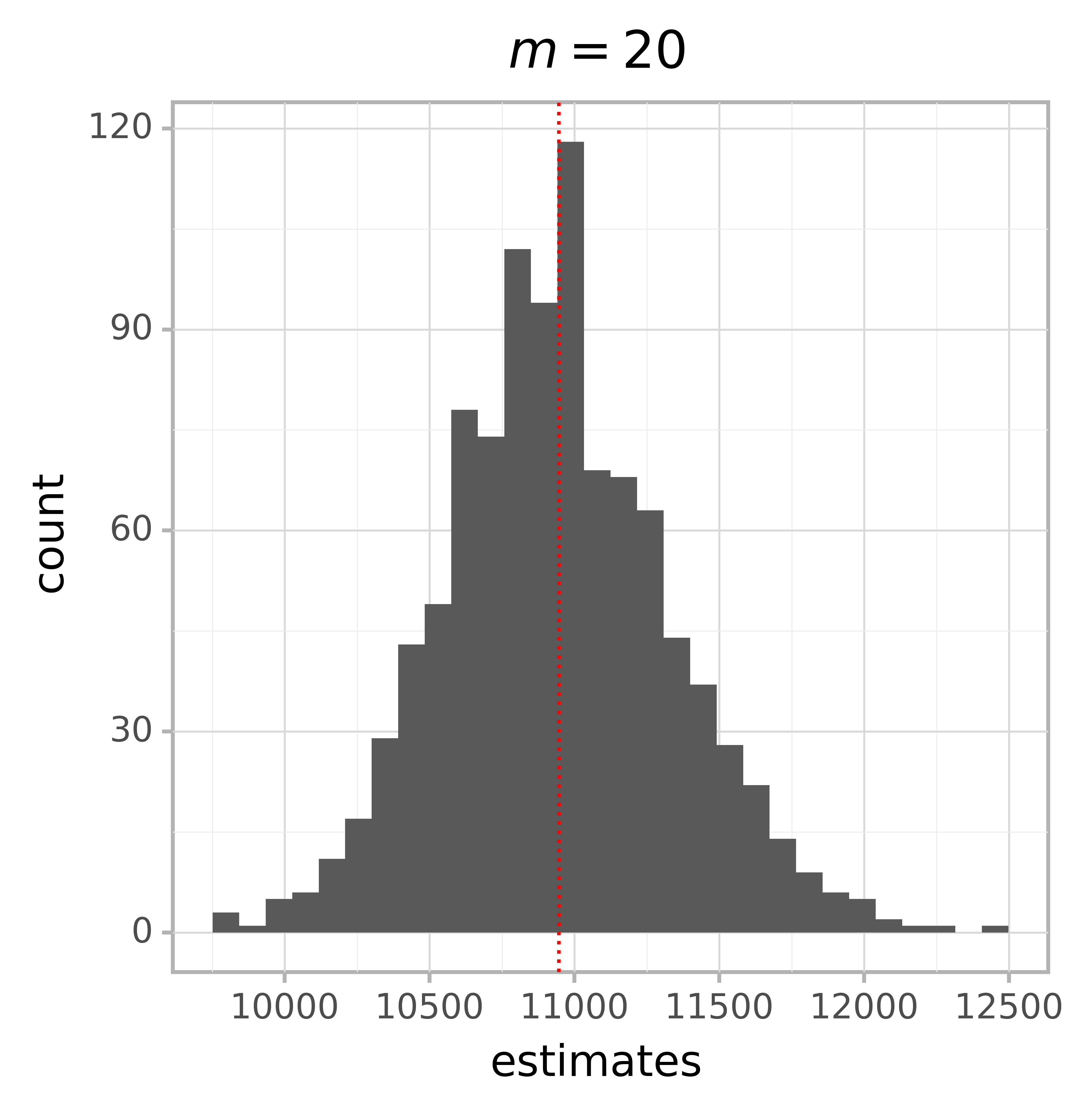}
\end{subfigure}%
\begin{subfigure}{.35\textwidth}
  \centering \includegraphics[width=\linewidth]{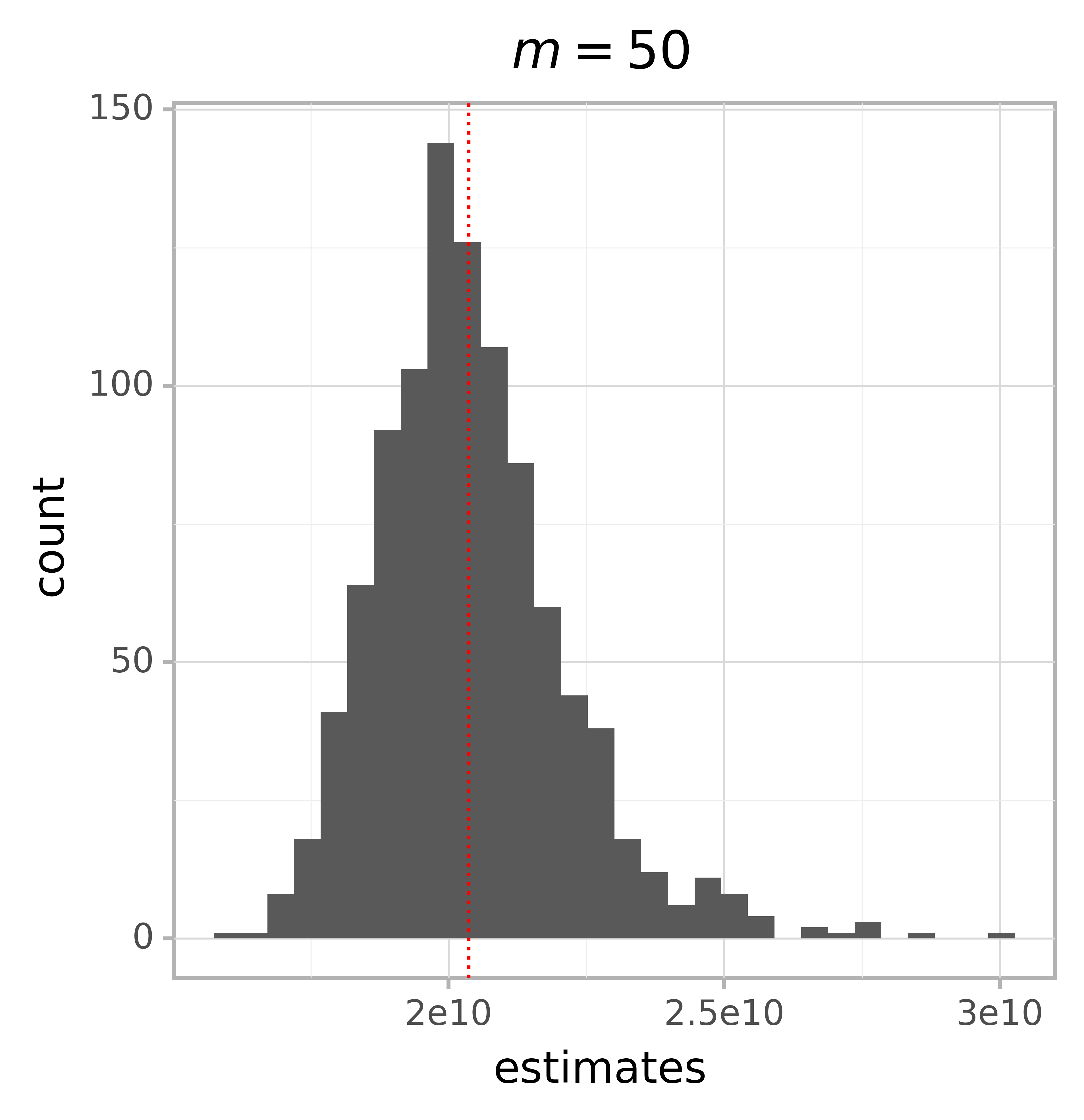}
\end{subfigure}

\begin{subfigure}{.35\textwidth}
  \centering
  \includegraphics[width=\linewidth]{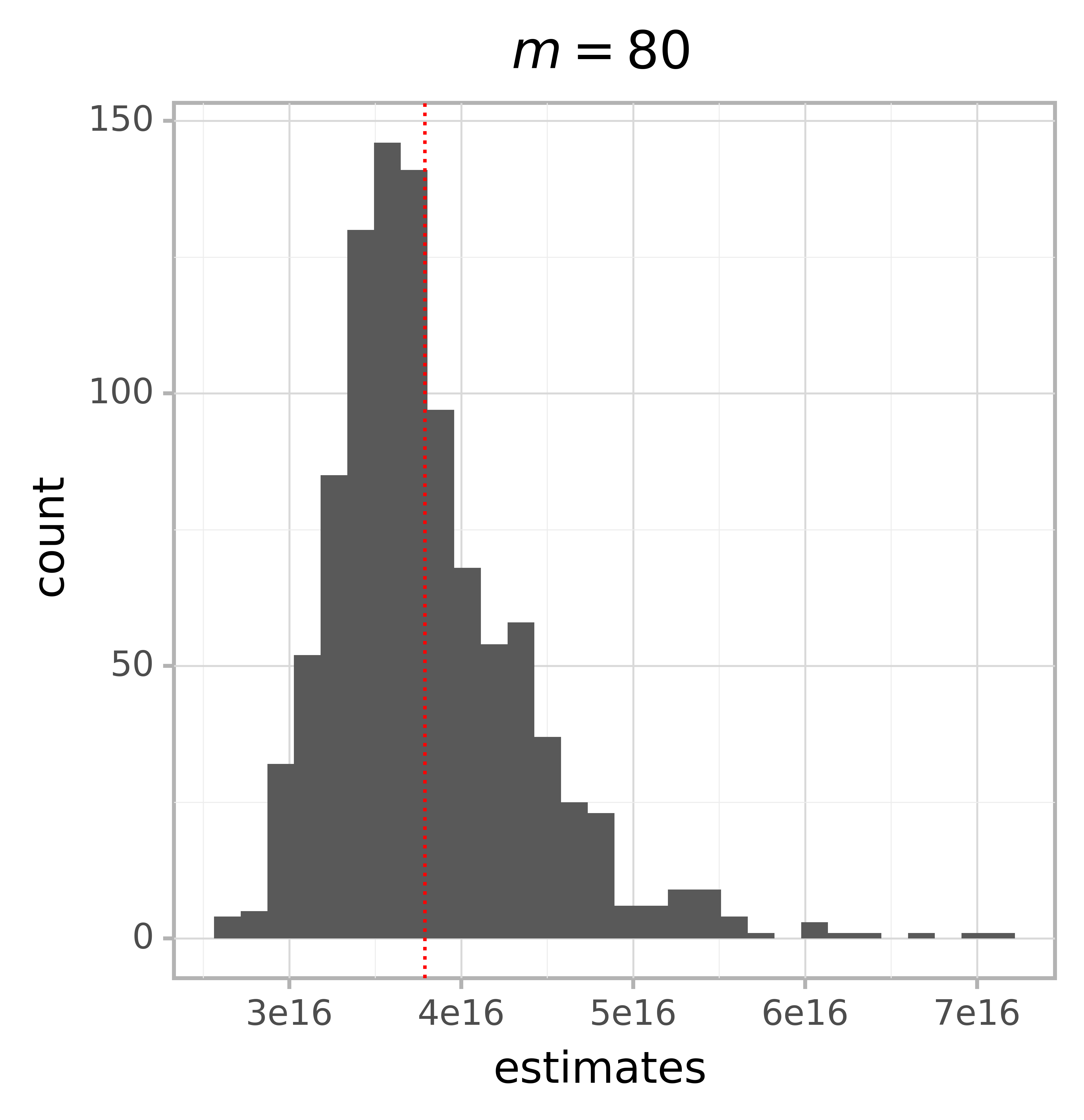}
\end{subfigure}%
\begin{subfigure}{.35\textwidth}
  \centering \includegraphics[width=\linewidth]{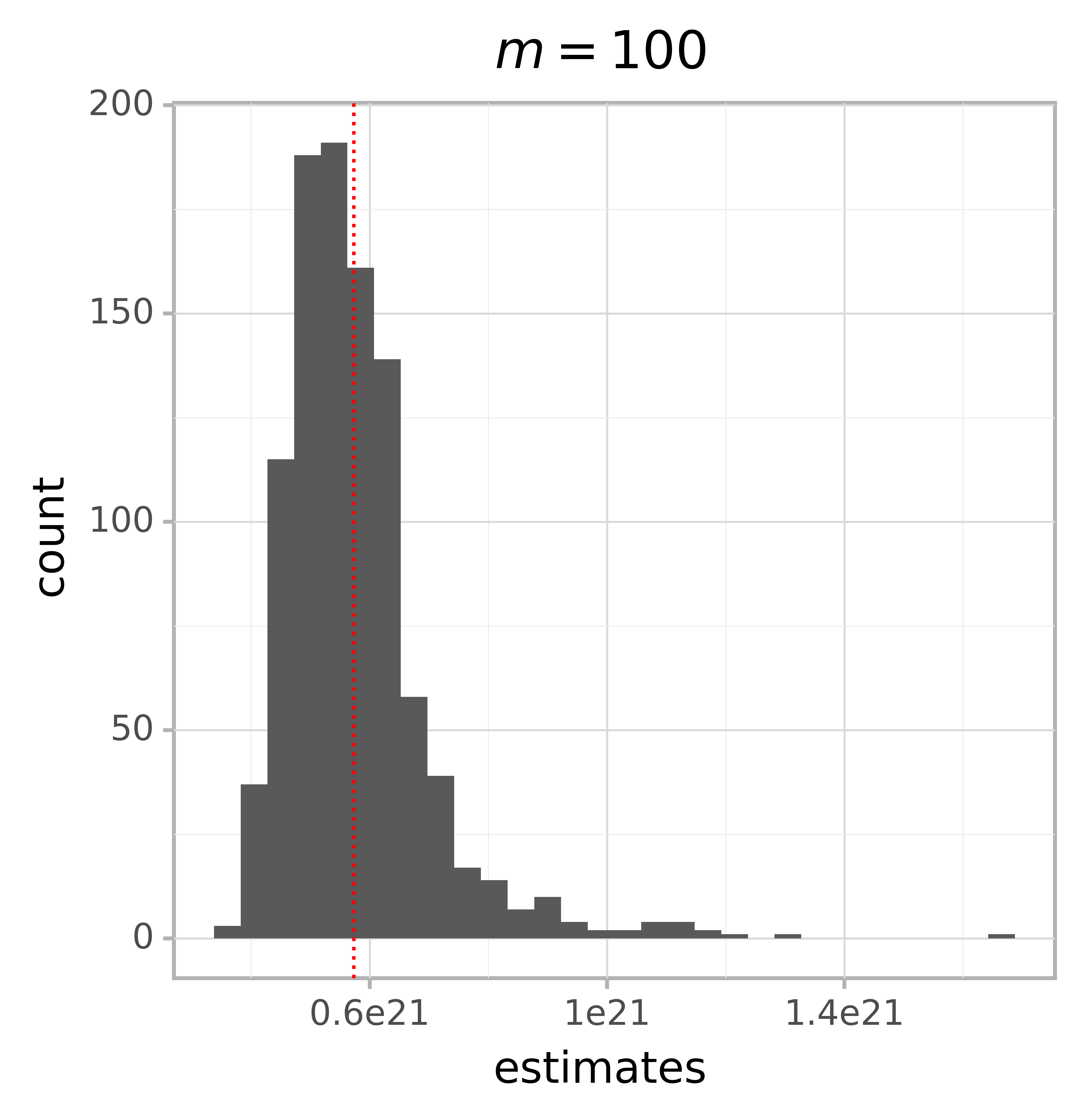}
\end{subfigure}
\caption{Histograms for each $\hat{\theta}_i$, for Fibonacci permutations of $m=20, 50, 80, 100$; in each plot, $n=1000$. As $m$ increases, extremely large but rare estimates appear. The red line indicates the true $\theta$.}
\label{fig:fibonacci_permutations_hist}
\end{figure}

Thus, importance sampling requires generating Fibonacci permutations $X_i$ from $q$ and calculating
\begin{equation*}
\hat{\theta}_i = \eta_{i}^{(1)} \cdot \eta_{i}^{(2)} \cdots \eta_{i}^{(n)},
\end{equation*}
where $\eta_i^{(j)}$ refer to the number of possible positions object $j$ could take when the previous objects $1, \ldots, j-1$ have been placed at positions $X_i^{(1)}, \ldots, X_i^{(j-1)}$. The usual importance sampling estimate is obtained by aggregating these estimates via the sample mean:
\begin{equation*}
\IS = \frac{1}{n} \sum_{i=1}^{n} \hat{\theta}_i.
\end{equation*}
As an alternative, one can consider aggregating $\hat{\theta}_i$ via the Bayesian median of means.

Figure \ref{fig:fibonacci_permutations_boxplot} has the results of using $\IS$, $\BMM$ and $\aBMM$ when $m=20, 50, 80, 100$ and $n=1000$. It becomes quite clear that $\BMM$, and even more so $\aBMM$, improve on the usual importance sampling estimates by severely biasing the large weights towards the mean (as illustrated in Figure \ref{fig:intuition_BMM}), resulting in larger gains in terms of MSE the larger the $m$.

Figure \ref{fibonacci_permutations_change_n} shows the relative MSE, measured as $\E[(\hat{\theta}-\theta)^2/\theta^2]$, for $\SM$, $\BMM$ and $\aBMM$. The top picture displays the case in which $n=100$ samples are used, and the lower picture has $n=1000$. Note that higher $m$ makes $\BMM$ and $\aBMM$ better since the variance component of MSE overwhelms the bias. For small $m$, all estimators perform comparably. Also, higher $n$ means all estimators are closer to each other, and the differences in relative MSE are lower.

\begin{figure}
\centering

{\fontfamily{cmss}\selectfont\text{\textbf{Boxplots for importance sampling estimates}}} 

\vspace{10mm}

\begin{subfigure}{.5\textwidth}
  \centering
  $m=20$ \includegraphics[width=\linewidth]{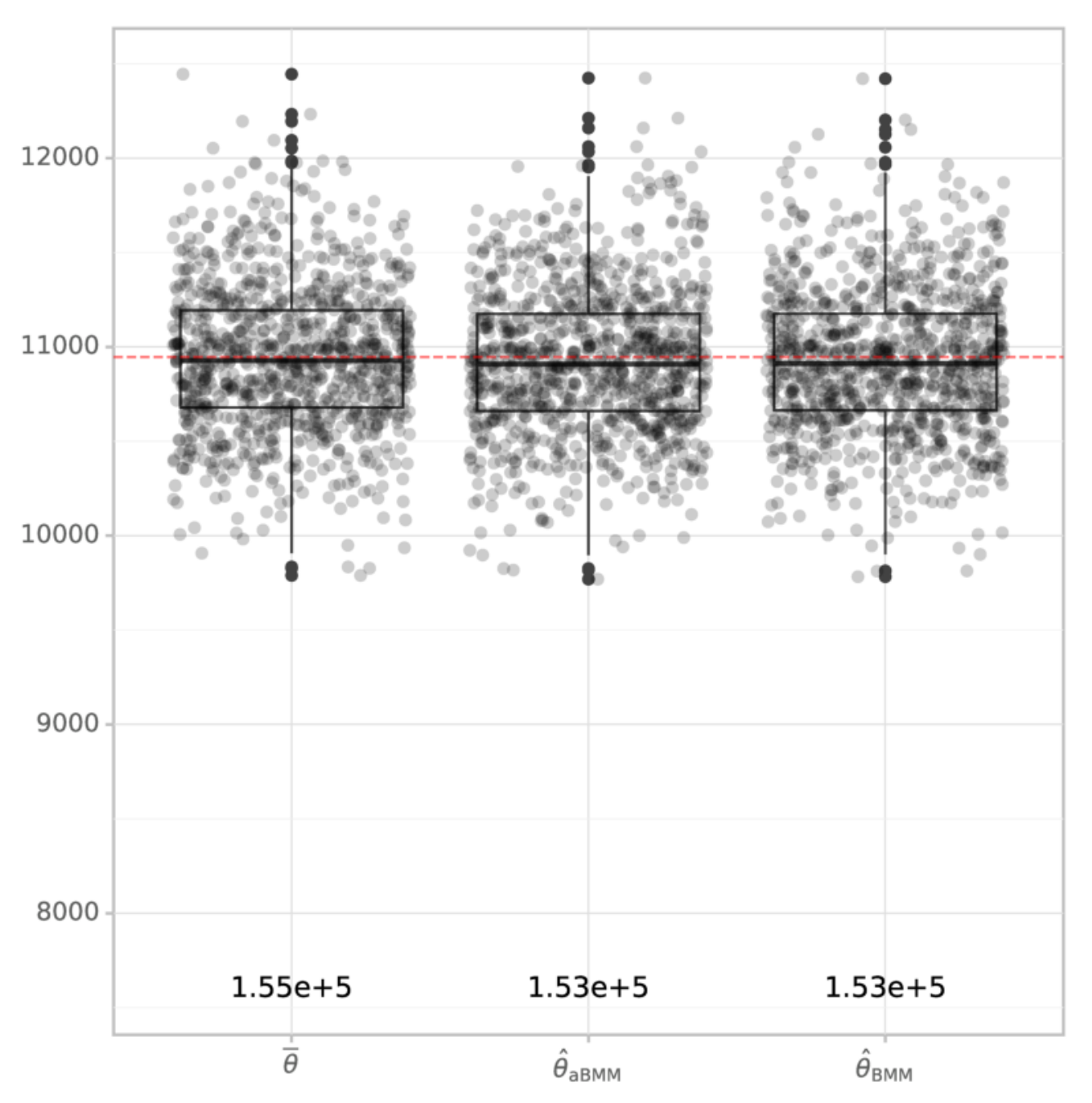}
\end{subfigure}%
\begin{subfigure}{.5\textwidth}
  \centering
  $m=50$ \includegraphics[width=\linewidth]{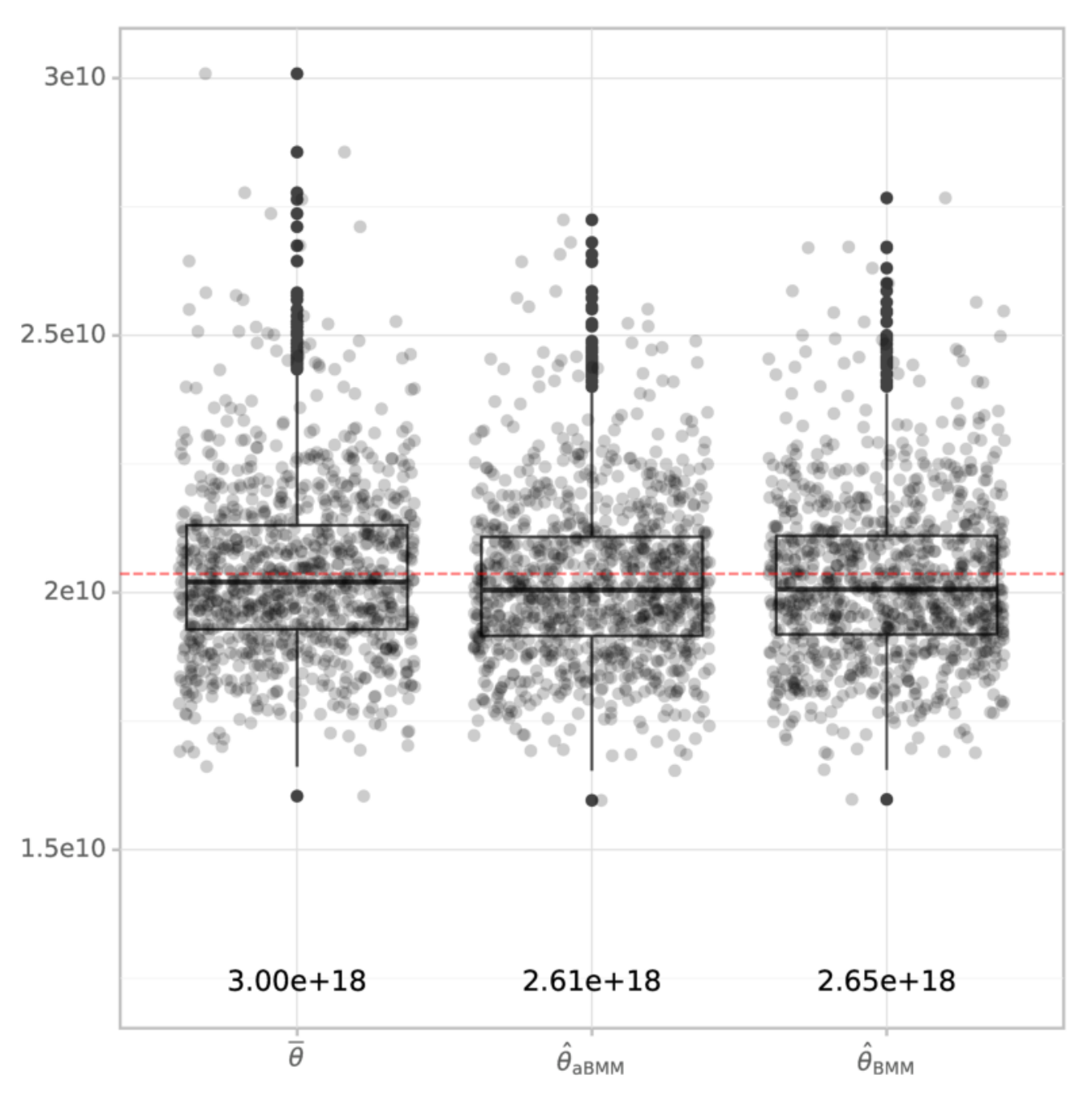}
\end{subfigure}

\vspace{5mm}

\begin{subfigure}{.5\textwidth}
  \centering
 $m=80$ \includegraphics[width=\linewidth]{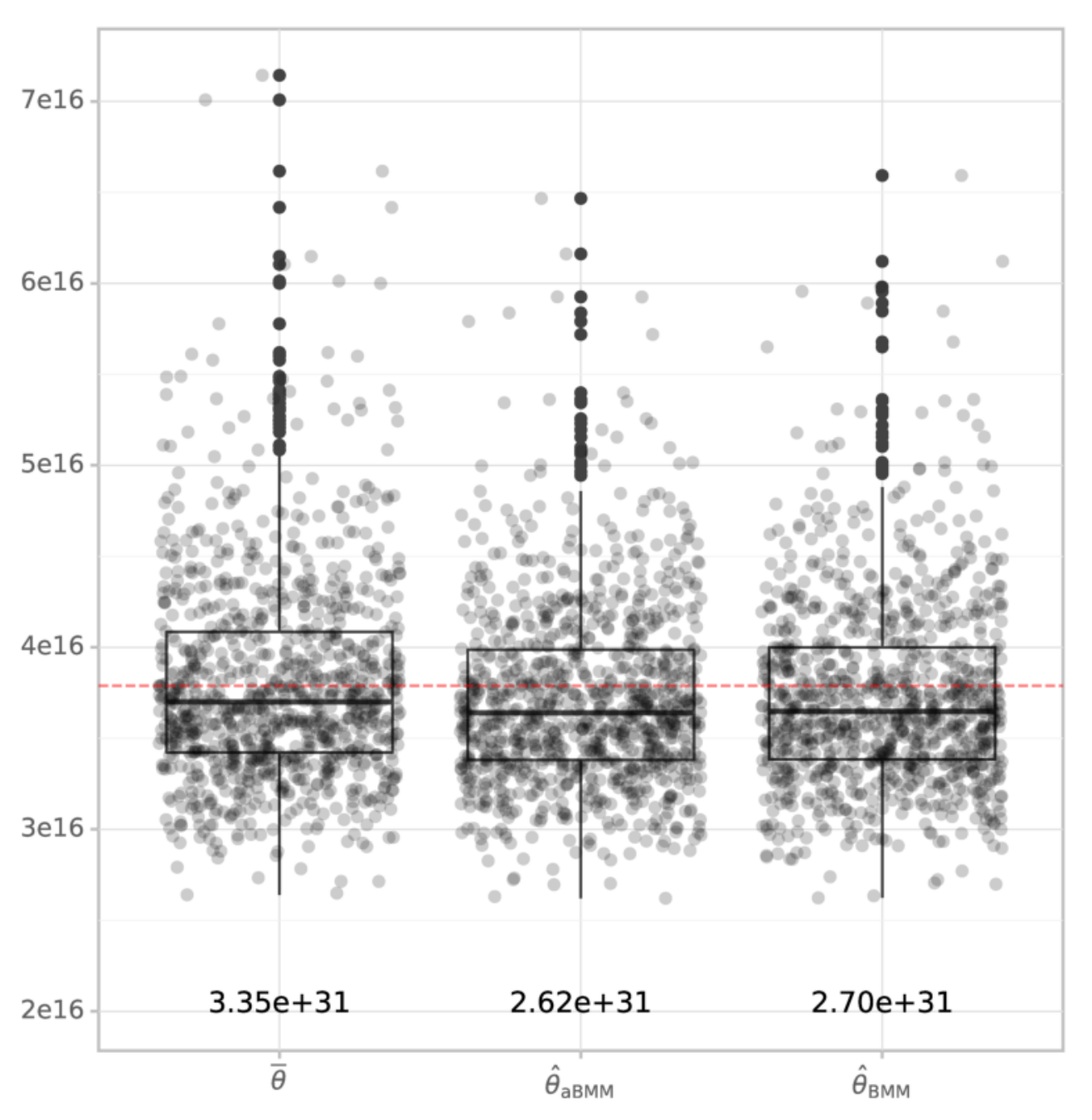}
\end{subfigure}%
\begin{subfigure}{.5\textwidth}
  \centering
  $m=100$ \includegraphics[width=\linewidth]{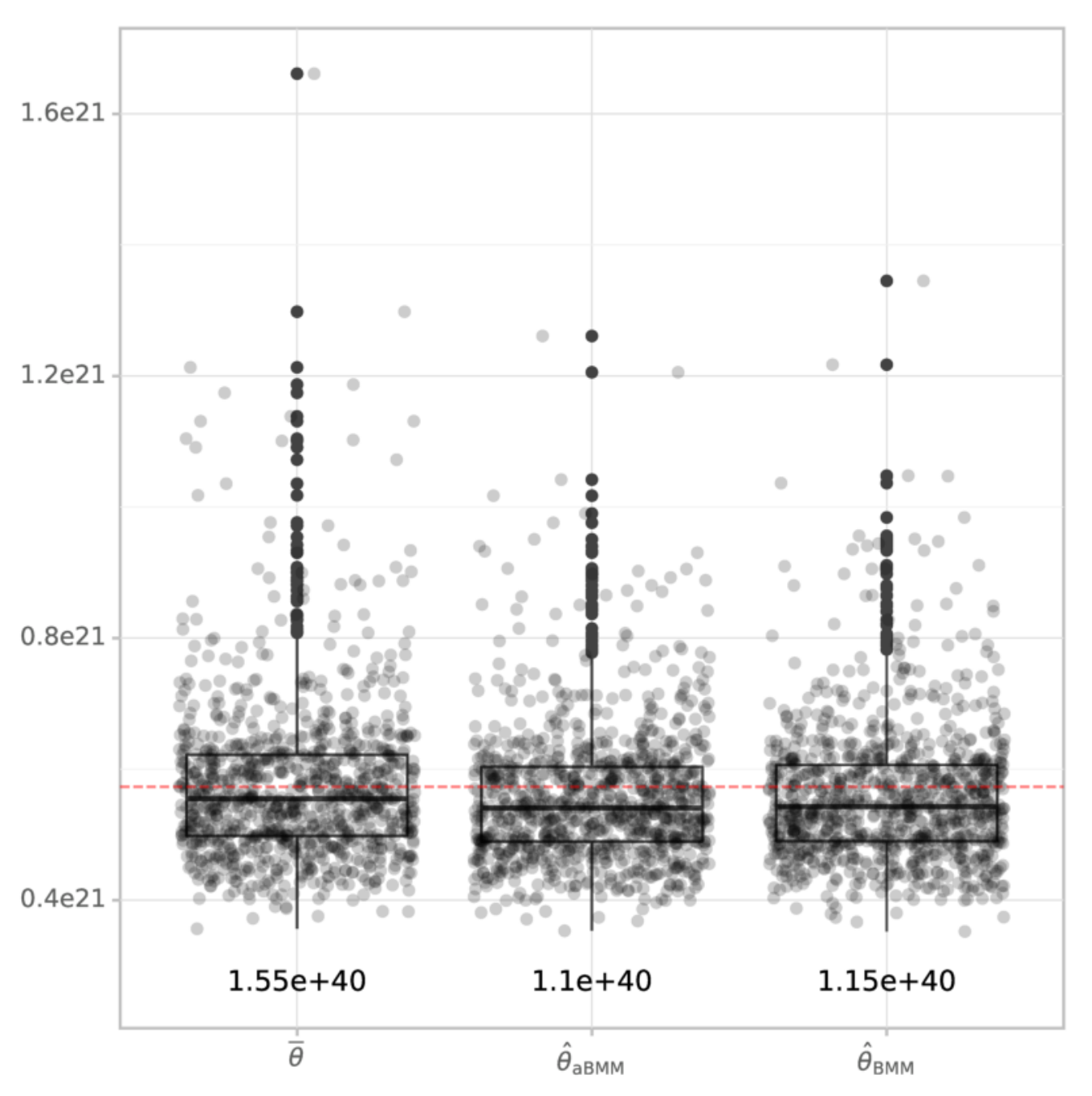}
\end{subfigure}
\caption{Boxplots for $\IS$, $\aBMM$ and $\BMM$, for Fibonacci permutations of $m=20, 50, 80, 100$; in each plot, $n=1000$, and the red line denotes the true value. The MSE for the estimators are shown in black. Note $\BMM$, and even more so $\aBMM$, improve on the usual IS estimator by biasing the estimates (the more so the more extreme the estimate).}
\label{fig:fibonacci_permutations_boxplot}
\end{figure}

\begin{figure}[htbp]
 \begin{center}
{\fontfamily{cmss}\selectfont\text{\textbf{Relative MSE for importance sampling estimates}}}
\vspace{5mm}

$n=100$
\includegraphics[width=1\textwidth]{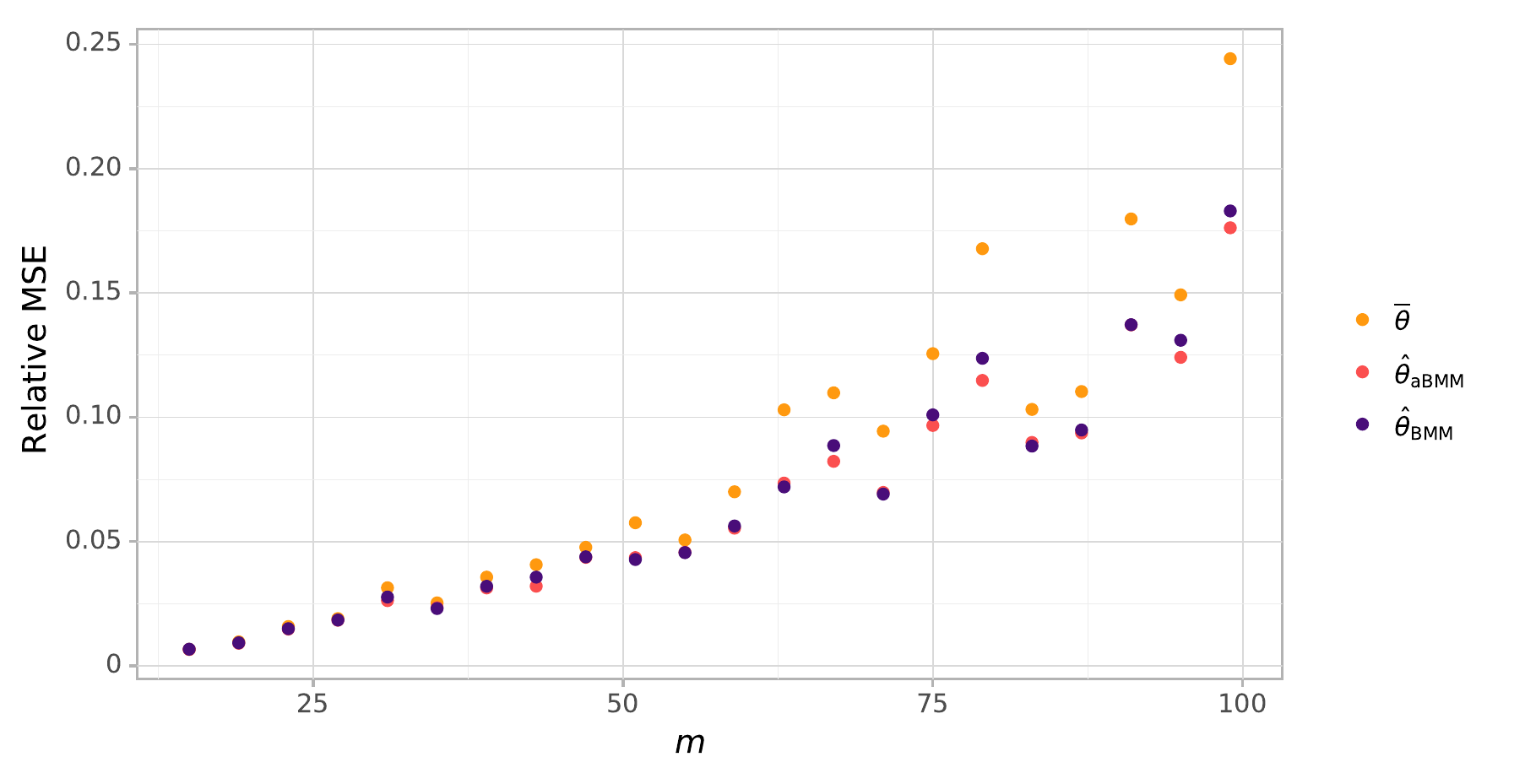}

\vspace{5mm}

$n=1000$ \includegraphics[width=1\textwidth]{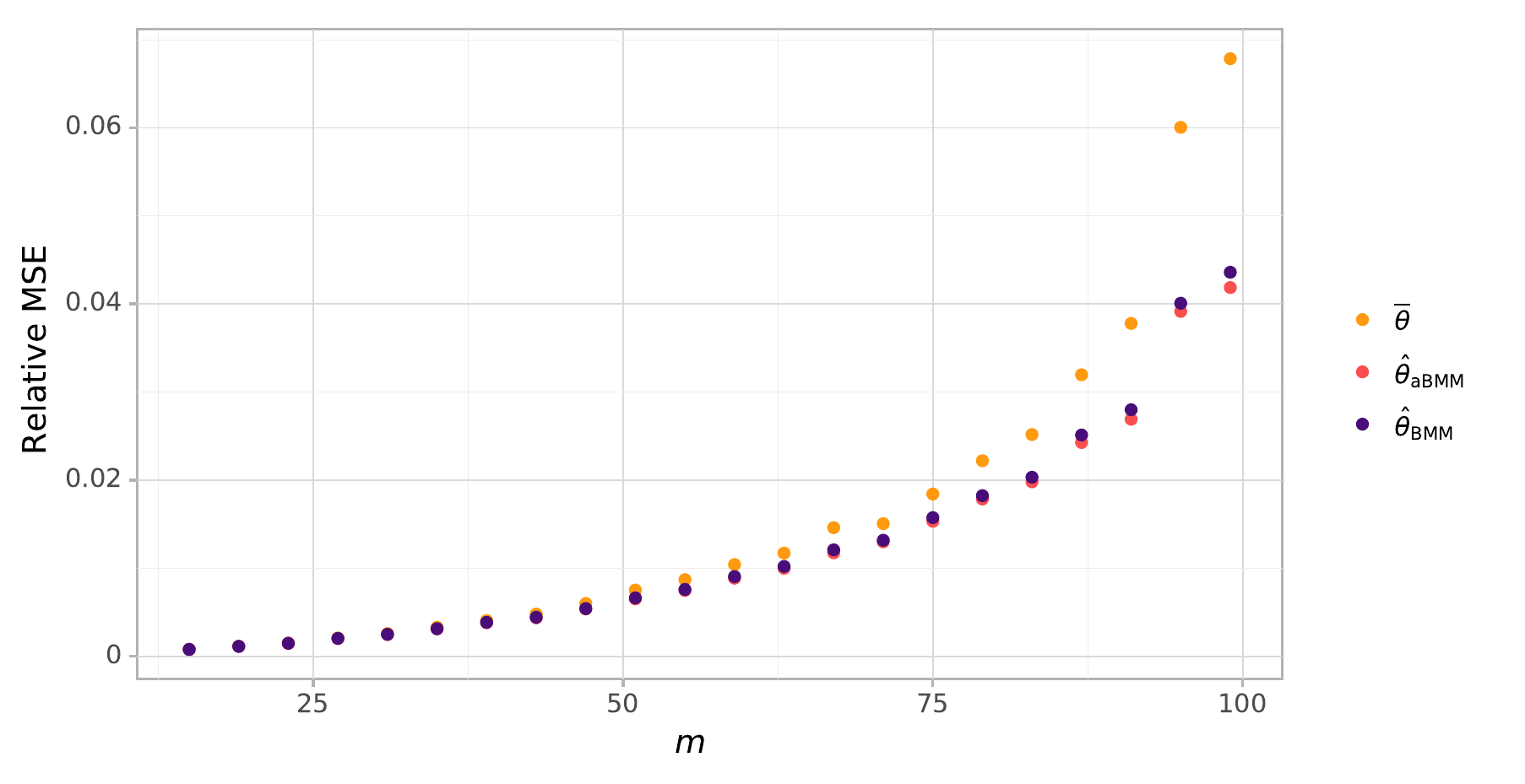}
\caption{Relative MSE, estimating $\E[(\hat{\theta}-\theta)^2/\theta^2]$, with $\hat{\theta}=\SM$, $\aBMM$ and $\BMM$, for $m=15, \ldots, 100$. Top picture shows case $n=100$ and lower picture shows $n=1000$. The $\BMM$ and $\aBMM$ gains are bigger the larger the $m$, the lower the $n$.}
\label{fibonacci_permutations_change_n}
 \end{center}
\end{figure}

\subsubsection{Cross-validation}

The Bayesian median of means can be used to improve many statistical methods that involve resampling the data. For instance, consider cross-validation, which is used to estimate the test error of a procedure, say linear regression. One divides the data into folds, and try to predict each fold using the other folds as training data. The estimated test error in each fold is aggregated via the sample mean, but consider using the Bayesian median of means instead. The gains should be especially evident when the estimates from different folds vary significantly.

Below, both methods are compared by estimating the test error in a variety of datasets. In each case, $20\%$ of the data is set aside as a validation set, and the remaining $80\%$ is used to generate the usual cross-validation (CV) estimates as well as the Bayesian median of means aggregation (BMM). Then, the full training data is used to train the statistical procedure, and the validation set provides an accurate estimate of test error (error). For each dataset, this is repeated for 100 different seeds, and the root mean squared error (RMSE) between CV and error, and BMM and error are computed.

The datasets were chosen from \cite{james2013introduction} and include the following:

\begin{enumerate}[label=(\roman*)]
\item \texttt{Advertising}: given advertising budgets for TV, radio and newspaper, predict sales of a product ($n=200$ data points, $p=4$ features).
\item \texttt{Auto}: given several features about a car, such as horsepower, weight and acceleration, predict its overall displacement ($n=392$, $p=9$).
\item \texttt{College}: using features from different colleges, such as whether they are private, number of accepted or enrolled students and number of graduate students, predict its graduation rate ($n=777$, $p=19$).
\item \texttt{Credit}: knowing a person's income, age, education and other financial data, predict their credit card balance ($n=400$, $p=12$). 
\item \texttt{Heart}: given features from people with chest pain, such as their age, sex and fitness measurements, predict whether they have heart disease or not ($n=297$, $p=17$).
\end{enumerate}

The first four datasets above involve regression tasks, while the last one is a classification problem. For the first four, the underlying procedures used were linear regression, random forest (with $20$ different trees) and $k$-nearest neighbors (with $5$ nearest neighbors). All features available were used, and dummies were created when necessary. The loss functions used was mean squared error for regression, and misclassification error for classification. The Bayesian median of means was used with $J=1000$ Dirichlet draws.

Generally, Tables \ref{table:advertising}-\ref{table:heart} show that the Bayesian median of means adds a small amount of stability to the cross-validation estimates, resulting in lower mean squared error. While the improvement is consistent across different datasets, it is often small. This is likely the consequence of estimates from different folds begin similar, so the Bayesian median of means cannot leverage its gains in variance reduction. In fact, in some cases, the usual cross-validation exhibit smaller standard error, such as in Table \ref{table:college}, but the only case in which cross-validation exhibits smaller MSE is in Table \ref{table:credit} for linear regression and $k$-nearest neighbors. Some care must be taken in interpreting these results, as real data represents only one realization from the probability distribution characterizing the data-generating mechanism. Still, the results support the conclusion that the Bayesian median of means is a competitive way to perform aggregation in data resampling schemes. 

\begin{table}[!ht]
\begin{center}
 \caption{{\fontfamily{cmss}\selectfont\textbf{Results for \texttt{Advertising} data}}} \label{table:advertising}
 \begin{tabular}{c||c|c|c||c|c|c||c|c|c}
  & \multicolumn{3}{c||}{Linear Regression} & \multicolumn{3}{c||}{Random forest} & \multicolumn{3}{c}{$k$NN} \\ \hline
  & CV & BMM & error & CV & BMM & error & CV & BMM & error \\ \hline
 mean & 3.002 & 2.939 & 2.889 & 0.714 & 0.7 & 0.667 & 2.532 & 2.514 & 2.264  \\ \hline
 std & 0.246 & 0.23 & 0.937 & 0.08 & 0.076 & 0.227 & 0.205 & 0.201 & 0.592  \\ \hline
 RMSE & 1.184 & \textbf{1.16} & - & 0.263 & \textbf{0.257} & - & 0.737 & \textbf{0.725} & - \\ \hline
 \end{tabular}
\end{center}
\end{table}

\begin{table}[!ht]
\begin{center}
 \caption{{\fontfamily{cmss}\selectfont\textbf{Results for \texttt{Auto} data}}} \label{table:auto}
 \begin{tabular}{c||c|c|c||c|c|c||c|c|c}
  & \multicolumn{3}{c||}{Linear Regression} & \multicolumn{3}{c||}{Random forest} & \multicolumn{3}{c}{$k$NN} \\ \hline
  & CV & BMM & error & CV & BMM & error & CV & BMM & error \\ \hline
 mean & 536.23 & 531.63 & 508.11 & 389.7 & 384.74 & 366.69 & 1345.65 & 1322.42 & 1235.03  \\ \hline
 std & 27.89 & 27.19 & 105.39 & 36.69 & 34.28 & 120.37 & 100.9 & 95.69 & 339.19
  \\ \hline
 RMSE & 134.72 & \textbf{133.09} & - & 135.12 & \textbf{133.92} & - & 443.72 & \textbf{432.38} & -
  \\ \hline
 \end{tabular}
\end{center}
\end{table}

\begin{table}[!ht]
\begin{center}
 \caption{{\fontfamily{cmss}\selectfont\textbf{Results for \texttt{College} data}}} \label{table:college}
 \begin{tabular}{c||c|c|c||c|c|c||c|c|c}
  & \multicolumn{3}{c||}{Linear Regression} & \multicolumn{3}{c||}{Random forest} & \multicolumn{3}{c}{$k$NN} \\ \hline
  & CV & BMM & error & CV & BMM & error & CV & BMM & error \\ \hline
 mean & 172.15 & 171.802 & 168.24 & 180.363 & 179.883 & 176.662 & 212.989 & 212.687 & 210.43 \\ \hline
 std & 6.601 & 6.682 & 23.478 & 8.332 & 8.37 & 24.348 & 8.303 & 8.285 & 24.506 \\ \hline
 RMSE &  30.044 & \textbf{29.989} & - & 31.478 & \textbf{31.445} & - & 31.949 & \textbf{31.861} & - \\ \hline
 \end{tabular}
\end{center}
\end{table}

\begin{table}[!ht]
\begin{center}
 \caption{{\fontfamily{cmss}\selectfont\textbf{Results for \texttt{Credit} data}}} \label{table:credit}
 \begin{tabular}{c||c|c|c||c|c|c||c|c|c}
  & \multicolumn{3}{c||}{Linear Regression} & \multicolumn{3}{c||}{Random forest} & \multicolumn{3}{c}{$k$NN} \\ \hline
  & CV & BMM & error & CV & BMM & error & CV & BMM & error \\ \hline
 mean & 10116.4 & 10080.4 & 10451.1 & 13662 & 13485.5 & 12966.4 & 48018.5 & 47829.4 & 47001.5  \\ \hline
 std & 444.9 & 452.6 & 1640.4 & 1201.4 & 1159.8 & 3912 & 2548.6 & 2599.8 & 7793.7  \\ \hline
 RMSE & \textbf{2098} & 2111.9 & - & 4492.9 & \textbf{4481.8} & - & \textbf{10004.7} & 10027.3 & -  \\ \hline
 \end{tabular}
\end{center}
\end{table}

\begin{table}[!ht]
\begin{center}
 \caption{\fontfamily{cmss}\selectfont \textbf{Results for \texttt{Heart} data}} \label{table:heart}
 \begin{tabular}{c||c|c|c||c|c|c||c|c|c}
  & \multicolumn{3}{c||}{LDA} & \multicolumn{3}{c||}{Neural network} & \multicolumn{3}{c}{$k$NN} \\ \hline
  & CV & BMM & error & CV & BMM & error & CV & BMM & error \\ \hline
 mean & 0.1651 & 0.1645 & 0.1527 & 0.1786 & 0.1781 & 0.1707 & 0.3541 & 0.354 & 0.349  \\ \hline
 std & 0.0139 & 0.0139 & 0.0401 & 0.0147 & 0.0147 & 0.0408 & 0.0236 & 0.0237 & 0.0598  \\ \hline
 RMSE & 0.0514 & \textbf{0.0512} & - & 0.0504 & \textbf{0.0502} & - & 0.0741 & \textbf{0.0739} & -  \\ \hline
 \end{tabular}
\end{center}
\end{table}

\subsubsection{Ensembling}

The Bayesian median of means can also be used to improve many statistical algorithms that rely on ensemble learning, such as bagging. In this case, a particular algorithm, say a regression tree, is trained on different subsets of the data to produce distinct final estimates, with the purpose of achieving some variance reduction. This section considers using the Bayesian median of means to boost these gains.

The advantages in using the Bayesian median of means are greater the more varying the estimates are. Thus, highly non-convex methods or high-variance data provide a setting in which aggregation via the Bayesian median of means can outperform classical bagging. Note, however, that the data here are not independent anymore, since there is a large overlap between training folds. Still, all the conditional guarantees from Section \ref{sec:theoretical_guarantees} hold.

To investigate how these methods fare in real datasets, consider the performance of using both the sample mean and the Bayesian median of means in conjunction with bagging. Six randomly chosen datasets from the UCI data repository (\cite{Dua:2019}) were picked to span different domains of interest:
\begin{enumerate}[label=(\roman*)]
\item \texttt{bike}: with features such as weather and number of registered users at the moment, the target is the number of bikes rented at a given time ($n=17379$ data points, $p=15$ features);
\item \texttt{demand}: for given a collection of business metrics from a large logistics company, such as orders of different types, predict the total daily orders ($n=60$, $p=11$);
\item \texttt{fires}: predict the burned area of forest fires in Portugal using meteorological and spatial data ($n=517$, $p=10$);
\item \texttt{GPUs}: predict the running time of for the multiplication of two $2048 \times 2048$ matrices using a GPU OpenCL SGEMM kernel with features regarding the computational environment and past runs ($n=241600$, $p=16$);
\item \texttt{news}: given features regarding an online news webpage such as number of words, links, date of publication and subject area, predict how popular the article will be ($n=39644$, $p=59$);
\item \texttt{superconductors}: given features about superconductors, predict the critical temperature ($n=21263$, $p=80$).
\end{enumerate}
Minimal modifications were made to the datasets, such as converting categorical predictors to dummy variables.

For each dataset, $15\%$ of the data were reserved for testing, and $20$ regression trees were grown, each fitted with different subsets of the training data and using all the features available. The $20$ predictions from each tree were aggregated using the sample mean and the Bayesian median of means. 

As measures of test error for each dataset, consider the proportional reduction in MSE and MAD when using $\BMM$:
\begin{align*}
\frac{\sum_{i=1}^{n_{\text{test}}} (\hat{y}_{\text{SM}, i} - y_i)^2 - (\hat{y}_{\text{BMM}, i} - y_i)^2}{\sum_{i=1}^{n_{\text{test}}} (\hat{y}_{\text{SM}, i} - y_i)^2}, \qquad  \frac{\sum_{i=1}^{n_{\text{test}}}|\hat{y}_{\text{SM}, i} - y_i)| - |\hat{y}_{\text{BMM}, i} - y_i|}{\sum_{i=1}^{n_{\text{test}}} |\hat{y}_{\text{SM}, i} - y_i|}
\end{align*}
where $y_i$ is the realized outcome in the test data for a given dataset, $\hat{y}_{\text{SM}, i}$ is the corresponding prediction using the sample mean for aggregation, and $\hat{y}_{\text{SM}, i}$ is the prediction using the Bayesian median of means. While these measures try to address how well this modified bagging performs, the exercise must be viewed with caution since the underlying data distribution is not known. 

Table \ref{table:bagging_datasets} shows an estimate for the proportional reduction in MSE and MAd, for each dataset; larger entries favor the Bayesian median of means. 

\begin{table}[h]
  \centering%
  \caption{\fontfamily{cmss}\selectfont \textbf{Reduction in MSE and MAD when using $\BMM$ instead of $\SM$}}\label{table:bagging_datasets}
 \vspace{2mm}
 \begin{tabular}{l|c|c|c|c|c|c}
         & \texttt{bike} & \texttt{demand} & \texttt{fires} & \texttt{GPUs} & \texttt{news} & \texttt{superconductors} \\ \hline
 MSE reduction  & 1.9\% & 17.3\% & 7.0\% & 0.6\% & 3.3\%  & -0.8\% \\ \hline
 MAD reduction & -0.3\% & 11.7\% & 7.0\% & 0.8\% & 4.4\% & 0.3\%
 \end{tabular}
\end{table}

In most cases, the Bayesian median of means provides an improvement over traditional bagging both in terms of MSE and MAD, although the improvements are modest. Whenever there is decrease in performance, it is generally small compared to the gains. The largest improvements happens on datasets with a small number of samples, so the 20 estimates to be aggregated are much less stable. From Section \ref{sec:asymptotic_approximation}, the smaller the variance or skewness of the estimates the more similar the procedure becomes to regular bagging. Overall, the gains in using the Bayesian median of means in datasets from a wide variety of areas displays its potential as an ensembling technique.

\section{Conclusion} \label{sec:conclusion}

This paper introduced the Bayesian median of means, a non-parametric aggregation procedure that leads to estimates with relatively small variance at the expense of some bias. Asymptotically, the added bias is negligible with respect to the reduction in variance, much like maximum likelihood estimates in the parametric setting. Furthermore, the Bayesian median of means is asymptotically unbiased, and essentially reduces to the sample mean, a widely used location estimator, when the variance or skewness of the underlying sample is small. Computationally, it is easy to implement and can be parallelized in a straightforward way.

The Bayesian median of means is, however, a randomized procedure, so a deterministic approximation was developed, and dubbed the approximate Bayesian median of means. It resembles a shrinkage-type estimator, and has similar performance to the Bayesian median of means while being much faster to run.

Both methods were empirically tested on a variety of datasets and simulations. Overall, the Bayesian median of means was shown to be competitive with the sample mean in settings of low variance, and performed significantly better as the variance increased. When tested in real datasets, it displayed small but consistent gains over the sample mean, and its performance on applications such as importance sampling, cross-validation and bagging showed it can be useful in adding robustness to data resesampling schemes.

There are several directions for future work. First, more refined concentration bounds are likely possible, and with it a better understanding of the theoretical properties of the algorithm. Second, the approximate Bayesian median of means calls for an a theoretical investigation of its own, including issues of minimaxity and admissibility. Third, more extensive empirical analyses can illuminate the extent to which the procedure can overcome the sample mean, or fail altogether.

Lastly, this paper intends to encourage research into designing better general-purpose aggregation procedures, and stimulate further work in an area that has far-reaching applications in statistics and beyond.

\bibliography{bibliography}{}

\begin{thebibliography}{}

\bibitem[Alon et~al., 1999]{alon1999space}
Alon, N., Matias, Y., and Szegedy, M. (1999).
\newblock The space complexity of approximating the frequency moments.
\newblock {\em Journal of Computer and system sciences}, 58(1):137--147.

\bibitem[Antille, 1974]{antille1974linearized}
Antille, A. (1974).
\newblock A linearized version of the {H}odges-{L}ehmann estimator.
\newblock {\em The Annals of Statistics}, pages 1308--1313.

\bibitem[Basu and DasGupta, 1997]{basu1997mean}
Basu, S. and DasGupta, A. (1997).
\newblock The mean, median, and mode of unimodal distributions: a
  characterization.
\newblock {\em Theory of Probability \& Its Applications}, 41(2):210--223.

\bibitem[Bickel and Lehmann, 1981]{bickel1981minimax}
Bickel, P. and Lehmann, E. (1981).
\newblock A minimax property of the sample mean in finite populations.
\newblock {\em The Annals of Statistics}, pages 1119--1122.

\bibitem[Brownlees et~al., 2015]{brownlees2015empirical}
Brownlees, C., Joly, E., Lugosi, G., et~al. (2015).
\newblock Empirical risk minimization for heavy-tailed losses.
\newblock {\em The Annals of Statistics}, 43(6):2507--2536.

\bibitem[Bubeck et~al., 2013]{bubeck2013bandits}
Bubeck, S., Cesa-Bianchi, N., and Lugosi, G. (2013).
\newblock Bandits with heavy tail.
\newblock {\em IEEE Transactions on Information Theory}, 59(11):7711--7717.

\bibitem[B{\"u}hlmann, 2003]{buhlmann2003bagging}
B{\"u}hlmann, P.~L. (2003).
\newblock Bagging, subagging and bragging for improving some prediction
  algorithms.
\newblock In {\em Research report/Seminar f{\"u}r Statistik, Eidgen{\"o}ssische
  Technische Hochschule (ETH)}, volume 113. Seminar f{\"u}r Statistik,
  Eidgen{\"o}ssische Technische Hochschule (ETH), Z{\"u}rich.

\bibitem[Catoni, 2012]{catoni2012challenging}
Catoni, O. (2012).
\newblock Challenging the empirical mean and empirical variance: a deviation
  study.
\newblock {\em Annales de l'IHP Probabilit{\'e}s et statistiques},
  48(4):1148--1185.

\bibitem[Chan and He, 1994]{chan1994simple}
Chan, Y. and He, X. (1994).
\newblock A simple and competitive estimator of location.
\newblock {\em Statistics \& Probability Letters}, 19(2):137--142.

\bibitem[Cifarelli and Melilli, 2000]{cifarelli2000some}
Cifarelli, D.~M. and Melilli, E. (2000).
\newblock Some new results for {D}irichlet priors.
\newblock {\em Annals of statistics}, pages 1390--1413.

\bibitem[Cifarelli and Regazzini, 1990]{cifarelli1990distribution}
Cifarelli, D.~M. and Regazzini, E. (1990).
\newblock Distribution functions of means of a {D}irichlet process.
\newblock {\em The Annals of Statistics}, pages 429--442.

\bibitem[Cifarelli and Regazzini, 1993]{cifarelli1993}
Cifarelli, D.~M. and Regazzini, E. (1993).
\newblock Some remarks on the distribution function of means of a {D}irichlet
  process.
\newblock Technical Report~4, IMATI-CNR, Milan.

\bibitem[Cifarelli and Regazzini, 1994]{Cifarelli_1994}
Cifarelli, D.~M. and Regazzini, E. (1994).
\newblock Correction: Distribution functions of means of a {D}irichlet process.
\newblock {\em The Annals of Statistics}, 22(3):1633--1634.

\bibitem[Damilano and Puig, 2004]{damilano2004efficiency}
Damilano, G. and Puig, P. (2004).
\newblock Efficiency of a linear combination of the median and the sample mean:
  the double truncated normal distribution.
\newblock {\em Scandinavian journal of statistics}, 31(4):629--637.

\bibitem[Devroye et~al., 2016]{devroye2016sub}
Devroye, L., Lerasle, M., Lugosi, G., Oliveira, R.~I., et~al. (2016).
\newblock Sub-gaussian mean estimators.
\newblock {\em The Annals of Statistics}, 44(6):2695--2725.

\bibitem[Diaconis et~al., 2001]{diacoms2001statistical}
Diaconis, P., Graham, R., and Holmes, S. (2001).
\newblock Statistical problems involving permutations with restricted
  positions.
\newblock {\em State of the Art in Probability and Statistics: Festschrift for
  Willem R. Van Zwet}, 36:195.

\bibitem[Diaconis and Kemperman, 1996]{diaconis1996some}
Diaconis, P. and Kemperman, J. (1996).
\newblock Some new tools for {D}irichlet priors.
\newblock {\em Bayesian statistics}, 5:97--106.

\bibitem[Dua and Graff, 2017]{Dua:2019}
Dua, D. and Graff, C. (2017).
\newblock {UCI} machine learning repository.

\bibitem[Efron et~al., 1996]{efron1996using}
Efron, B., Tibshirani, R., et~al. (1996).
\newblock Using specially designed exponential families for density estimation.
\newblock {\em The Annals of Statistics}, 24(6):2431--2461.

\bibitem[Friedman and Hall, 2007]{friedman2007bagging}
Friedman, J.~H. and Hall, P. (2007).
\newblock On bagging and nonlinear estimation.
\newblock {\em Journal of statistical planning and inference}, 137(3):669--683.

\bibitem[Hodges et~al., 1956]{hodges1956efficiency}
Hodges, J.~L., Lehmann, E.~L., et~al. (1956).
\newblock The efficiency of some nonparametric competitors of the $ t $-test.
\newblock {\em The Annals of Mathematical Statistics}, 27(2):324--335.

\bibitem[Hodges~Jr and Lehmann, 1963]{hodges1963estimates}
Hodges~Jr, J.~L. and Lehmann, E.~L. (1963).
\newblock Estimates of location based on rank tests.
\newblock {\em The Annals of Mathematical Statistics}, pages 598--611.

\bibitem[Hsu and Sabato, 2016]{hsu2016loss}
Hsu, D. and Sabato, S. (2016).
\newblock Loss minimization and parameter estimation with heavy tails.
\newblock {\em The Journal of Machine Learning Research}, 17(1):543--582.

\bibitem[James et~al., 2013]{james2013introduction}
James, G., Witten, D., Hastie, T., and Tibshirani, R. (2013).
\newblock {\em An introduction to statistical learning}, volume 112.
\newblock Springer.

\bibitem[Jerrum et~al., 1986]{jerrum1986random}
Jerrum, M.~R., Valiant, L.~G., and Vazirani, V.~V. (1986).
\newblock Random generation of combinatorial structures from a uniform
  distribution.
\newblock {\em Theoretical Computer Science}, 43:169--188.

\bibitem[Joly and Lugosi, 2016]{joly2016robust}
Joly, E. and Lugosi, G. (2016).
\newblock Robust estimation of {U}-statistics.
\newblock {\em Stochastic Processes and their Applications},
  126(12):3760--3773.

\bibitem[Joshi, 1968]{10.2307/2239052}
Joshi, V.~M. (1968).
\newblock Admissibility of the sample mean as estimate of the mean of a finite
  population.
\newblock {\em The Annals of Mathematical Statistics}, 39(2):606--620.

\bibitem[Kerman, 2011]{kerman2011closed}
Kerman, J. (2011).
\newblock A closed-form approximation for the median of the beta distribution.
\newblock {\em arXiv preprint arXiv:1111.0433}.

\bibitem[Lai et~al., 1983]{lai1983adaptive}
Lai, T., Robbins, H., and Yu, K. (1983).
\newblock Adaptive choice of mean or median in estimating the center of a
  symmetric distribution.
\newblock {\em Proceedings of the National Academy of Sciences},
  80(18):5803--5806.

\bibitem[Minsker et~al., 2014]{pmlr-v32-minsker14}
Minsker, S., Srivastava, S., Lin, L., and Dunson, D. (2014).
\newblock Scalable and robust {B}ayesian inference via the median posterior.
\newblock {\em Proceedings of the 31st International Conference on Machine
  Learning}, 32:1656--1664.

\bibitem[Nemirovsky and Yudin, 1983]{nemirovsky1983problem}
Nemirovsky, A.~S. and Yudin, D.~B. (1983).
\newblock {\em Problem complexity and method efficiency in optimization.}
\newblock Wiley.

\bibitem[Newton and Raftery, 1994]{newton1994approximate}
Newton, M.~A. and Raftery, A.~E. (1994).
\newblock Approximate {B}ayesian inference with the weighted likelihood
  bootstrap.
\newblock {\em Journal of the Royal Statistical Society: Series B
  (Methodological)}, 56(1):3--26.

\bibitem[Pitman, 2018]{pitman2018random}
Pitman, J. (2018).
\newblock Random weighted averages, partition structures and generalized
  arcsine laws.
\newblock {\em arXiv preprint arXiv:1804.07896}.

\bibitem[Purkayastha, 1998]{purkayastha1998simple}
Purkayastha, S. (1998).
\newblock Simple proofs of two results on convolutions of unimodal
  distributions.
\newblock {\em Statistics \& probability letters}, 39(2):97--100.

\bibitem[Regazzini et~al., 2000]{regazzini2000}
Regazzini, E., Guglielmi, A., and Di~Nunno, G. (2000).
\newblock Theory and numerical analysis for exact distributions of functionals
  of a {D}irichlet process.
\newblock Technical Report 00.12, CNR-IAMI.

\bibitem[Regazzini et~al., 2002]{regazzini2002theory}
Regazzini, E., Guglielmi, A., Di~Nunno, G., et~al. (2002).
\newblock Theory and numerical analysis for exact distributions of functionals
  of a {D}irichlet process.
\newblock {\em The Annals of Statistics}, 30(5):1376--1411.

\bibitem[Von~Neumann, 1941]{von1941distribution}
Von~Neumann, J. (1941).
\newblock Distribution of the ratio of the mean square successive difference to
  the variance.
\newblock {\em The Annals of Mathematical Statistics}, 12(4):367--395.

\bibitem[Watson, 1956]{watson1956joint}
Watson, G.~S. (1956).
\newblock On the joint distribution of the circular serial correlation
  coefficients.
\newblock {\em Biometrika}, 43(1/2):161--168.

\bibitem[Weng, 1989]{weng1989second}
Weng, C.-S. (1989).
\newblock On a second-order asymptotic property of the {B}ayesian bootstrap
  mean.
\newblock {\em The Annals of Statistics}, pages 705--710.

\end{thebibliography}
\nocite{*}
\bibliographystyle{apalike}

\end{document}